\documentclass[a4paper,11pt,leqno]{amsart}
\usepackage{amsmath}
\usepackage{amsthm}
\usepackage{amsfonts}
\usepackage{amssymb}

\usepackage{texdraw}
\usepackage[mathscr]{eucal}
\usepackage{eucal}

\usepackage{graphicx}
\usepackage{psfrag}

\usepackage[all]{xy}

\newtheorem{theorem}{Theorem}[section]{\bf}{\it}
\newtheorem{lemma}[theorem]{Lemma}{\bf}{\it}
\newtheorem{proposition}[theorem]{Proposition}{\bf}{\it}
\newtheorem{corollary}[theorem]{Corollary}{\bf}{\it}
{\bf}{\it} 
{\bf}{\it}

\newtheorem{remark}[theorem]{Remark}{\bf}{\it}
{\bf}{\it}

{\bf}{\it}
{\bf}{\it}

\theoremstyle{definition}
\newtheorem{definition}[theorem]{Definition}{\bf}{\it}

\numberwithin{equation}{section}

\newcommand{\diam}{\operatorname{diam}}

\newcommand{\R}{\mathbb R}

\newcommand{\Z}{\mathbb Z}

\newcommand{\inter}{{\operatorname{int}\,}}
\newcommand{\dist}{{\operatorname{dist}}}
\newcommand{\id}{{\operatorname{id}}}

\newdimen\vintkern\vintkern11pt
\def\vint{-\kern-\vintkern\int}

\newcommand{\bZ}{\mathbb{Z}}
\newcommand{\bS}{\mathbb{S}}

\newcommand{\haus}{\mathcal{H}}
\newcommand{\dhaus}{\;\mathrm{d}\haus}
\newcommand{\dmu}{\;\mathrm{d}\mu}

\newcommand{\cC}{\mathcal{C}}
\newcommand{\cW}{\mathcal{W}}
\newcommand{\cS}{\mathcal{S}}

\newcommand{\cE}{\mathcal{E}}
\newcommand{\cX}{\mathcal{X}}

\newcommand{\cJ}{\mathcal{J}}

\newcommand{\fC}{\mathscr{C}} 
\newcommand{\fW}{\mathscr{W}} 
\newcommand{\fA}{\mathscr{A}} 
\newcommand{\meri}{\mathscr{M}} 
\newcommand{\cxi}{\mathtt{c}} 
\newcommand{\bB}{\mathbb{B}}

\newcommand{\Mod}{\mathrm{Mod}}
\newcommand{\interior}{\mathrm{int}}
\newcommand{\wind}{\mathrm{circ}}
\newcommand{\longi}{\Sigma}
\newcommand{\growth}{\gamma}

\newcommand{\genus}{\mathrm{g}}
\newcommand{\tree}{\mathrm{Tree}}

\newcommand{\sing}{\mathrm{sing}}

\newcommand{\diff}{\mathrm{diff}}

\newcommand{\lvl}{\mathrm{level}}
\newcommand{\depth}{\rho}

\newcommand{\Bd}{\mathrm{Bd}}
\newcommand{\Db}{\mathrm{Db}}
\newcommand{\Wh}{\mathrm{Wh}}

\begin{document}

\title[Semmes spaces]{Geometry and quasisymmetric parametrization of Semmes spaces}
\author{Pekka Pankka \and Jang-Mei Wu}
\begin{abstract}
We consider decomposition spaces $\R^3/G$ that are manifold factors and admit defining sequences consisting of cubes-with-handles of finite type.  Metrics on $\R^3/G$ constructed via modular embeddings of $\R^3/G$ into a Euclidean space promote the controlled topology to a controlled geometry.

The quasisymmetric parametrizability of the metric space $\R^3/G\times \R^m$ by $\R^{3+m}$ for any $m\ge 0$ imposes quantitative topological constraints, in terms of the circulation and the growth of the cubes-with-handles, on the defining sequences for $\R^3/G$. We give a necessary condition and a sufficient condition for the existence of such a parametrization.

The necessary condition answers negatively a question of Heinonen and Semmes on quasisymmetric parametrizability of spaces associated to the Bing double. The sufficient condition gives new examples of quasispheres in $\bS^4$.
\end{abstract}

\thanks{This work was supported in part by the Academy of Finland projects 126836 and 256228 and the National Science Foundation grants DMS-0757732 and DMS-1001669.}
\date{\today}

\maketitle

\tableofcontents

\section{Introduction}

\subsection{}A homeomorphism $f\colon X\to Y$ between metric spaces $(X,d_X)$ and $(Y,d_Y)$ is called \emph{quasisymmetric} if there exists a homeomorphism $\eta\colon [0,\infty)\to [0,\infty)$ so that
\[
\frac{d_Y(f(x),f(y))}{d_Y(f(x),f(z))} \le \eta\left( \frac{d_X(x,y)}{d_X(x,z)}\right)
\]
for all triples $\{x,y,z\}$ in $X$. Quasisymmetry generalizes quasiconformality from Euclidean spaces to general metric spaces.
A metric space $(X,d)$ is called a \emph{metric $n$-sphere} if it is homeomorphic to $\bS^n$.

\emph{When is a metric $n$-sphere $(X,d)$ quasisymmetrically equivalent to the standard $\bS^n$?} The goal is to find intrinsic qualitative metric properties of the space $(X,d)$ that recognize such geometric equivalence. A complete characterization of quasispheres is known only for dimensions $1$ and $2$.

In dimension $1$, a result of Tukia and V\"ais\"al\"a \cite{TukiaP:Quaems} states \emph{a metric $1$-sphere $(X,d)$ is quasisymmetrically equivalent to $\bS^1$  if and only if $X$ is doubling and is of bounded turning.} Bonk and Kleiner \cite[Theorem 1.1]{BonkM:Quaptd} give a characterization in dimension $2$. A consequence of their theorems states that \emph{a metric $2$-sphere $(X,d)$ is quasisymmetrically equivalent to $\bS^2$ if $X$ is linearly locally contractible and Ahlfors $2$-regular}. Semmes proved this result earlier for metric spaces with some added smoothness properties \cite[Section 5]{SemmesS:Choasw}. Wildrick proved recently an analogue of Bonk and Kleiner's result for $\R^2$ \cite{WildrickK:Quaptd}.

A metric space $(X,d)$  is said to be \emph{linearly locally contractible} if for a fixed $C>1$ every ball of radius $r<1/C$ is contractible in a concentric ball of radius $Cr$; and $X$ is said to be \emph{Ahlfors $2$-regular} if there exists a measure $\mu$ on the space so that the $\mu$-measure of every ball of radius $r$ is uniformly comparable to $r^2$.

Could a metric space which is homeomorphic to $\bS^n$, or $\R^n$, and  resembles ${\bS}^n$, or $\R^n$,  geometrically (linearly locally contractible), measure-theoretically (Ahlfors $n$-regular), and analytically (supports Poincar\'e and Sobolev inequalities) in dimensions $n\ge 3$, fail to be quasisymmetrically equivalent to ${\bS}^n$, or $\R^n$?

Semmes' counterexample \cite{SemmesS:Goomsw} to this natural question in dimension $3$ is a geometrically self-similar space modeled on the decomposition space $\R^3/\Bd$ associated to the Bing double $\Bd$. The classical construction of R.H. Bing in geometric topology gives an example of a wild involution in $\R^3$. As a topological space $\R^3/\Bd$ is homeomorphic to $\R^3$.

Semmes shows that this space admits a metric that is smooth Riemannian outside a totally disconnected closed set and, in many ways, indistinguishable from the standard metric on $\R^3$, and yet the space is not quasisymmetrically equivalent to $\R^3$. In Semmes' metric the $2^k$ tori at $k$-th stage of the construction of $\R^3/\Bd$ are uniformly round and thick, whereas under any homeomorphism from $\R^3/\Bd$ to $\R^3$, there exists a sequence of tori that are distorted into a shape longer and thinner than allowed by any fixed quasisymmetry.
Semmes' construction is robust and essentially available in all decomposition spaces of $\R^3$ arising from topologically self-similar constructions.

The natural  conditions for metric $n$-spheres listed earlier are also insufficient in higher dimensions. The decomposition space $\R^3/\Wh$ associated to the Whitehead continuum $\Wh$ is not homeomorphic to $\R^3$, but $\R^3/\Wh\times \R$ is homeomorphic to $\R^4$.
In \cite{HeinonenJ:Quansa} Heinonen and the second author showed that \emph{the decomposition space $\R^3/\Wh$ associated to the Whitehead continuum $\Wh$ admits a linearly locally contractible and Ahlfors $3$-regular metric, but $(\R^3/\Wh) \times \R^m $ is not quasisymmetrically equivalent to $\R^{3+m}$ for any $m \ge 1$}. The metric on $\R^3/\Wh$ is due to Semmes; as in the case of $\R^3/\Bd$ this metric makes the tori in the construction of the Whitehead continuum uniformly round and thick.

\begin{figure}[h!]
\includegraphics[scale=0.60]{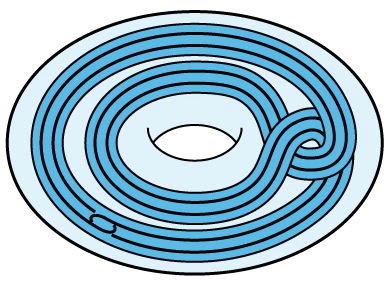}
\caption{Two generations of Whitehead links.}
\label{fig:Whitehead}
\end{figure}

The Whitehead link, formed by a meridian of the first torus and the core of the second torus, however prevents the \emph{conformal modulus} of a sequence of surface families over longitudes of the nested tori from being quasi-preserved under any homeomorphism $\R^3/\Wh\times \R^m \to \R^{3+m}$.

\subsection{}The decomposition space $\R^3/\Wh$ is only one example of an exotic manifold factor of $\R^4$. By a theorem of Edwards and Miller \cite{EdwardsR:Celc0d}, decomposition spaces that are exotic factors of $\R^4$ exist in abundance. In fact, cell-like closed $0$-dimensional upper semicontinuous decomposition spaces $\R^3/G$ are manifold factors of $\R^4$, that is, $\R^3/G\times \R$ is homeomorphic to $\R^4$. Furthermore, under mild assumptions on the decomposition, these spaces are definable by nested sequences $\cX=(X_k)_{k\ge 0}$ of unions of cubes-with-handles, i.e., the degenerate part of the decomposition $G$ is $\bigcap_{k\ge 0} X_k$; see Lambert and Sher \cite{LambertH:Poil0d} and Sher and Alford \cite{SherR:Not0dd}. This class of decomposition spaces provides a natural environment for testing the quasisymmetric parametrization.

In this article, we consider a subclass of decomposition spaces $\R^3/G$ that are manifold factors and admit defining sequences of finite type. The corresponding defining sequences $\cX=(X_k)_{k\ge 0}$ satisfy the \emph{a priori} condition
\[
\# \{ [H\setminus \interior X_{k+1}]_{\mathrm{PL}}\colon k\ge 0\ \mathrm{and\ } H\ \mathrm{is\ a\ component\ of\ } X_k\} < \infty;
\]
here $[E]_{\mathrm{PL}}$ the PL-homeomorphism equivalence class of a PL-manifold $E\subset \R^3$.

The definition of finite type is based on the notion of welding.
A \emph{welding structure} $(\fC,\fA,\fW)$ consists of \emph{condensers} $\fC$, an \emph{atlas} $\fA$ composed of \emph{charts}, and \emph{weldings} $\fW$ determined by the atlas $\fA$. The condensers can be seen as fixed geometric realizations of PL-homeomorphism equivalence classes of components of differences $X_k\setminus \interior X_{k+1}$ in the defining sequence $\cX$, and the charts in the atlas $\fA$ determine the parametrization of these components. The weldings, in turn, are transition maps between the charts; see Section \ref{sec:ws}.

\begin{definition}
A defining sequence $\cX=(X_k)_{k\ge 0}$ has \emph{finite type} if there exists a welding structure $(\fC,\fA,\fW)$ with finitely many condensers $\fC$ and finitely many weldings $\fW$. A decomposition space $(\R^3/G,\cX)$ is of \emph{finite type} if the defining sequence $\cX$ has finite type.
\end{definition}

\begin{figure}[h!]
\includegraphics[scale=0.50]{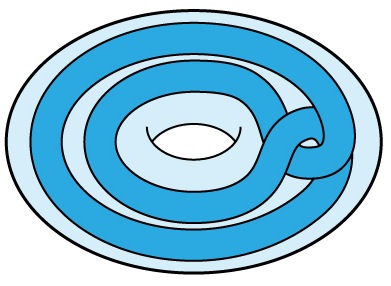}
\quad \,
\includegraphics[scale=0.50]{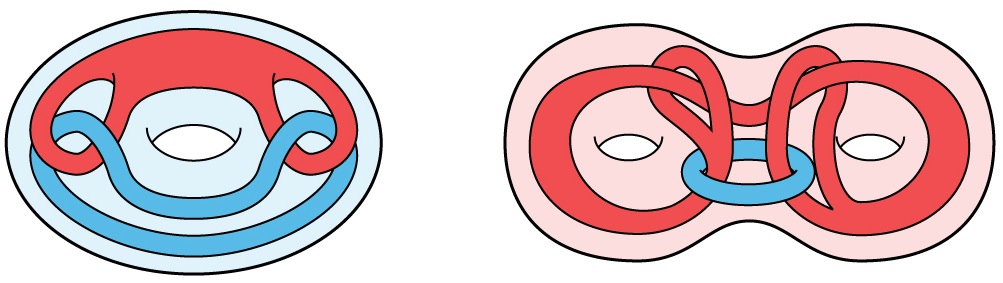}
\caption{These condensers generating an infinite number of defining sequences of finite type.}
\label{fig:three_condensers}
\end{figure}

We take up a systematic study of the geometrical realizations which promote the controlled topology to a controlled geometry. Using results from classical geometric topology, we construct for every defining sequence of finite type a geometrically simple welding structure, called a \emph{rigid welding structure}, having translations as weldings; see Theorem \ref{thm:existence_of_Semmes}.

A rigid welding structure allows the natural geometrization of the decomposition space $\R^3/G$. Given a rigid welding structure $(\fC,\fA,\fW)$ and a \emph{scaling factor $\lambda\in (0,1)$}, we show that there exists a \emph{modular embedding} of $\R^3/G$ into a Euclidean space that respects the atlas $\fA$ and the chosen scaling factor $\lambda$; see Theorem \ref{thm:FT_embedding}. The metric $d_\lambda$ induced on $\R^3/G$ by a modular embedding is called a \emph{Semmes metric} and the corresponding metric space \emph{Semmes space}; these metrics naturally extend the class of metrics constructed by Semmes in \cite{SemmesS:Goomsw}.

For a fixed rigid welding structure, the Semmes spaces $(\R^3/G,d_\lambda)$ for all scalings are mutually quasisymmetric. We find it appealing that, although $\R^3/G$ does not admit a canonical  metric, \emph{there exists a natural class of metrics on $\R^3/G$ respecting the defining sequence $\cX$ whose quasisymmetry equivalence classes are parametrized by rigid welding structures on $\cX$ modulo compatible atlases}; see Proposition \ref{prop:R3G_qs_type}.

We summarize the essential features of this geometrization process of decomposition spaces by Semmes metrics in the following theorem; see Sections \ref{sec:me} and \ref{sec:Semmes} for these results. Given a Semmes metric $d_\lambda$ on $\R^3/G$, we equip the space $\R^3/G\times\R^m$ with the product metric $d_{\lambda,m}((x,v),(y,w)) = d_\lambda(x,y)+|v-w|$.

\begin{theorem}
\label{thm:summary}
Let $(\R^3/G,\cX)$ be a decomposition space of finite type. Then there exists a Semmes metric $d_\lambda$ on $\R^3/G$ so that, for each $m\ge 0$, $(\R^3/G\times \R^m,d_{\lambda,m})$ is a quasiconvex Ahlfors $(3+m)$-regular Loewner space that admits an isometric embedding into a Euclidean space. Moreover, the space $(\R^3/G\times \R^m,d_{\lambda,m})$ is linearly locally contractible if the sequence $\cX$ is locally contractible.
\end{theorem}

A defining sequence $\cX=(X_k)_{k\ge 0}$ is \emph{locally contractible} if every component of $X_{k+1}$ is contractible in $X_k$ for all $k\ge 0$. We emphasize that the linear local contractibility and the Loewner property are necessary for the existence of quasisymmetric parametrization; see Semmes \cite{SemmesS:Goomsw}, Heinonen and Koskela \cite{HeinonenJ:Quamms}, and Tyson \cite{TysonJ:Quaqmm}.

\subsection{}Having this general theory at our disposal, we now discuss the problem of quasisymmetric parametrization.

Due to the quasi-invariance of the conformal modulus, the existence of a quasisymmetric homeomorphism between $\R^3/G\times \R^m$ and $\R^{3+m}$ imposes a constraint between geometry (growth of the handlebodies and the fixed scaling factor) and topology (circulation of the handlebodies).

The order of growth of $\cX$ controls the growth of the number of components of $X_k$ in the sequence as $k$ tends to infinity; see Definition \ref{def:growth}.
The order of circulation of $\cX$ reflects the growth of the unsigned linking numbers of the longitudes of handlebodies of $X_{k'}$ with respect to the meridians of $X_{k}$, for $k'>k$; see Definition \ref{def:circulation}.

\begin{theorem}
\label{thm:G_first}
Let $\R^3/G$ be a decomposition space of finite type associated to a locally contractible defining sequence $\cX$. Suppose that the order of growth of the defining sequence $\cX$ is at most $\gamma$, the order of circulation is at least $\omega$, where $\gamma,\omega\in [0,\infty]$, and
\begin{equation}
\label{eq:omega_lambda}
\omega^3 > \gamma^2.
\end{equation}
Then there exists a Semmes metric on $\R^3/G$ so that $\R^3/G\times \R^m$ is a linearly locally contractible, Ahlfors $(3+m)$-regular, Loewner space but not quasisymmetrically equivalent to $\R^{3+m}$ for any $m\ge 0$.
\end{theorem}

For the Whitehead continuum and the Bing double we may take the pair $(\gamma,\omega)$ to be $(1,2)$ and $(2,2)$, respectively. In particular, Theorem \ref{thm:G_first} provides a negative answer to a question of Heinonen and Semmes in \cite[Question 11]{HeinonenJ:33ynq}. The case $m=0$ in the following theorem is Semmes' quasisymmetric non-parametrizability result of $\R^3/\Bd$ in \cite{SemmesS:Goomsw}.

\begin{theorem}\label{thm:BD_first}
The decomposition space $\R^3/\Bd$ associated to the Bing double admits a metric that is Ahlfors $3$-regular, Loewner, and linearly locally contractible but none of the spaces $\R^3/\Bd \times \R^m$ for $m\ge 0$ is quasisymmetrically equivalent to $\R^{3+m}$.
\end{theorem}

Theorem \ref{thm:G_first} admits a local formulation as stated in Theorem \ref{thm:General_local}. This local version examines inequality \eqref{eq:omega_lambda} on a branch of the defining sequence at a point in $\R^3/G$; it is generally more applicable.
Whereas inequality \eqref{eq:omega_lambda} gives a natural necessary condition for quasisymmetric parametrization of $(\R^3/G\times \R^m,d_{\lambda,m})$ by $\R^{3+m}$, the pointwise inequality \eqref{eq:omega_lambda_local} provides a criterion for the quasisymmetric equivalence of product spaces $(\R^3/G\times \R, d_{\lambda,1})$ for $0<\lambda<1$. This yields, for example, the following inequivalence result for the product spaces associated to the Bing double.

\begin{theorem}
\label{thm:Bing_stab}
Let $(\R^3/\Bd, d_\lambda)$ be a Semmes space associated to the Bing double.
For any $\lambda' \in (1/2,1)$ and  $\lambda \in (0,\lambda')$,
spaces $(\R^3/\Bd \times \R, d_{\lambda,1})$ and $(\R^3/\Bd\times \R, d_{\lambda',1})$ are quasisymmetrically inequivalent.
\end{theorem}

When applying Theorem \ref{thm:G_first} and Theorem \ref{thm:General_local}, estimating the order of circulation from below for a particular decomposition space can be a challenging topological problem of its own. For decomposition spaces associated to the Bing double \cite{SemmesS:Goomsw}, to the Whitehead continuum \cite{HeinonenJ:Quansa}, or to Bing's dogbone (Section \ref{sec:Bing_Dogbone}), the circulation is estimated by adapting a theorem of Freedman and Skora \cite{FreedmanM:Strags} on counting essential intersections by relative homologies; see Section \ref{sec:Bing_Double}.

\subsection{}
Theorem \ref{thm:G_first} imposes a topological condition for the parametrization. In the opposite direction, additional Euclidean restrictions on the welding structure yield positive results about the quasisymmetric parametrization of $\R^3/G$ by $\R^3$. These geometric assumptions on the defining sequence are encapsulated in the notion of \emph{flat welding structure}; see Section \ref{sec:positive}.

\begin{theorem}
\label{thm:positive_first}
Let $\R^3/G$ be a decomposition space of finite type that admits a defining sequence with a flat welding structure in $\R^3$. Then there exists a linearly locally contractible, Ahlfors $3$-regular metric on $\R^3/G$ so that $\R^3/G$ is quasisymmetric to $\R^3$. Moreover, there exist an isometric embedding $\theta \colon \R^3/G \to \R^4$ and a quasisymmetric homeomorphism $f\colon \R^4\to \R^4$ so that $f(\R^3) = \theta(\R^3/G)$.
\end{theorem}

These decomposition spaces give new examples of \emph{quasispheres} in $\R^4$ as formulated in the second part of the theorem.

In light of Theorem \ref{thm:positive_first}, we ask about the sharpness of the condition \eqref{eq:omega_lambda} in Theorem \ref{thm:G_first}, especially for a fixed $m$. In case of $\R^3$ (i.e.\;$m=0$) the construction of Antoine's necklaces $G$ using $I$ linked tori yields decomposition spaces with the order of growth $I$ and the order of circulation at least $2$. Semmes' result on the Bing double implies that the decomposition space $(\R^3/G,d)$ associated with a necklace constructed using two linked tori, when equipped  with a Semmes metric $d$, is not quasisymmetric to $\R^3$.

The existence of a quasisymmetric parametrization of $\R^3/G$ when $I$ is sufficiently large has been observed by Heinonen and Rickman \cite{HeinonenJ:QuamS3} using round similar tori. Using rectangular tori in place of round tori, we prove in Theorem \ref{thm:necklace_theorem} that
\emph{for every $I\ge 10$,
the decomposition space $\R^3/G$ associated to Antoine's $I$-necklace may be equipped with a Semmes metric so that it is quasisymmetrically equivalent to $\R^3$}.

Having these examples at hand, the real test for the sharpness of Theorem \ref{thm:G_first} seems to be the quasisymmetric parametrizability of the decomposition space associated to the Antoine's $3$-necklace.

\bigskip
{\bf Acknowledgments.}
Authors are grateful to the anonymous
referee for his thoughtful comments, which led to the clarification of some mathematical concepts and the improvement of the presentation.

Authors would also like to thank Julie H.-G. Kaufman for the schematic drawings in this article.

The first author thanks Eero Saksman for discussions, and the Department of Mathematics at the University of Illinois at Urbana-Champaign for hospitality during his numerous visits.

\section{Preliminaries}
Unless otherwise stated, we assume that $\R^n$, $n\ge 1$, is equipped with the Euclidean metric with and the standard basis $(e_1,\ldots,e_n)$. We denote by $B^n(x,r)$ the closed Euclidean ball in $\R^n$ of radius $r$ and center $x$. For brevity, the closed balls about origin are denoted $B^n(r)=B^n(0,r)$ for $r>0$ and $\bB^n = B^n(1)$. Similarly, $S^{n-1}(x,r)=\partial B^n(x,r)$ is the Euclidean sphere of radius $r$ and center $x$ in $\R^n$, and, $S^{n-1}(r)=S^{n-1}(0,r)$ for $r>0$ and $\bS^{n-1}=S^{n-1}(1)$.

For all $1\le m <n$, we identify $\R^m$ with the subspace $\R^m\times \{0\}$ in $\R^n$ where $\{0\}$ is the origin in $\R^{n-m}$, and identify a set $A\subset\R^m$ with the set $A \times \{0\}$ in $\R^m \times \R^{n-m}$. When $\R^n$ is expressed as $\R^m\times \R^p\times \R^q$ with $m,p,q>0,\,  m+p+q$, a subset of $\R^n$ in the form $A\times B\times C$  is understood to have the property that $A\subset\R^m, B\subset\R^p,$ and $C\subset\R^q$.

By a map, we always mean a continuous map. Given a map $F\colon X\times [0,1]\to Y$, we denote by $F_t \colon X\to Y$ the map $F_t(x)=F(x,t)$. We say that a homotopy $F\colon X\times [0,1]\to Y$ is an \emph{isotopy} if $F_t$ is a homeomorphism for all $t\in [0,1]$.

We call a map $\alpha \colon I\to X$ from an interval in $\R$ into a metric space $X$ as a \emph{path} and maps $\bS^1 \to X$ as \emph{loops}. If there is no confusion we do not make distinction between a map and its image. Images of paths and loops are also called as \emph{curves}. A loop $\bS^1 \to X$ is \emph{simple} if it is an embedding.

Given a set $E$ in a metric space $(X,d)$ and a number $a>0$, we call
\[
N_d(E,a) =\{x\in X\colon \dist_d(x,E)<a\}
\]
the $a$-neighborhood of $E$ in $X$. When $X=\R^n$ and $d$ is the Euclidean metric, we write $N^n(E,a)$ for $N_d(E,a)$. We denote by $\cC(E)$ the set of all connected components of $E$.

Given a metric space $(X,d)$ so that points in the space can be connected by rectifiable paths, we denote by $\hat d$ the \emph{path metric of $(X,d)$} defined by
\[
\hat d(x,y) = \inf_\gamma \ell_d(\gamma)
\]
for $x,y\in X$, where $\ell_d(\gamma)$ is the length of path $\gamma$ in metric $d$ and the infimum is taken over all paths $\gamma$ connecting $x$ and $y$ in $X$. A metric space $(X,d)$ is called \emph{quasiconvex} if $\id \colon (X,\hat d)\to (X,d)$ is bilipschitz.

A metric space $(X,d)$ is \emph{Ahlfors $Q$-regular} if there exist a Borel measure $\mu$ in $X$ and a constant $C\ge 1$ so that
\[
\frac{1}{C} r^Q \le \mu(B(x,r)) \le C r^Q
\]
for every ball $B(x,r)$ of radius $r\le \diam X$ about $x$ in $X$. Furthermore, the space $(X,d)$ is \emph{locally linearly contractible} if there exists $C\ge 1$ so that the ball $B(x,r)$ in $X$ is contractible in $B(x,Cr)$ for all $r<1/C$.

We say that a mapping $f\colon (X,d_X) \to (Y,d_Y)$ between metric spaces is a \emph{$(\lambda,L)$-quasisimilarity} if
\[
\frac{\lambda}{L}d_X(x,y) \le d_Y(f(x),f(y)) \le \lambda L d_X(x,y)
\]
for all $x,y \in X$. Clearly, quasisimilarities are a subclass of quasisymmetries. As usual, we call $(\lambda,1)$-quasisimilarities as \emph{similarities} and  $1$-similarities as \emph{isometries}. The $(1,L)$-quasisimilarities are \emph{$L$-bilipschitz mappings}. In what follows, we abuse the notation and denote $|x-y|=d(x,y)$ when there is no ambiguity on the metric in question.

In a metric measure space $(X,d,\mu)$ we define the \emph{$p$-modulus of an $m$-chain family} as follows. In what follows, we consider only Lipschitz chains of multiplicity one, that is, we consider only $m$-chains $\sigma$ so that $\sigma = \sum_{i=1}^k \sigma_i$, where $\sigma_i \colon [0,1]^m \to X$ is a Lipschitz map for $i=1,\ldots, k$.

Given a family $\Sigma$ of $m$-chains in a $X$, the $p$-modulus of $\Sigma$ is
\begin{equation}
\label{eq:modulus_def}
\Mod_p(\Sigma) = \inf_\rho \int_X \rho^p \dmu,
\end{equation}
where $\rho$ is a non-negative Borel function satisfying
\begin{equation}
\label{eq:admissibility}
\sum_{i=1}^k \int_{\sigma_i([0,1]^m)} \rho \dhaus^m \ge 1
\end{equation}
for all $\sigma = \sum_{i=1}^k \sigma_i \in \Sigma$.

In what follows, handlebodies are three dimensional piecewise linear cubes-with-handles embedded into $\R^n$. For this we assume in what follows that $\R^n$ is given a fixed PL-structure for every $n\ge 3$.

We use the following topological facts on cubes-with-handles; see \cite[Chapter 2]{HempelJ:3man} for more details. We say that $H$ is a cube-with-handles if it is a regular neighborhood of an embedded \emph{rose} $\iota(\bigvee^g \bS^1)$, where $\iota \colon \bigvee^g \bS^1 \to \R^3$ is a PL-embedding. Here $\bigvee^g \bS^1$ is the \emph{wedge of $g$ circles}, that is, identification of $g$ circles at a point; $\bigvee^0 \bS^1$ is a point. The number $g$ of circles in the rose is called the \emph{number of handles of $H$} or the \emph{genus of $H$}. The image $\iota(\bigvee^g \bS^1)$ is called a \emph{core of $H$}.

The genus of $H$ is also the maximal number of essentially embedded $2$-disks whose union does separate $H$. We say that a disk $D$ in $H$ is \emph{essentially embedded} if there exists an embedding $\varphi \colon (\bB^2,\partial \bB^2) \to (H,\partial H)$ so that $\varphi|\partial  \bB^2\colon \partial \bB^2 \to \partial H$ is not null-homotopic in $\partial H$.

The genus of $H$ is a topological invariant. Two cube-with-handles $H$ and $H'$ in $\R^3$ are PL homeomorphic if and only if they have the same number of handles and both are either orientable or non-orientable (\cite[Theorem 2.2]{HempelJ:3man}). We denote by $\genus(H)$ the \emph{genus of $H$}.

Three dimensional cube-with-handles in $\R^n$ need not be orientable for $n>3$, but three dimensional cubes-with-handles in $\R^3$ inherit orientation from $\R^3$ and are therefore orientable.

\section{Decomposition spaces}
\label{sec:dsds}

We begin this section by reviewing some classical results on decomposition spaces relevant to our study. We do not aim at the full generality and refer to Daverman \cite{DavermanR:Decm} for details.

A \emph{decomposition} $G$ of a topological space $X$ is a partition of $X$. Associated with $G$ is the decomposition space $X/G$ equipped with the topology induced by the quotient map $\pi_G \colon X \to X/G$, the richest topology for which $\pi_G$ is continuous, see \cite[p. 8]{DavermanR:Decm}.

A decomposition $G$ is \emph{upper semi-continuous} (usc) if each $g\in G$ is closed and if for every $g\in G$ and every neighborhood $U$ of $g$ in $X$ there exists a neighborhood $V$ of $g$ contained in $U$ so that every $g'\in G$ intersecting $V$ is contained in $U$. If $G$ is upper semi-continuous then $X/G$ is metrizable \cite[Definition I.2  and Proposition I.2.2]{DavermanR:Decm}; however there is not a canonical metric on $X/G$.

Suppose that $G$ is a usc decomposition of an $n$-manifold  $M$ and $d$ is a metric on $M/G$. The decomposition map $\pi_G \colon M\to M/G$ can be approximated by homeomorphisms if and only if $G$ satisfies Bing's shrinkability criterion \cite[Theorem II.5.2]{DavermanR:Decm}. In particular, $M/G$ is homeomorphic to $M$.

Bing's \emph{shrinkability criterion} states that for every $\varepsilon>0$ there is a homeomorphism $h\colon M\to M$ such that
\begin{enumerate}
\item $\diam h(g) <\varepsilon$ for each $g\in G$, and
\item $d(\pi_G h(x),\pi_G(x))<\varepsilon$ for every $x\in M$.
\end{enumerate}

Suppose $M$ is an $n$-manifold. If $G$ is a shrinkable usc decomposition then each $g\in G$ is cellular, therefore cell-like  \cite[Proposition II.6.1 and Corollary III.15.2B]{DavermanR:Decm}. A subset $Z$ of $M$ is \emph{cellular} if for each open $U\supset Z$ there is an $n$-cell $E$ such that $Z\subset \interior E\subset E \subset U$; recall that an $n$-cell is a subset homeomorphic to $\bB^n$. A compact set $Z$ in a space $X$ is \emph{cell-like in $X$} if $Z$ can be contracted to a point in every neighborhood of $Z$.

Certain decomposition spaces can be constructed from \emph{defining sequences}. A \emph{defining sequence for a decomposition} of an $n$-manifold $M$ is a sequence $\cX=(X_k)_{k \ge 0}$ of compact sets satisfying $\inter X_k \supset X_{k+1}$. The \emph{decomposition} $G$ associated to the defining sequence $\cX$ consists of the components of $X_\infty = \bigcap_{k\ge 0} X_k$ and the singletons from $M\setminus X_\infty$, see \cite[p.61]{DavermanR:Decm}. The decomposition $G$ associated to $\cX$ is upper semi-continuous and $\pi_G(X_\infty)$ is compact and $0$-dimensional, see \cite[Proposition II.9.1]{DavermanR:Decm}.

In the context of defining sequences, a sufficient condition for $\R^3/G$ to be homeomorphic with $\R^3$ is the following shrinking criterion: \emph{For each $k\ge 1$ and each $\epsilon>0$, there exist $\ell\ge 1$ and a homeomorphism $h$ of $\R^3$ onto itself satisfying $h|(\R^3 \setminus X_k)=\id$, and $\diam h(H)<\epsilon$ for all components $H$ of $X_{k+\ell}$.}

\medskip

\noindent{\bf{Convention.}} In what follows, all decomposition spaces $\R^3/G$
are derived from defining sequences $\cX$ consisting of (unions of) cubes-with-handles. At times, we denote the space by $(\R^3/G, \cX)$ to emphasize the role of the sequence $\cX$.

\medskip

We fix some notations for following sections. Let $\cX=(X_k)_{k\ge 0}$ be a defining sequence. We denote by $\cC(\cX)=\bigcup_k \cC(X_k)$ all components of the defining sequence $\cX=(X_k)_{k \ge 0}$; here $\cC(E)$ denotes the set of all components of a set $E$.

Given $H\in \cC(\cX)$ there is a unique index $k\ge 0$ so that $H\in \cC(X_k)$. We call the index $k$ the \emph{level} of $H$ and denote $\lvl(H)=k$. For every $H\in \cC(\cX)$, we denote
\[
H^\diff=H\setminus \interior X_{\lvl(H)+1}.
\]
Then $\cC(H\setminus \interior H^\diff)$ consists of all components of $X_{\lvl(H)+1}$ contained in $H$.

Given two cubes-with-handles $H$ and $H'$ in $\cC(\cX)$, we have
\[
H=H',\ H' \subset \interior H,\ H\subset \interior H',\ \mathrm{or}\ H\cap H' = \emptyset.
\]
Thus $\partial H\cap X_\infty=\emptyset$ for every $H\in \cC(\cX)$. Since $X_\infty$ is closed in $\R^3$, there exists, for every $H\in \cC(\cX)$, a neighborhood $\Omega_{\partial H}$ of $\partial H$ in $\R^3$ so that $\pi_G|\Omega_{\partial H}$ is an embedding.

At times we shall write $\R^3/X_\infty$ for $\R^3/G$ for simplicity, in particular when $X_\infty$ is a Whitehead continuum, a necklace, a the Bing double or a Bing's dogbone space.

\subsection{Decomposition spaces as manifold factors}
\label{sec:manifold_factor}

Our main interest lies in decomposition spaces $\R^3/G$ that are homeomorphic to $\R^3$ or whose product with an Euclidean space, $\R^3/G\times \R^m$,  is homeomorphic to $\R^{3+m}$ for some $m>0$. Decomposition spaces of the latter type are called \emph{manifold factors of Euclidean spaces}.

By results of Sher and Alford and Lambert and Sher (\cite[Theorem 1]{SherR:Not0dd} and \cite{LambertH:Poil0d}), if $G$ is a cell-like usc decomposition of $\R^3$ so that the closure of all non-degenerate elements of $G$ is $0$-dimensional, then $G$ admits a defining sequence consisting of (unions of) cubes-with-handles. Subsequently Edwards and Miller \cite[p.192]{EdwardsR:Celc0d} proved that if $G$ satisfies the conditions of Lambert and Sher, then $\R^3/G$ is a \emph{factor} of $\R^4$, that is,
\begin{equation}\label{stabilization}
\R^3/G \times \R \approx \R^4,
\end{equation}
and $G\times\R$ is a shrinkable decomposition of $\R^4$, see also \cite[Section V. 27]{DavermanR:Decm}. In particular, the quotient map $\pi'\colon \R^{3+m}\to \R^{3+m}/(G\times\R^m)$ can be approximated by homeomorphisms. The composition $(\pi_G\times \id)\circ (\pi')^{-1}\colon \R^{3+m}/(G\times\R^m)\to \R^3/G \times \R^m $ is a homeomorphism \cite[Proposition I.2.4]{DavermanR:Decm}. Therefore
\[
\R^3/G \times \R^m \approx \R^{3+m},
\]
and $\pi_G\times \id \colon \R^{3+m}\to \R^3/G \times\R^m$ can be approximated by homeomorphisms.

Let $\R^3/G$ be a decomposition space associated to a locally contractible defining sequence $\cX=(X_k)_{k\ge 0}$ consisting of unions of cubes-with-handles. That is, every component of $X_{k+1}$ is contractible in $X_k$, for all $k\ge 0$. Then, by Edwards--Miller, $\R^3/G\times \R$ is homeomorphic to $\R^4$. Indeed, under this assumption on $\cX$, the components of $X_\infty$ are cell-like and $\pi_G(X_\infty)$ is compact and $0$-dimensional.

\subsection{Local contractibility}
\label{sec:top_lc}

In this section, we establish a local contractibility property for $\pi_G(X_k)$ in the decomposition space $\R^3/G$ from the local contractiblity of a defining sequence $\cX=(X_k)_{k\ge 0}$.

\begin{lemma}
\label{lemma:llc_1}
Let $\R^3/G$ be a decomposition space associated to a locally contractible defining sequence $\cX=(X_k)_{k\ge 0}$. Then components of $\pi_G(X_{k+1})$ are contractible in $\pi_G(X_k)$ for $k\ge 0$.
\end{lemma}

Lemma \ref{lemma:llc_1} follows directly from the following cellularity property of the decomposition $G\times \R^m$ for $m\ge 1$.
\begin{lemma}
\label{lemma:llc_2}
Let $m\ge 1$ and let $\R^3/G$ be a decomposition space associated to a locally contractible defining sequence $\cX=(X_k)_{k\ge 0}$. Then, for every $k\ge 0$, $H'\in \cC(X_k)$, $H\in \cC(X_{k+1} \cap H')$, and $r>0$, there exists a $(3+m)$-cell $E$ so that
\[
\pi_G(H)\times [-r,r]^m\subset E\subset \pi_G(H')\times (-2r,2r)^m.
\]
\end{lemma}

\begin{proof}[Proof of Lemma \ref{lemma:llc_1}]
Let $k\ge 0$, $H'\in \cC(X_k)$, and $H\in \cC(X_{k+1} \cap H')$.
To show that $\pi_G(H)$ is contractible in $\pi_G(H')$, let $m=1$, $r>0$, and $E$ a $4$-cell in $\R^3/G\times \R$ as in Lemma \ref{lemma:llc_2}.

Denote by $\mathrm{proj}$ the projection $\R^3/G \times \R \to \R^3/G$. We identify  $\R^3/G$ with $\R^3/G\times \{0\}$ in $\R^3/G\times \R$. Since $E$ is an $4$-cell, $\pi_G(H)$ is contractible in $E$. Thus $\pi_G(H)$ is contractible in $\mathrm{proj}(E) \subset \pi_G(H')$.
\end{proof}

The proof of Lemma \ref{lemma:llc_2} is based on an approximation of the quotient map $\pi_G\times \id \colon \R^{3+m} \to \R^3/G\times \R^m$ by homeomorphisms, and a classical Penrose--Whitehead--Zeeman lemma (\cite[Lemma 2.7]{PenroseR:Imbmec}): \emph{Let $M$ be an $n$-manifold and let $P\subset \interior M$ be an $(q-1)$-dimensional polyhedron $(1\le q\le n/2)$ such that the inclusion map $i\colon P \to M$ is homotopic in $M$ to a constant. Then there exists an $n$-cell $E\subset \interior M$ such that $P\subset \interior E$.}

\begin{proof}[Proof of Lemma \ref{lemma:llc_2}]
Let $r>0$, $k\ge 0$, and $H'\in \cC(X_k)$ and $H\in \cC(X_{k+1} \cap H')$ be cubes-with-handles in $\cX$. Let $\delta$ be any metric on the decomposition space $\R^3/G$, and $\delta_m$ be the product of $\delta$ with the Euclidean metric on $\R^3/G\times \R^m$. Let $a_0=\min\{r, \dist_\delta (\partial \pi_G(H),\partial \pi_G(H'))\}$.

We fix cores $\mathcal R$ and $\mathcal R'$ of $H$ and $H'$, respectively. Then $H$ and $H'$ are regular neighborhoods of $\mathcal R$ and $\mathcal R'$, respectively. By adding a one-sided collar of the boundary $\partial H$ to $H$, we obtain a regular neighborhood $H''$ of $\mathcal R$ containing $H$ in the interior. Similarly, by removing a one-sided collar of $\partial H'$ in $H'$, we obtain a regular neighborhood $H'''$ of $\mathcal R$ contained in $H'$; see \cite[Corollaries 2.26 and 3.17]{RourkeC:Intplt}. Moreover, we require that
\[
H\subset \interior H''\subset  H''\subset \interior H'''\subset  H''' \subset \interior H' \subset H',
\]
and that $a_0/10 <\dist_\delta(x, \partial \pi_G H)<a_0/9$ for all $x\in \partial \pi_G H''$ and $a_0/10 <\dist_{\delta}(x, \partial \pi_G H')<a_0/9$ for all $x\in \partial \pi_G H'''$.

Since $H$ is contractible in $H'$, we have that $H''$ is contractible in $H'''$. By the Penrose--Whitehead--Zeeman lemma, there exists a $(3+m)$-cell $E'$ so that
\[
\mathcal R\times \{0\} \subset H''\times (-\tfrac{5}{4}r, \tfrac{5}{4}r)^m \subset E'\subset H'''\times (-\tfrac{3}{2}r, \tfrac{3}{2}r)^m.
\]

Since $\pi_G\times \id$ can be approximated by homeomorphisms, by the Edwards--Miller theorem, we may fix a homeomorphism $h\colon  \R^{3+m}\to \R^3/G \times\R^m$ so that
\[
\max_{(x,v)\in X_0 \times [-3r,3r]^m} \delta_m(h(x,v), (\pi_G(x),v)) < a_0/100.
\]
Then $h^{-1}(\pi_G H \times [-r,r]^m)\subset  H''\times (-\frac{5}{4}r, \frac{5}{4}r)^m$ and $h ( H'''\times (-\frac{3}{2}r, \frac{3}{2}r)^m) \subset \pi_G H'\times (-2r, 2r)^m$. Thus $E=h(E')$ is a $(3+m)$-cell satisfying
\[
\pi_G(H) \times [-r,r]^m \subset E \subset \pi_G(H')\times (-2r, 2r)^m.
\]
\end{proof}

\section{Welding structures}
\label{sec:ws}

Let $n\ge 3$. By abusing the standard terminology in potential theory,  we say that a pair $(A,B)$ is a \emph{condenser in $\R^n$} if $A$ is a $3$-dimensional cube-with-handles in $\R^n$ and $B$ a disjoint union of $3$-dimensional cubes-with-handles in $\R^n$ so that $B\subset \interior A$; here $\interior A$ is the manifold interior of $A$.
Given a condenser $\cxi=(A,B)$, we denote by
\[
\cxi^\diff = A\setminus \interior B.
\]
Given two condensers $\cxi=(A,B)$ and $\cxi'=(A',B')$ in $\R^n$, a PL-embedding $\psi \colon \partial A' \to \partial B$ is said to be a \emph{welding of $\cxi'$ to $\cxi$}. Since $\partial A'$ is a closed surface and $\partial B$ is a disjoint union of closed surfaces in $\R^n$, $\psi(\partial A')$ is a component of $\partial B$.
Here $\partial M$ is the two dimensional manifold boundary of a $3$-manifold $M$.

Let $\cX$ be a defining sequence and $\fC$ a family of condensers in $\R^n$. Suppose that for each $H\in \cC(\cX)$, there exist a condenser  $\cxi_H=(A_H,B_H)\in \fC$ and a PL-homeomorphism $\varphi_H \colon H^\diff \to \cxi_H^\diff$  satisfying $\varphi_H(\partial H) = \partial A_H$ and $\varphi_H(\partial H^\diff\setminus \partial H) = \partial B_H$. Then we call
\[
\fA=\{\varphi_H\}_{H\in \cC(\cX)}
\]
an \emph{atlas for $\cX$}, and the elements of $\fA$ \emph{charts}.

Let $\fC$ be a family of condensers and $\fA=\{\varphi_H\}_{H\in \cC(\cX)}$ an atlas for $\cX$. Given $H\in \cC(\cX)$ and $H'\in \cC(H\cap X_{\lvl(H)+1})$, let $\cxi_H=(A_H,B_H)$ and $\cxi_{H'}=(A_{H'},B_{H'})$ be the corresponding condensers in $\fC$. We define the \emph{induced welding} $\psi_{H,H'} \colon \partial A_H' \to \partial B_H$ by the formula
\[
\psi_{H,H'} = \varphi_H \circ \varphi_{H'}^{-1}|\partial A_{H'}.
\]
We denote the induced welding scheme by
\[
\fW = \{ \psi_{H,H'} \colon \partial A_{H'} \to \partial B_H\}_{(H,H')},
\]
where $H\in \cC(\cX)$ and $H'\in \cC(H\cap X_{\lvl(H)+1})$;
\[
\xymatrix{
  & \partial H' \ar[dl]_{\varphi_{H'}|\partial H'} \ar[dr]^{\varphi_H|\partial H'} & \\
\partial A_{H'} \ar[rr]_{\psi_{H,H'}} & & \partial B_H }
\]
The triple $(\fC, \fA,\fW)$ is called a \emph{welding structure on $\cX$}.

\begin{figure}[h!]
\includegraphics[scale=0.65]{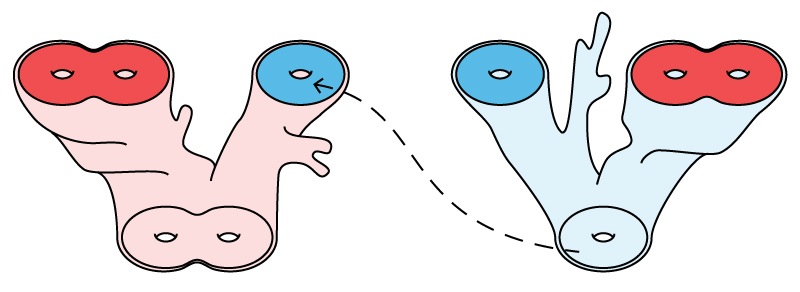}
\caption{A welding between two condensers.}
\label{fig:one_welding}
\end{figure}

\subsection{Defining sequences of finite type}
\label{sec:sft}

Recall from the Introduction that a defining sequence $(X_k)_{k\ge 0}$ is of \emph{finite type} if there exists a welding structure $(\fC,\fA,\fW)$ with $\#\fC<\infty$ and $\#\fW<\infty$. A decomposition space $(\R^3/G,\cX)$ is of finite type if $\cX$ has finite type.

The definition of a welding structure allows condensers to lie in high dimensional Euclidean spaces. However, a welding structure in $\R^3$ can always be built from the  original defining sequence.

\begin{proposition}
\label{prop:caw_existence}
Let $\cX$ be a defining sequence of finite type. Then there exists a welding structure $(\fC,\fA,\fW)$ in $\R^3$ so that
\begin{itemize}
\item[(i)] for each $\cxi=(A,B) \in\fC$, there exists $H\in \cC(\cX)$ for which $(A,B) = (H,H\cap X_{\lvl(H)+1})$; and
\item[(ii)] $\#\fC<\infty$ and $\#\fW<\infty$.
\end{itemize}
\end{proposition}

\begin{proof}
Let $(\fC',\fA',\fW')$ be a welding structure for $\cX$ so that $\#\fC'<\infty$ and $\#\fW'<\infty$. We may assume that each condenser $\cxi=(A,B)\in \fC'$ is the image of a chart, that is, there exists $H\in \cC(\cX)$ for which $\cxi_H=\cxi$ and $\varphi_H \colon (H^\diff,\partial H)\to (\cxi_H^\diff,\partial A_H)$. We fix for each $\cxi \in \fC'$ such a cube-with-handles and denote it by $H_\cxi$. Let $\phi_\cxi \colon \cxi^\diff \to H_\cxi^\diff$ be the inverse of the chart $\varphi_{H_\cxi}$.

Denote
\[
\fC = \{ (H_\cxi,H_\cxi \cap X_{\lvl(H_{\cxi})+1}\}_{\cxi\in \fC'}
\]
and
\[
\fA=\{ \phi_{\cxi_H} \circ \varphi_H\}_{H\in \cC(\cX)},
\]
Since $\cW'$ is a finite collection, the charts in $\fA$ induce a finite collection weldings $\fW$ between boundary components of condensers in $\fC$. Thus $(\fC,\fA,\fW)$ satisfies the conditions of the claim.
 \end{proof}

Let $\cX=(X_k)_{k\ge 0}$ be a defining sequence of finite type. Then the cubes-with-handles in $\cC(\cX)$ have uniformly bounded genus; we denote
\[
\bar \genus_{\cX} = \max\{ \genus(H) \colon H \in \cC(\cX)\}.
\]
Furthermore, $\cX$ has a finite \emph{(upper) growth}
\[
\bar \growth_{\cX} = \max\{ \# \cC(X_{k+1}\cap H) \colon H \in \cC(X_k),\ k\ge 0\}.
\]

\begin{definition}
\label{def:growth}
The \emph{order of growth $\growth_{\cX}$} of $\cX$ is defined to be
\begin{equation}
\label{eq:growth}
\growth_{\cX} = \lim_{r \to \infty} \max\{ \# \cC(X_{k+1}\cap H) \colon H \in \cC(X_k),\ k\ge r \}.
\end{equation}
\end{definition}

\subsection{Self-Similar Spaces}
\label{sec:ft_ip}

\emph{Self-similar decomposition spaces} are examples of decomposition spaces of finite type. Semmes' \emph{initial packages} for defining self-similar decomposition spaces yield almost directly finite welding structures on the defining sequences if the initial packages are understood in the PL-category instead of smooth category; see \cite[Section 3]{SemmesS:Goomsw}.

An \emph{initial package} $(T,T_1,\ldots, T_N, \phi_1,\ldots, \phi_N)$ consists of cubes-with-handles $T,T_1,\ldots,T_N$ in $\R^3$ with $T_i\subset \interior T$ and $T_i\cap T_{i'}=\emptyset$ for $i\neq i'$,  together with PL-embeddings $\phi_i \colon U \to T$ of a neighborhood $U$ of $T$ into $T$ so that $\phi_i(T)=T_i$ and the images $\phi_i(U)$ are mutually disjoint neighborhoods of $T_i$'s. The defining sequence $\cX=(X_k)_{k\ge 0}$ is given by $X_0 = T$ and
\[
X_k = \bigcup_{\alpha} \phi_\alpha(T)
\]
for $k\ge 1$, where $\alpha=(\alpha_1,\ldots, \alpha_k) \in \{1,\ldots, N\}^k$ and $\phi_\alpha = \phi_{\alpha_1}\circ \cdots \circ \phi_{\alpha_k}$.

Let $\cxi =\left(T,\bigcup_{i=1}^N \phi_i(T)\right)$ be a condenser. Then PL-homeomorphisms $(\phi_\alpha|\cxi^\diff)^{-1} \colon \phi_\alpha(T)^\diff \to \cxi^\diff$, $\alpha\in \bigcup_{k\ge 0} \{1,\ldots, N\}^k$, form an atlas $\fA$ for $\cX$. Although $\fA$ is an infinite atlas, the associated collection of weldings
\[
\fW=\{ \phi_i|\partial T \colon  1\le i \le N\}
\]
is finite. We call $(\{\cxi\}, \fA, \fW)$ the \emph{welding structure associated to the initial package $(T,T_1,\ldots, T_N, \phi_1,\ldots, \phi_N)$}.

\begin{figure}[h!]
\includegraphics[scale=0.60]{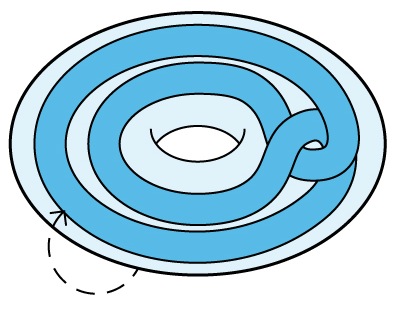}
\qquad \qquad \qquad \qquad
\includegraphics[scale=0.65]{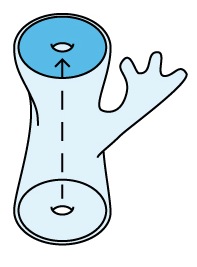}
\caption{Two welding structures associated to the \newline
Whitehead construction in Figure \ref{fig:Whitehead}.}
\end{figure}

We refer to \cite[Section 3]{SemmesS:Goomsw} for more details on initial packages for self-similar decomposition spaces.

\section{Rigid welding structures}

We introduce now rigid welding structures which correspond to \emph{excellent packages} of Semmes (\cite[Definition 3.2]{SemmesS:Goomsw}). In our terminology,  Semmes' excellent packages translate to welding structures with one condenser in $\R^4$, whose boundary lies entirely in $\R^3\times \{0\}$, and with similarities as weldings. Semmes showed the existence of excellent packages for defining sequences associated to the Whitehead continuum, Bing's dogbone, and Bing double; see \cite[Sections 4-6]{SemmesS:Goomsw}.

\begin{definition}
Let $(\fC,\fA,\fW)$ be a welding structure on a defining sequence $\cX$ of finite type. We call $(\fC,\fA,\fW)$ a \emph{rigid welding structure} in $\R^n$, $n\ge 4$, if $\fC$ consists of finitely many condensers and
\begin{enumerate}
\item[(S1)] all boundary components of differences $\{\cxi^\diff\colon \cxi\in \fC\}$ of the same genus are translations of one another,
\item[(S2)] weldings in $\fW$ are translations,
\item[(S3)] for every $\cxi=(A,B)\in \fC$ we have that $\partial A\subset \bB^3 \times \{0\} \subset \R^4 \subset \R^n$, $B \subset \bB^3\times\{1\} \subset \R^4\subset \R^n$, and $\interior (\cxi^\diff) \subset \bB^3\times (0,1) \times \R^{n-4}$.
\end{enumerate}
\end{definition}

For self-similar defining sequences, the existence of an excellent package induces a natural embedding of the space $\R^3/G$ into $\R^4$ (see \cite[Lemma 3.21]{SemmesS:Goomsw}). The possibility to place $B\cup \partial A$ on two separate levels and to use all dimensions $n\ge 4$ is less restrictive than the requirements for excellent packages. For this reason, all defining sequence of finite type admit rigid welding structures. Whereas Semmes' excellent packages lie in $\R^4$, the rigid welding structures lie in a fixed space $\R^{16}$. As the dimension of the ambient space containing condensers does not play a significant role in the construction of metrics, we make no attempt to obtain the optimal ambient dimension for rigid welding structures.

\begin{theorem}[{\bf Existence of rigid welding structures}]
\label{thm:existence_of_Semmes}
Let $(\R^3/G,\cX)$ be a decomposition space of finite type. Then $\cX$ admits a rigid welding structure in $\R^{16}$.
\end{theorem}

To straighten the condensers and the weldings between condensers, we apply the \emph{Klee trick}; see \cite[Proposition II.10.4]{DavermanR:Decm}.

\begin{lemma}
\label{lemma:Klee_trick}
Let $m\ge 1$, $k \ge 1$, $E$ a PL compact set in $\R^k$, and $f \colon E \to \R^m$ a PL-embedding. Then there exists a PL-homeomorphism $h \colon \R^{k+m} \to \R^{k+m}$ so that $h|E\times \{0\} = f$.
\end{lemma}

To obtain condensers satisfying (S3), we use the following lemma based on general position.

\begin{lemma}
\label{lemma:S3}
Suppose $\cxi=(A,B)$ is a condenser in $\R^n$, $n\ge 8$, so that $\partial A\subset \bB^3\times \{0\} \subset \R^4 \subset \R^n$ and $B \subset \bB^3\times \{1\}\subset \R^4\subset \R^n$. Then there exists is PL-embedding $F\colon A \to \R^n$ so that $F|\partial A\cup B=\id$ and $F(\cxi^\diff)\subset \bB^3\times (0,1)\times \R^{n-4}$.
\end{lemma}

\begin{proof}
We fix $t>1$ and $t'<0$ so that $A \subset \R^3\times (t',t) \times \R^{n-4}$.

Since $\partial B\times [1,t]$ is a $3$-dimensional PL-manifold in $\R^4\subset \R^n$ and $\cxi^\diff$ is $3$-dimensional, there exists, by general position (see \cite[Theorem 5.3]{RourkeC:Intplt}), a PL-homeomorphism $h \colon \R^3\times \R \times \R^{n-4} \to \R^3\times \R \times \R^{n-4}$ satisfying $h|\R^3\times (\R\setminus (1,t)) \times \R^{n-4} = \id$ and $h(\partial B\times (1,t])\cap \cxi^\diff = \emptyset$.

Let $B'=B+te_4$, $A'=(A\setminus B)\cup h(\partial B\times [1,t])\cup B'$, and $\cxi'=(A',B')$. Since $h(\partial B\times [1,t])$ is a one-sided collar of $\partial B$, there exists a PL-homeomorphism $k \colon A \to A'$ so that $k|\partial A = \id$ and $k|B$ is the translation $(x,1,y)\mapsto (x,t,y)$, where $x\in \R^3$ and $y\in \R^{n-4}$.

Let $g\colon \R^3\times \R \times \R^{n-4} \to \R^3\times \R \times \R^{n-4}$ be the map $(x,s,y) \mapsto (x,s/t,y)$. Then $\cxi''=g(\cxi')=(A'',B'')$ is a condenser so that $(\cxi'')^\diff\subset \R^3\times (-\infty,1)\times \R^{n-4}$. Note that $g \circ k |(\partial A \cup B) = \id$. Since the same argument can be applied to $[t',0]$, we may assume that $\interior (\cxi'')^\diff \subset \R^3\times (0,1)\times \R^{n-4}$.

We fix a piecewise linear function $\nu \colon \R \to (0,1)$ and a PL-homeomorphism $f \colon \R^3\times \R \times \R^{n-4} \to \R^3\times \R \times \R^{n-4}$, $f(x,s,y) = (\nu(s)x,s,y)$, so that $\nu(s) = 1$ for $s\not \in (0,1)$ and $f((\cxi'')^\diff \cap \R^3\times \{s\})\subset \bB^3\times \{s\}$ for $s\in (0,1)$. Since $f|\partial A\cup B = \id$, the composition $F=f\circ g \circ k$ satisfies the requirements of the claim.
\end{proof}

\begin{proof}[Proof of Theorem \ref{thm:existence_of_Semmes}]
Let $(\fC,\fA,\fW)$ be the welding structure in $\R^3$ associated to $\cX$ as in Proposition \ref{prop:caw_existence}. As a preliminary step, we fix, for every $0\le g\le \bar\genus_\cX$, a cube-with-handles $T_g$ of genus $g$ in $\R^3$.

\emph{Step 1:} We straighten the boundary components of condensers.

Let $\cxi=(A,B)$ be a condenser in $\fC$. We fix a point $z_D\in \R^3\times \{1\}$ for each component $D\in \cC(B)$ so that cubes-with-handles $T_{g_D}+z_D$ are pair-wise disjoint, where $g_D$ is the genus of $D$.
Fix also a regular neighborhood $E$ of $\partial A$ in $\R^3$ so that $E\cap B =\emptyset$ and an embedding $f\colon E\cup B \to \R^3\times \{0, 1\}$ such that $f(\partial A)=\partial T_{g_A}\subset \R^3\times \{0\}$ and $f(D)=T_{g_D}+z_D\subset \R^3\times \{1\}$ for each $D\in \cC(B)$. Then, by Lemma \ref{lemma:Klee_trick}, there exists a PL-homeomorphism $h_\cxi \colon \R^8 \to \R^8$ such that $h_\cxi(\partial A) = \partial T_{g_A}\subset \R^3 \times \{0\}$ and $h_\cxi(D) = T_{g_D}+z_D$ for every $D\in \cC(B)$.

Homeomorphisms $h_\cxi$ induce a new welding structure with condensers $\tilde \fC = \{ ( h_\cxi(A), h_\cxi(B)) \colon \cxi\in \fC\}$, atlas $\tilde \fA=\{h_{\cxi_H}\circ \varphi_H\colon H\in \cC(\cX) \}$, and the weldings $\tilde \fW $ defined by $\tilde \fC$ and $\tilde \fA$.

We denote the new structure $(\tilde \fC, \tilde\fA,\tilde \fW)$ in $\R^8$ again by $(\fC,\fA,\fW)$, and new condensers, charts and weldings again by $\cxi, \varphi_H,$ and $\psi_{H,H'}=\varphi_H\circ \varphi_{H'}^{-1}$ respectively.

\emph{Step 2:}
We now straighten the weldings from Step 1 to translations while expanding the collection of condensers.

Let $\psi \colon \partial A_1 \to \partial B_2$ be a welding in $\fW$ between condensers $(A_1,B_1)$ and $(A_2,B_2)$ in $\fC$.
Let $D\in \cC(B_2)$ be the component receiving $\psi$, that is, $\psi(\partial A_1)=\partial D$. Since $D$ and $A_1$ have the same genus, we may fix a translation $\tau_\psi\colon \partial A_1 \to \partial D$.
Set
\[
\hat \fW = \{ \tau_\psi\}_{\psi \in \fW}.
\]
We will add to $\fC$ a new condenser for each welding in $\hat \fW$ and modify the existing charts in $\fA$. This new atlas has $\hat \fW$ as the collection of induced weldings.

\medskip

\begin{figure}[h!]
\includegraphics[scale=0.65]{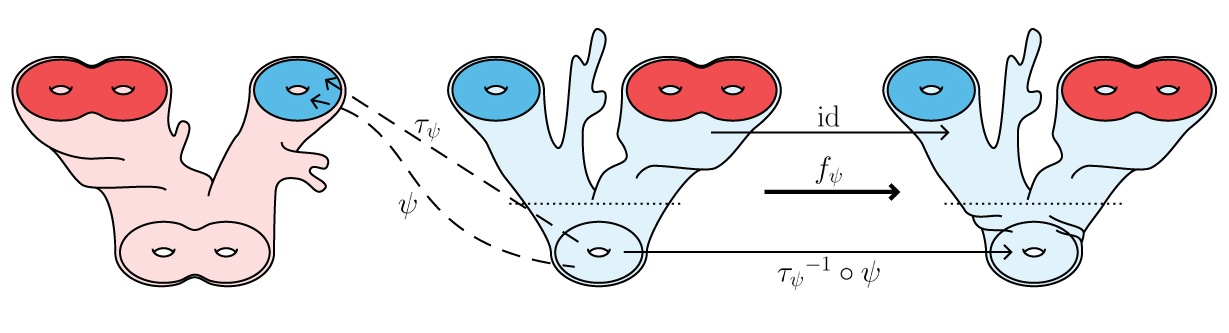}
\caption{The construction of a new condenser, $\cxi_\psi$.}
\label{fig:new_condenser}
\end{figure}

We first define the new condensers. Let $\psi \colon \partial A_1 \to \partial B_2$ be a welding in $\fW$ between condensers $(A_1,B_1)$ and $(A_2,B_2)$ in $\fC$.
Fix a one-sided collar $E$ of $\partial A_1$ in $A_1\setminus B_1$ and an open set $U\subset \R^8$ satisfying $A_1\cap U \subset E$. By the Klee trick (Lemma \ref{lemma:Klee_trick}), there exists a homeomorphism $f_\psi \colon \R^{16}\to \R^{16}$ so that $f_\psi|A_1\setminus U = \id$ and $f_\psi|\partial A_1 = \tau_\psi^{-1} \circ \psi$. We set $\cxi_\psi = (f_\psi(A_1),B_1)$, and note that $B_1 \subset A_1\setminus U$ and thus $f_\psi|B_1=\id$. We define
\[
\hat \fC = \fC\cup \{ \cxi_\psi \colon \psi\in \fW\}.
\]
Since $\#\fC+\#\fW<\infty$, $\hat \fC$ is a finite collection of condensers satisfying (S1).

We finish the proof by defining the atlas $\hat \fA$. For $H'\in \cC(\cX)$ with $\lvl(H')=0$, we define $\hat \cxi_{H'}=\cxi_{H'}$ and $\hat \varphi_{H'} = \varphi_{H'}$. Suppose now that $H'\in \cC(\cX)$ has level at least $1$ and let $H\in \cC(\cX)$ be the cube-with-handles satisfying $H'\in \cC(H\cap X_{\lvl(H)+1})$. Let $\varphi_H \colon H^\diff \to \cxi_H$ and $\varphi_{H'}\colon {H'}^\diff\to \cxi_{H'}$ be the corresponding charts in $\fA$, and $\psi = \psi_{H,H'}$ the welding induced by $\varphi_H$ and $\varphi_{H'}$. We denote $\hat \cxi_{H'} = \cxi_\psi$ and set $\hat\varphi_{H'} \colon H'^\diff \to \cxi^\diff_\psi$ by the formula $f_{\psi} \circ \varphi_{H'}$. Define
\[
\hat \fA = \{ \hat\varphi_{H'}\}_{H'\in \cC(\cX)}.
\]
To check that weldings induced by charts in $\hat \fA$ are in $\hat \fW$, let $H\in \cC(\cX)$ and $H'\in \cC(H\cap X_{\lvl(H)+1})$ be as above, and denote $\hat \cxi_{H'} = (\hat A_{H'},\hat B_{H'})$. Since $\hat \varphi_H|\partial H' = \varphi_H|\partial H'$, we have
\begin{eqnarray*}
\hat \varphi_H \circ \hat \varphi^{-1}_{H'}|\partial \hat A_{H'} &=& \varphi_H \circ \varphi^{-1}_{H'} \circ f^{-1}_{\psi_{H,H'}}|\partial \hat A_{H'} \\
&=& \psi_{H,H'} \circ f^{-1}_{\psi_{H,H'}}|\partial \hat A_{H'} = \tau_{\psi_{H,H'}}.
\end{eqnarray*}
Thus weldings are in $\hat \fW$ and satisfy (S2).

\emph{Step 3:} To obtain condensers satisfying condition (S3) we first apply a translation and a scaling in $\R^3$ (with the same scaling constant) to all condensers in $\hat \fC$ so that the assumptions of Lemma \ref{lemma:S3} are satisfies. Then we apply Lemma \ref{lemma:S3} and change the atlas accordingly.

\begin{figure}[h!]
\includegraphics[scale=0.65]{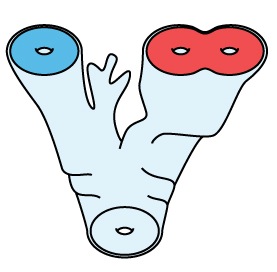}
\caption{A condenser in a rigid structure.}
\label{fig:rigid_condenser}
\end{figure}
\end{proof}

\section{Modular embeddings}
\label{sec:me}
In this section we discuss embeddings of decomposition spaces of finite type into Euclidean spaces.
Given a welding structure, we first introduce the notion of a \emph{modular embedding} of $\R^3/G$ into a Euclidean space, which respects the quasisimilarity type of that structure. This embedding defines a geometrically natural \emph{modular metric} on the decomposition space $\R^3/G$.

It has been shown in the previous section that a defining sequence of finite type admits a rigid welding structure. Theorem \ref{thm:FT_embedding} proves the existence of modular embedding with respect to any rigid structure.

Given a welding structure $(\fC,\fA,\fW)$ on a decomposition space $(\R^3/G,\cX)$ of finite type, $0<\lambda < 1$, and $n\ge 3$, we say that an embedding $\theta \colon \R^3/G \to \R^n$ is \emph{$\lambda$-modular} (with respect to $(\fC,\fA,\fW)$) if $\theta \circ \pi_G|(\R^3\setminus X_0) = \id$ and there exists $L\ge 1$ so that
\begin{equation}
\label{eq:str_emb}
\theta \circ \pi_G \circ \varphi_H^{-1} \colon \cxi_H^\diff \to \R^n
\end{equation}
is a $(\lambda^k, L)$-quasisimilarity for every $H\in \cC(X_k)$ and $k\ge 0$;
\[
\xymatrix{
H^\diff \ar[d]_{\pi_G|H^\diff} \ar[r]^{\varphi_H} & c_H^\diff \ar[d]^{\theta\circ \pi_G \circ \varphi_H^{-1}} \\
\R^3/G \ar[r]_{\theta} & \R^n}
\]

Given a $\lambda$-modular embedding $\theta \colon \R^3/G \to \R^n$ with respect to a welding structure $(\fC,\fA,\fW)$, we define the \emph{$\lambda$-modular metric} $d_\theta$ on $\R^3/G$ by
\begin{equation}
\label{eq:metric}
d_\theta(x,y) = |\theta(x)-\theta(y)|;
\end{equation}
here $|\cdot|$ is the Euclidean norm in $\R^n$.

We need the notion of compatible atlases to compare modular metrics induced by modular embeddings with respect to two different welding structures. Welding structures $(\fC,\fA,\fW)$ and $(\fC',\fA',\fW')$ on $\cX$ are said to have \emph{compatible atlases} if there exists $L\ge 1$ so that
\begin{equation}\label{eq:compatible}
\varphi'_H \circ \varphi_H^{-1}|\cxi_H^\diff \colon \cxi_H^\diff \to (\cxi'_H)^\diff
\end{equation}
is $L$-bilipschitz for every $H\in \cC(\cX)$, where homeomorphisms $\varphi_H \colon H^\diff \to \cxi_H^\diff$ and $\varphi'_H \colon H^\diff \to (\cxi'_H)^\diff$ are charts in $\fA$ and $\fA'$, respectively.

\begin{lemma}
\label{lemma:modular_bilip_invariance} Let $(\R^3/G, \cX)$ be a decomposition space of finite type and $\lambda \in (0,1)$. Suppose $(\fC_i,\fA_i,\fW_i)$, $i=1,2$, are welding structures on $\cX$ having compatible atlases, and let $\theta_i \colon \R^3/G \to \R^{m_i}$ be $\lambda$-modular embeddings associated to $(\fC_i,\fA_i,\fW_i)$, respectively. Then path metrics $\hat d_{\theta_1}$ and $\hat d_{\theta_2}$ on $\R^3/G$ are bilip\-schitz equivalent.
\end{lemma}

\begin{proof}
Since $(\fC_1,\fA_1,\fW_1)$ and $(\fC_2,\fA_2,\fW_2)$ have compatible atlases, there exists $L\ge 1$ so that, for every $H\in \cC(\cX)$,
\[
\tau_H = \varphi_H^2 \circ (\varphi_H^1)^{-1} \colon (\cxi_H^1)^\diff \to (\cxi_H^2)^\diff
\]
is $L$-bilipschitz, where $\varphi_H^i \colon H^\diff \to (\cxi_H^i)^\diff$ is the chart for $H$ in $\fA_i$ for $i=1,2$.

Since $\theta_1$ and $\theta_2$ are $\lambda$-modular embeddings, we also have constants $L_1$ and $L_2$ so that $\theta_i\circ \pi_G \circ (\varphi_H^i)^{-1}$ is $(\lambda^k, L_i)$-quasisimilarity for each $H\in \cC(X_k)$ and every $k\ge 0$.

Let $H\in \cC(\cX)$. We denote
\[
\theta_H = \theta_2 \circ \theta_1^{-1}|\theta_1(\pi_G(H^\diff)) \colon \theta_1(\pi_G(H^\diff))\to \theta_2(\pi_G(H^\diff)).
\]
Since
\[
\theta_H = \left(\theta_2\circ \pi_G \circ (\varphi_H^2)^{-1}\right) \circ \tau_H \circ (\theta_1\circ \pi_G\circ (\varphi_H^1)^{-1})^{-1},
\]
$\theta_H$ is $LL_1L_2$-bilipschitz.

Let $\Omega = (\R^3/G)\setminus \pi_G(X_\infty)$. Since $\theta_H$ is uniformly bilipschitz on each $\theta_1(\pi_G(H^\diff)$, we observe that $\theta_2 \circ \theta_1^{-1}|\theta_1(\Omega)$ is a bilipschitz map in the path metric from $\theta_1(\Omega)$ to $\theta_2(\Omega)$. Since $\theta_1(\R^3/G)$ is the closure of $\theta_1(\Omega)$, we observe that $\theta_2\circ \theta_1^{-1}$ is bilipschitz in the path metric from $\theta_1(\R^3/G)$ to $\theta_2(\R^3/G)$. The claim now follows.
\end{proof}

We state the Modular Embedding Theorem with respect to a given rigid welding structure as follows.
\begin{theorem}[{\bf{Modular Embedding Theorem}}]
\label{thm:FT_embedding}
Let  $(\R^3/G, \cX)$ be a decomposition space of finite type and $(\fC,\fA,\fW)$ a rigid welding structure on $\cX$. Then for every $0<\lambda<1$, there exists a $\lambda$-modular embedding $\theta \colon \R^3/G\to \R^n$ where $n\ge 16$, whose image $\theta(\R^3/G)$ is quasiconvex in the Euclidean metric. Moreover, there exists $L=L(\theta)\ge 1$ so that, any two distinct points $x,y\in \theta(\R^3/G)$ are contained in an $L$-bilipschitz image of a closed Euclidean $3$-ball of radius $|x-y|$.
\end{theorem}

The proof of Modular Embedding Theorem is divided into two parts. First we consider a tree $\tree_\cX$ derived from the combinatorial structure of the defining sequence $\cX$ and a bilipschitz embedding of $\tree_\cX$ into some Euclidean space $\R^d$. In the second part, we obtain an embedding of $\R^3/G$ into $\R^{16+d}$ by gluing reshaped and rescaled condensers in a rigid welding structure provided by Theorem \ref{thm:existence_of_Semmes}. This gluing is guided by the embedded structural tree $\tree_\cX$.

\subsection{Combinatorial trees}
\label{sec:combi_tree}

Let $\R^3/G$ be a decomposition space with a defining sequence $\cX=(X_k)$. We denote by $\tree_\cX$ the tree with vertices $\cC(\cX)$ and unoriented edges $\langle H,H'\rangle$, where $H\in \cC(X_k)$ and $H'\in \cC(H \cap X_{k+1})$.

Given $H,H'\in \cC(\cX)$, we define
\begin{equation}\label{eq:depth}
\depth_\cX(H,H') = \max\{ \lvl(H'') \in \Z\colon H\cup H'\subset H'' \in \cC(\cX)\}.
\end{equation}
Since $\tree_\cX$ is a tree there exists a unique shortest chain $H=H_1,\ldots, H_\ell=H'$ so that $\langle H_i,H_{i+1}\rangle$ is an edge in $\tree_\cX$ for every $i=1,\ldots, \ell-1$. In particular, there exists unique index $i_0=i_0(H,H')$ so that $\lvl(H_{i_0})=\depth_\cX(H,H')$.

Given $\lambda>0$ we define the metric $\delta_\lambda$ on $\tree_\cX$ by the formula
\[
\delta_\lambda(H,H') = \sum_{i=1}^{\ell-1} \lambda^{\min\{\lvl(H_i),\lvl(H_{i+1})\}},
\]
where $H,H'\in \cC(\cX)$ and the sum is taken over the shortest chain $H=H_1,\ldots, H_\ell=H'$. The metric $\delta_1$ is the standard \emph{graph distance} on $\tree_\cX$. The definition of the metric $\delta_\lambda$ immediately yields a distance estimate
\begin{equation}
\label{eq:delta_d}
\lambda^{\depth_\cX(H,H')} \le \delta_\lambda(H,H') \le C \lambda^{\depth_\cX(H,H')}
\end{equation}
for all $H,H'\in \cC(\cX)$, $H\ne H'$, where $C=C(\lambda)$.

This distance estimate implies that metric trees $(\tree_\cX,\delta_\lambda)$, $\lambda >0$, are quasisymmetrically equivalent. We record this observation in the following lemma.

\begin{lemma}
\label{lemma:tree_qs}
Let $\lambda_1,\lambda_2 >0 $. The identity map $(\tree_\cX,\delta_{\lambda_1})\to (\tree_\cX,\delta_{\lambda_2})$ is  $\eta$-quasisymmetric with $\eta(t) = C t^p$, where $p=\log \lambda_2/\log \lambda_1$ and $C=C(\lambda_1,\lambda_2)$.
\end{lemma}

The metric trees $(\tree_\cX,\delta_\lambda)$ embed bilipschitzly into Euclidean spaces. Recall that $(e_1,\ldots,e_n)$ is the standard basis of $\R^n$ for $n\ge 1$.

\begin{lemma}
\label{lemma:tree_embeds}
Let $(\R^3/G,\cX, (\fC,\fA,\fW))$ be a decomposition space of finite type and $0<\lambda <1$. Then there exist $n= n(\cX,\lambda)$ and a map $e_\cX \colon \cC(\cX)\to \{e_1,\ldots, e_n\}$ so that the map $\vartheta \colon (\tree_\cX,\delta_\lambda) \to \R^n$ defined inductively by $\vartheta(X_0)=0$ and $\vartheta(H') = \vartheta(H)+\lambda^k e_\cX(H')$ for $H\in \cC(X_k)$ and  $H'\in \cC(H \cap X_{k+1})$, is a bilipschitz embedding.
\end{lemma}

\begin{proof}
Let $m_0>0$ be the smallest integer satisfying
\begin{equation}
\label{eq:k_0}
\sum_{j=1}^\infty \lambda^{jm_0} < 1/4.
\end{equation}
Since $\cX$ has finite type, there exists $n$ depending on $\fC$ and $m_0$, thus depending on $\cX$ and $\lambda$, so that
\[
\# \left( \bigcup_{i=k}^{k+2m_0} \cC(X_i\cap H) \right) \le n
\]
for all $k\ge 0$ and $H\in \cC(X_k)$. We fix a map $e_\cX\colon \cC(\cX) \to \{e_1,\ldots, e_n\}$ so that if $e_{\cX}(H)=e_{\cX}(H')$ then the graph distance $\delta_1(H,H')\ge m_0$.

We show now that the mapping $\vartheta \colon  \tree_\cX \to \R^n$, defined in the statement, is a bilipschitz embedding.

Let $H,H'\in \cC(\cX)$ and let $H=H_1, \ldots, H_\ell = H'$ be the unique shortest chain.
Let $I_j = \{ i \colon 1 \le i \le \ell, i\neq i_0(H,H')\, \text{and}\,\ e_\cX(H_i)=e_j\}$ for $j=1,\ldots, n$.
Then
\begin{eqnarray*}
\vartheta(H)-\vartheta(H') = \sum_{i=1}^{\ell-1} \vartheta(H_i)-\vartheta(H_{i+1}) = \sum_{j=1}^n \left( \sum_{i \in I_j} \pm \,( \vartheta(H_i)-\vartheta(H_{\tilde{i}}))\right),
\end{eqnarray*}
where $\tilde{i}$ is either $i+1$ or $i-1$ such that $\lvl H_i=\lvl H_{\tilde{i}} +1$, $+$ sign is chosen when $\tilde{i}=i+1$, and $-$ sign is chosen when $\tilde{i}=i-1$.

By orthogonality,
\[
\left| \vartheta(H)-\vartheta(H')\right| = \left( \sum_{j=1}^n \left| \sum_{i\in I_j} \pm \,(\vartheta(H_i)-\vartheta(H_{\tilde{i}}))\right|^2 \right)^{1/2}.
\]
Since $\vartheta(H_i)- \vartheta(H_{\tilde{i}}) = \lambda^{\lvl(H_i)} e_j$ for $i\in I_j$, we have
\[
\frac{3}{4} \lambda^{k_j} \le \left| \sum_{i\in I_j}\pm \,( \vartheta(H_i)-\vartheta(H_{\tilde{i}}))\right| \le \frac{5}{2} \lambda^{k_j},
\]
where  $k_j = \min\{ \lvl(H_i) \colon i\in I_j\}$. Since
\[
\depth_\cX(H,H') = \min\{ k_j \colon 1\le j \le n\}-1,
\]
we have
\[
\frac{3}{4}\lambda^{\depth_\cX(H,H')+1} \le |\vartheta(H)-\vartheta(H')| \le \frac{5\sqrt{n}}{2} \lambda^{\depth_\cX(H,H')+1}.
\]
Thus, by \eqref{eq:delta_d}, $\vartheta$ is bilipschitz.
\end{proof}

\subsection{Bending and reshaping of condensers}

Suppose that $\cxi=(A,B)$ is a condenser in a rigid welding structure in $\R^m$ for some $m\ge 4$ and that $e \colon \cC(B) \to \{e_{m+1},\ldots, e_{m+n}\}$ is an injection, where $(e_1,\ldots, e_{m+n})$ is an orthonormal basis of $\R^{m+n}$.  We say that
a bilipschitz PL-homeomorphism $b_{\cxi,e} \colon \R^{m+n} \to \R^{m+n}$ is a \emph{bending of $\cxi$ by $e$} if
\begin{enumerate}
\item $b_{\cxi,e}|\partial A = \id$,
\item $b_{\cxi,e}|D \colon x \mapsto x + e(D)$ for every $D\in \cC(B)$, and
\item $b_{\cxi,e}(\interior \cxi^\diff)\subset \bB^3\times(0,1)\times \R^{m+n-4}$.
\end{enumerate}
Bendings of $\cxi$ by $e$ can be easily found.

Let $k\ge 4$ and $\lambda \in (0,1)$. We define the \emph{$\lambda$-reshaping} $s_\lambda \colon \R^k \to \R^k$ to be
\[
s_\lambda(x,t,y) = (c(t)x,t,y)
\]
for $(x,t,y)\in \R^3\times \R\times \R^{k-4}$, where
\[
c(t) = \left\{ \begin{array}{ll}
\lambda, & t \ge 1 \\
1-(1-\lambda)t, & 0\le t \le 1 \\
1, & t \le 0.
\end{array}\right.
\]

\subsection{Proof of the Modular Embedding Theorem}
To prove Theorem \ref{thm:FT_embedding}, we
construct first an auxiliary sequence of PL submanifolds $(M_j)$ of a fixed Euclidean space which tends to a PL submanifold $M_\infty$. The image of the embedding $\theta|\pi_G(\R^3\setminus X_\infty)$ will be the manifold $M_\infty$. This embedding is  then extended to $\R^3/G$ by continuity.

Assume, as we may by Theorem \ref{thm:existence_of_Semmes}, that $(\fC,\fA,\fW)$ is a rigid welding structure for $\cX$ in $\R^{16}$.

\emph{Auxiliary sequence $(M_j)$.} Let  $e_\cX\colon \cC(\cX)\to \{e_{16+1},\ldots, e_{16+n}\}$ be the map and $\vartheta \colon \tree_\cX \to \{0\}\times \R^n$ be the embedding defined in Lemma \ref{lemma:tree_embeds}, with a natural shift of coordinates; recall that $\vartheta(X_0)=0$.

We enumerate cubes-with-handles in $\cC(\cX)$ by $H_0,H_1,\ldots$ so that $H_0=X_0$ and if $H_i\in \cC(X_k)$ then $H_{i+1}\in \cC(X_k)\cup \cC(X_{k+1})$. We may assume that the condenser $\cxi_0=(A_0,B_0)$ in $\fC$ and the chart $\varphi_{X_0} \colon X_0^\diff \to \cxi_0^\diff$  are chosen so that $A_0=X_0$ and $\varphi_{X_0}|\partial X_0  = \id$. We denote by $\cxi_i= (A_i,B_i)$ the condensers $\cxi_{H_i} \in \fC$ and by $\varphi_i$ the charts $\varphi_{H_i}\colon H_i^\diff \to \cxi_i^\diff$ in $\fA$ for $i\ge 0$.

Submanifolds $(M_j)$ will constructed by gluing together bended and reshaped condensers $\{\cxi_i\colon i\geq 0\}$ guided by the embedded tree  $\vartheta(\tree_\cX)$.

We start by defining the directions for bending. Given $i\ge 0$, we denote by $\Phi_i \colon \cC(H_i\cap X_{\lvl(H_i)+1}) \to \cC(B_i)$ the bijection between components induced by charts $\varphi_{H_i}$ so that $\Phi_i|\partial H' = \varphi_{H_i}|\partial H'$ for $H' \in \cC(H_i\cap X_{\lvl(H_i)+1})$. Furthermore, we define $e^i\colon \cC(B_i) \to \{e_{16+1},\ldots, e_{16+n}\}$ by $e^i = e_\cX \circ \Phi_i^{-1}$.
\[
\xymatrix{
\cC(H_i\cap X_{\lvl(H_i)+1}) \ar[r]^{\Phi_i} \ar[d]^{\text{incl.}} & \cC(B_i) \ar[d]^{e^i} \\
\cC(\cX) \ar[r]_{e_\cX} & \{e_{16+1},\ldots, e_{16+n}\}
}
\]
We then fix a family of  bendings $\{b_i=b_{\cxi_i,e^i} \colon \R^{16+m} \to \R^{16+m} \colon i \ge 0\}$. To ensure that the collection $\{b_i\colon i\ge 0\}$ of bendings is finite, we require $b_i=b_j$ when $(\cxi_i,e^i)=(\cxi_j,e^j)$.
Thus, bendings in $\{b_i, i\ge 0\}$  are uniformly bilipschitz.

Fix a constant $C_0>0$ so that  $B^{16+n}(C_0)$ contains all condensers in $\fC$.

Let $M_{-1} = \left( \R^3\setminus X_0\right) \times \{0\}\subset\R^4\subset\R^{16} \subset \R^{16+n}$ and let $\theta_{-1} \colon \R^3\setminus  X_0 \to \R^{16+n}$ to be the natural inclusion. We define $M_0$ by
\[
M_0 = M_{-1} \cup \cxi_0^\diff
\]
and an embedding $\theta_0 \colon \R^3\setminus X_1 \to \R^{16+n}$ by $\theta_0|\R^3\setminus X_0 = \theta_{-1}$ and $\theta_0|X_0^\diff = \varphi_{X_0}$.

Suppose now that we have defined manifolds $M_{-1},\ldots, M_{j-1}$ and embeddings $\theta_{-1},\ldots, \theta_{j-1}$ so that, for $i=-1,\ldots, j-1$,
\begin{enumerate}
\item $M_i = M_{i-1}\cup f_i(\cxi_i^\diff)$ where $H_i\in \cC(X_{k_i})$, $f_i$ is the quasisimilarity
\begin{equation}
\label{eq:f_i}
x\mapsto \lambda^{k_i} (s_\lambda\circ b_i)(x) + \vartheta(H_i) + w_i,
\end{equation}
and $w_i$ is a point in $\R^{16}$ satisfying $|w_i|<C_0 \sum_{0\le r\le k_i} \lambda^r$; and
\item the embedding $\theta_i\colon (\R^3\setminus X_0) \cup (H_0^\diff \cup \cdots \cup H_i^\diff) \to M_i$ is defined by $\theta_i|(\R^3\setminus X_0) \cup (H_0^\diff \cup \cdots \cup H_{i-1}^\diff) = \theta_{i-1}$ and $\theta_i|H_i^\diff = f_i \circ \varphi_i$.
\end{enumerate}
We construct the set $M_j$ and the embedding $\theta_j$ as follows.
Suppose that $H_j\in \cC(X_k)$ and that $H_i,\, i<j \, ,$ is the unique cube-with-handles in $\cC(X_{k-1})$ with the property $H_j\in \cC(H_i \cap X_k)$;
and let $\psi=\varphi_i \circ \varphi_j^{-1}$ be the welding map from $\cxi_j$ to $\cxi_i^\diff$. Since $(\fC,\fA,\fW)$ is a rigid structure, $\psi$ is a translation $x\mapsto x+v_\psi$ in $\R^{16}$, where $v_\psi\in \R^{16}$, $\langle v_\psi,e_4\rangle =1$, and $|v_\psi|<C_0$ as in the construction.
By induction hypothesis, $f_i|\varphi_i(\partial H_j)$ is a similarity
\[
x\mapsto \lambda^k  x+ \lambda^k e_\cX(H_j) +\vartheta(H_i) + w_i=\lambda^k x+\vartheta(H_j) + w_i;
\]
here we use the fact that the bending $b_i|\varphi_i(\partial H_j)=e^i(H_j)= e_\cX(H_j)$ and the reshaping $s_\lambda$ on $\partial \varphi_i(H_j)$ is a scaling by $\lambda$. We set $f_j$ to be the quasisimilarity
\[
x \mapsto \lambda^k (s_\lambda \circ b_j)(x) +  \vartheta(H_j) + \lambda^k v_\psi+ w_i.
\]
Set also $w_j =\lambda^k  v_\psi + w_i \in \R^{16}$, and note by induction hypothesis that $|w_j|\le C_0 \sum_{0\le r \le k} \lambda^r $.

Since $\cxi_j$ is a condenser in a rigid structure and $\langle v_\psi,e_4\rangle =1$, we have that $M_{j-1}\cap f_j(\cxi_j)$ is a common boundary component of $M_{j-1}$ and $f_j(\cxi_j)$. Thus $M_j = M_{j-1}\cup f_j(\cxi_j)$ is a connected manifold with boundary satisfying (1) in the induction hypothesis.

We define now the embedding $\theta_j\colon (\R^3\setminus X_0) \cup (H_0^\diff \cup \cdots \cup H_j^\diff) \to M_j$ by formula $\theta_j|\R^3\cup (H_0^\diff \cup \cdots \cup H_{j-1}^\diff) = \theta_{j-1}$ and $\theta_j|H_j^\diff = f_j \circ \varphi_j$. This completes the induction step.

\emph{Construction of $M_\infty$.}
Define now the limit manifold $M_\infty$ by
\[
M_\infty = \bigcup_{j \ge 0} M_j
\]
and the limiting embedding $\theta_\infty \colon \R^3\setminus X_\infty \to M_\infty$ by $\theta_\infty|M_j = \theta_j$.

Since there exists $C>0$ so that $\diam \theta(H)\le C \lambda^k$ for every $H\in \cC(X_k)$, the components of $\overline{M_\infty}\setminus M_\infty$ are singletons. Thus $\theta_\infty\circ \pi_G^{-1}$ extends to a homeomorphism $\theta \colon \R^3/G \to \overline{M_\infty}$.

Note that  $\theta \circ \pi_G \circ \varphi_j^{-1} \colon \cxi_j^\diff \to \R^{16+n}$ is a  $(\lambda^{\lvl(H_j)},L)$-quasisimilarity, for a constant $L\ge 1$ depending only on the family of bendings $\{b_i\colon i\ge  0\}$ and $n$. Thus $\theta$ is a $\lambda$-modular embedding.

\emph{Metric properties of $\theta(\R^3/G)$.}  We show now the last claim in the statement: given $x,y \in \theta(\R^3/G)$, there exists an $L'$-bilipschitz  map $h \colon B^3(|x-y|) \to \theta(\R^3/G)$ so that $x,y \in h(B^3(|x-y|))$, where $L'=L'(\theta)\ge 1$.  In particular, $\theta(\R^3/G)$ is quasiconvex.

It suffices to consider the case $x, y \in \theta(\pi_G(\R^3\setminus X_\infty))$; other cases are obtained by similar arguments.

We observe, by the $(\lambda^{\lvl(H)},L)$-quasisimilarity of the mappings $\theta \circ \pi_G\circ \varphi_j^{-1}$ and the finiteness of condensers in $\fC$, the following. If $H$ and $H'$ are condensers in $\cC(\cX)$ satisfying $H^\diff \cap H'^\diff \neq \emptyset$, then any two points $x$ and $y$ in $\theta( \pi_G(H^\diff \cup H'^\diff))$ can be connected by a PL-curve contained in a $3$-cell in $\theta(\pi_G(H^\diff \cup H'^\diff))$ that is $L'$-bilipschitz equivalent to a Euclidean ball of diameter $|x-y|$, where $L'$ depends only on the data. In particular, the claim holds in  this case.

We now assume $x\in \theta(\pi_G(H^\diff))$, $y\in \theta(\pi_G(H'^\diff))$ and $H^\diff \cap H'^\diff = \emptyset$. Let $H=H_0, \ldots, H_\ell = H'$ be the unique shortest chain  in $\tree_\cX$ joining vertices $H$ and $H'$, and  $\depth_\cX(H,H') = \min\{ \lvl(\hat H) \in \Z\colon H\cup H'\subset \hat H \in \cC(\cX)\}.$ Then by the construction of the embedding $\theta$, there exists $C=C(\theta)\ge 1$ so that
\begin{equation}\label{eq:xy}
C^{-1} \lambda^{\depth_\cX(H,H') }  \le |x-y| \le C \lambda^{\depth_\cX(H,H') }.
\end{equation}
There exist $C'=C'(\theta)\ge 1$ and points $x=x_0, \ldots, x_\ell = y$ with $x_i\in \theta(\pi_G(H_i^\diff))$ so that each $x_i$, $1\le i \le \ell-1$, is contained in a $3$-cell $D_i\subset \theta(\pi_G(H_i^\diff))$ which is $L'$-bilipschitz equivalent to $B^3(\lambda^{\lvl(H_i)})$, and that
\[
C'^{-1}\lambda^{\lvl(H_i)} \le |x_i-x_{i+1}| \le C \lambda^{\lvl(H_i)}
\]
for $0 \le i \le \ell-1$. Consequently,
\[
\sum_{i=0}^{\ell-1} |x_i-x_{i+1}| \le C |x-y|.
\]
By the argument for the previous case, we find PL $3$-cells $E_i\subset \pi_G(H_i^\diff \cup H_{i+1}^\diff)$ that are $L'$-bilipschitz equivalent to $B^3(|x_i-x_{i+1}|)$ and contain points $x_i$ and $x_{i+1}$ in their interiors in $\pi_G(H_i^\diff \cup H_{i+1}^\diff)$, respectively. It is now easy to find a PL $3$-cell $E \subset \bigcup_{i=1}^{\ell-1} D_i \cup \bigcup_{i=0}^{\ell-1} E_i$ that is $L'$-bilipschitz equivalent to $B^3(|x-y|)$ and contains points $x$ and $y$. This concludes the proof of Theorem \ref{thm:FT_embedding}.

\medskip

\begin{remark}\label{rmk:bilip_ball}
The fact that any two points $x, y$ in $\theta(\R^3/G)$ are contained in a $3$-cell in $\theta(\R^3/G)$ that is $L$-bilipschitz equivalent to a Euclidean ball of diameter $|x-y|$, yields that $\theta(\R^3/G)$ has the Loewner property. We formulate this more precisely in Section \ref{sec:Loewner}.
\end{remark}

\section{Semmes spaces}
\label{sec:Semmes}

In this section we discuss quasiconvexity, Ahlfors regularity, linearly locally contractibility and the Loewner property of the modular metrics provided by the Modular Embedding Theorem as listed in Theorem \ref{thm:summary}.

\begin{definition}
\label{def:Semmes_metric}
Let $(\R^3/G,\cX)$ be a decomposition space of finite type, $(\fC,\fA,\fW)$ a rigid welding structure for $\cX$, $\theta \colon \R^3/G \to \R^n$ a modular embedding associated to $(\fC,\fA,\fW)$ as in Theorem \ref{thm:FT_embedding}, and let $d_\theta$ be a $\lambda$-modular metric associated to $\theta$. A metric space $(\R^3/G,d_\lambda)$ is called a \emph{Semmes space} if $d_\lambda$ is bilipschitz equivalent to $d_\theta$. In this case we say $d_\lambda$ is a \emph{Semmes metric} with a \emph{scaling factor $\lambda$}.
\end{definition}

At times we say that $(\R^3/G,\cX, (\fC,\fA,\fW), \theta, d_\lambda)$ is a Semmes space in order to emphasize the relation to between the structure, the embedding, and the metric.

Product spaces $\R^3/G\times \R^m$ carry the natural product metric $d_{\lambda,m}$ defined by
\begin{equation}
\label{eq:d_lambda_m}
d_{\lambda,m}((x,u),(y,v)) = d_\lambda(x,y) + |u-v|
\end{equation}
for $(x,u)$ and $(y,v)$ in $\R^3/G \times \R^m$.

We observe that the \emph{metric space $(\R^3/G,d_\lambda)$ is quasiconvex}; indeed, $\theta(\R^3/G)$ is a quasiconvex set in $\R^n$ by Theorem \ref{thm:FT_embedding}. By quasiconvexity and Lemma \ref{lemma:modular_bilip_invariance} we have the bilipschitz equivalence of modular metric spaces associated to rigid welding structures with compatible atlases. We record this observation as a lemma.

\begin{lemma}
\label{lemma:bilip_invariance} Let $\lambda\in (0,1)$ and suppose that $(\R^3/G, \cX,(\fC_i,\fA_i,\fW_i),\theta_i,d_{\theta_i})$ are $\lambda$-modular metric spaces, $i=1,2$. The metrics $d_{\theta_1}$ and $d_{\theta_2}$ are bilip\-schitz equivalent if rigid welding structures $(\fC_1,\fA_1,\fW_1)$ and $(\fC_2,\fA_2,\fW_2)$ have compatible atlases.
\end{lemma}

\subsection{Metric properties}\label{sec:metric properties}

We list some elementary metric and measure theoretic properties of Semmes spaces in the following remarks and the subsequent lemma. Let $(\R^3/G,\cX, (\fC,\fA,\fW), \theta, d_\lambda)$ be a Semmes space.

\begin{remark}
\label{rmk:pathmetric}
By quasiconvexity of $d_\theta$, the path metric space $(\R^3/G,\hat d_\theta)$ is a Semmes space. Similarly, the path metric space $(\R^3/G,\hat d_\lambda)$ of $(\R^3/G,d_\lambda)$ is a Semmes space.
\end{remark}

\begin{remark}
\label{rmk:comparison}
By modularity of the embedding $\theta$ and quasiconvexity of the metric $d_\theta$ there exists a constant $C=C(d_\lambda)$ so that
\[
C^{-1} \lambda^{\rho_\cX(H,H')} \le d_\lambda(x,y) \le C \lambda^{\rho_\cX(H,H')}
\]
for $x\in \pi_G(H^\diff)$ and $y\in \pi_G(H'^\diff)$ whenever $H,H'\in \cC(\cX)$ and $H^\diff \cap H'^\diff = \emptyset$.
\end{remark}

\begin{remark}
\label{rmk:mp1}
By the finiteness of the welding structure $(\fC,\fA,\fW)$ and quasisimilarity property \eqref{eq:str_emb} of modular embeddings, there exists $C>1$ so that for every $k\ge 0$ and $H \in \cC(X_k)$,
\begin{enumerate}
\item $ C^{-1} \lambda^k \le \dist_{d_\lambda}(\partial \pi_G(H),\partial \pi_G(H')) \le C \lambda^k$,  if $H'\in \cC(\cX)$, $H' \subset H$ and  $H'\ne H$;
\item $C^{-1} \lambda^k \le \diam_{d_\lambda} \pi_G H^\diff \le C\lambda^k$;
\item $C^{-1} \lambda^{3k}\le \haus^3_{d_\lambda}(\pi_G(H^\diff)) \le C \lambda^{3k}$; and
\item $C^{-1}r^3 \le \haus^3_{d_\lambda}(B_{d_\lambda}(x,r)) \le C r^3$,  if $B_{d_\lambda}(x,r) \subset \pi_G(X_{k-1}\setminus X_{k+2})$
\end{enumerate}

Observe also that components of $\pi_G(X_\infty)$ are singletons in $(\R^3/G, d_\lambda)$. Thus $\pi_G(X_\infty)$ is $0$-dimensional.
\end{remark}

\begin{lemma}
\label{lemma:measure1}
Let $(\R^3/G,\cX, (\fC,\fA,\fW), \theta, d_\lambda)$ be a Semmes space. Then there exists $C>1$ so that
\begin{equation}
\label{eq:diam1}
C^{-1} \lambda^k \le \diam_{d_\lambda} \pi_G H \le C\lambda^k,
\end{equation}
for every $k\ge 0$ and $H\in \cC(X_k)$.

If, in addition, $\lambda^3 \growth_{\cX}<1$, then $\haus^3_{d_\lambda}(\pi_G(X_\infty)) = 0$ and there exists $C>1$ so that
\begin{equation}
\label{eq:haus1}
C^{-1}\lambda^{3k} \le \haus^3_{d_\lambda}(\pi_G H) \le C \lambda^{3k}
\end{equation}
for every $k\ge 0$ and $H\in \cC(X_k)$.
\end{lemma}

\begin{proof}
Since
\[
\pi_G H = \overline{\bigcup_{i \ge k} \bigcup_{H' \in \cC(X_i\cap H)} \pi_G(H'^\diff)}.
\]
we have, by connectedness and Remark \ref{rmk:mp1}(2),
\[
C^{-1} \lambda^k \le \diam_{d_\lambda} H^\diff \le \diam_{d_\lambda} H \le \sum_{i \ge k} C \lambda^i \le C' \lambda^k,
\]
Similarly, we have that
\[
\haus^3_{d_\lambda}(\pi_G(H)) = \haus^3_{d_\lambda}(\pi_G(X_\infty\cap H)) + \sum_{i\ge k} \sum_{H'\in\cC(X_i\cap H)} \haus^3_{d_\lambda}(\pi_G(H'^\diff)).
\]

Suppose now that $\lambda^3 \growth_{\cX} <1$. Then, by \eqref{eq:diam1},
\[
\haus^3_{d_\lambda}(\pi_G(X_\infty)) \le \limsup_{i\to 0} \sum_{H'\in \cC(X_i)} (C\lambda^i)^3 \le  C^3 \limsup_{i\to 0} \lambda^{3i} \growth_{\cX}^i = 0.
\]
By Definition \ref{def:growth} and Remark \ref{rmk:mp1}(3), there exist $k_0\geq 1$ and $C>1$  so that
\[
C^{-1} \lambda^{3k} \le \sum_{i\ge k} \sum_{H'\in\cC(X_i\cap H)} \haus^3_{d_\lambda}(\pi_G(H'^\diff)) \le \sum_{i\ge k} C \lambda^{3i} \growth_{\cX}^{i-k} \le C \lambda^{3k}
\]
for $k\geq k_0$. After replacing $C$ by another constant that depends only on $C$, $\lambda$, $k_0$, and the upper growth $\overline{\gamma}_{\cX}$, we may obtain \eqref{eq:haus1} for all $k\geq 0$.
This concludes the proof.
\end{proof}

\begin{remark}
\label{rmk:epsilon}
We observe, by Remark \ref{rmk:mp1} and Lemma \ref{lemma:measure1}, that the number
\[
\hat \epsilon_\lambda = \min_{k\ge0}\min_{H\in \cC(X_k)}\left\{\frac{\dist_{d_\lambda}(\pi_G(\partial H),\pi_G(H\setminus H^\diff))}{\lambda^k} \right\}.
\]
is strictly positive. Furthermore, we may fix $\varepsilon_\lambda = \varepsilon_\lambda(d_\lambda) < \hat \varepsilon/10$ so that $N_{d_\lambda}(\pi_G(\partial H),\epsilon_\lambda \lambda^{\lvl(H)})$ is contained in a regular neighborhood of $\pi_G(\partial H)$ in $\pi_G(X_{\lvl(H)-1})\setminus \pi_G(H_{\lvl(H)+1})$.
\end{remark}

\subsection{Ahlfors regularity}

The Ahlfors regularity of Semmes spaces follows as in \cite[Lemma 3.45]{SemmesS:Goomsw}. We discuss the details for completeness of the exposition.

\begin{proposition}
\label{prop:Ahlfors_regularity}
Let $(\R^3/G,\cX, (\fC,\fA,\fW), \theta, d_\lambda)$ be a Semmes space, and suppose that $0< \lambda^3 \growth_{\cX}<1$.
Then the space $(\R^3/G,d_\lambda)$ is  Ahlfors $3$-regular, and spaces $(\R^3/G\times \R^m, d_{\lambda,m})$ are Ahlfors $(3+m)$-regular for $m \ge 1$.
\end{proposition}

\begin{proof}
It suffices to show that $(\R^3/G,d_\lambda)$ is Ahlfors $3$-regular. Then the Ahlfors $(3+m)$-regularity of spaces $(\R^3/G\times \R^m, d_{\lambda,m})$, $m\ge 1$, follows by taking products.

By the bilipschitz invariance of Ahlfors regularity, we may assume that $d_\lambda$ is the metric $d_\theta$ defined by a $\lambda$-modular embedding $\theta\colon \R^3/G\to \R^{16+n}$.
To simplify the exposition, we assume that $X_0 = \bB^3$, and denote by $X_{-j} = B^3(0,\lambda^{-j})$ for $j> 0$.

To show that
\begin{equation}
\label{eq:Ahl3}
C^{-1} r^3 \le \haus^3_{d_\lambda}(B_{d_\lambda}(x,r)) \le Cr^3
\end{equation}
for all balls $B_{d_\lambda}(x,r)$ in $(\R^3/G, d_\lambda)$,
we consider two cases: (a) $x\in \pi_G(\R^3\setminus X_\infty)$, and (b) $x\in \pi_G(X_\infty)$.

We first consider  case (a). Assume that $x\in \pi_G(X_0\setminus X_\infty)$, and suppose $x\in \pi_G(H^\diff)$ and $H\in \cC(X_k)$.
By  Remark \ref{rmk:mp1}(1), there exists a constant $C_1=C_1(d_\lambda)\in (0,1)$ so that if $r\le C_1 \lambda^{k}$ then $B_{d_\lambda}(x,r)\subset \pi_G(X_{k-1}\setminus X_{k+2})$. By \eqref{eq:diam1} of Lemma \ref{lemma:measure1}, there exists a constant $C_2=C_2(d_\lambda)>1$ so that if $r\ge C_2 \lambda^{k}$ then $\pi_G(H) \subset B_{d_\lambda}(x,r)$.

Case (b) follows from (a). Indeed, since $\pi_G(\R^3\setminus X_\infty)$ is dense in $\R^3/G$, given $x\in \pi_G(X_\infty)$ and $r>0$ there exists $y \in \pi_G(\R^3\setminus X_\infty)$ so that $d_\lambda(x,y)<r/2$. So $B_{d_\lambda}(y,r/2)\subset B_{d_\lambda}(x,r) \subset B_{d_\lambda}(y,2r)$, and \eqref{eq:Ahl3} follows by (a).

For $0<r \le C_1 \lambda^{k}$, the claim follows from Remark \ref{rmk:mp1}(4).

For $r\ge C_2 \lambda^{k}$, we fix $m\in \Z$ so that $\lambda^{m+1} \le r < \lambda^m$. Then, by Remark \ref{rmk:mp1}(1), there exist an integer $C_3=C_3(d_\lambda)>0$, and cubes-with-handles $H'\in \cC(X_{m+C_3})$ and $H''\in \cC(X_{m-C_3})$ so that
\[
\pi_G(H') \subset B_{d_\lambda}(x,r) \subset \pi_G(H'').
\]
Then, by   \eqref{eq:haus1} in Lemma \ref{lemma:measure1}, there exists $C=C(d_\lambda)>1$ so that
\[
C^{-1} \lambda^{3(m+C_3)} \le \haus^3_{d_\lambda}(B_{d_\lambda}(x,r)) \le C \lambda^{3(m-C_3)}.
\]

In the remaining subcase $C_1 \lambda^{k} <r<C_2 \lambda^{k}$, $ B_{d_\lambda}(x,r)$ contains the ball $B_{d_\lambda}(x,C_1\lambda^k)$ and is contained in a cube-with-handles in $ \cC(X_{k-C_4})$ for some $C_4=C_4(d_\lambda)>0$. Then \eqref{eq:Ahl3} follows by combining Remark \ref{rmk:mp1} and \eqref{eq:haus1} in Lemma \ref{lemma:measure1}. This concludes the proof.
\end{proof}

\subsection{Linear local contractibility}
\label{sec:llc}

In this section we show that a Semmes space $(\R^3/G,(X_k)_{k\ge 0}, (\fC,\fA,\fW), \theta, d_\lambda)$ is linearly locally contractible if $\cX$ is locally contractible. Recall that a defining sequence $\cX=(X_k)$ is locally contractible if components of $X_{k+1}$ are contractible in $X_k$ for $k\ge 0$.

The linear local contractibility of $(\R^3/G,d_\lambda)$ is a necessary condition for the existence of a quasisymmetric parametrization of $(\R^3/G,d_\lambda)$ by a Euclidean space.

\begin{proposition}
\label{prop:LLC}
Let $(\R^3/G,\cX, (\fC,\fA,\fW), \theta, d_\lambda)$ be a Semmes space having locally contractible defining sequence $\cX$.
Then, for every $m\ge 0$, $(\R^3/G\times \R^m, d_{\lambda,m})$ is linearly locally contractible.
\end{proposition}

As before, we assume as we may that $X_0=\bB^3$ and $X_{-j}=B^3(0,\lambda^{-j})$ for $j> 0$.

\begin{proof}
Since $X_{-k}$ is a $3$-cell for $k\ge 0$,  components of $\pi_G(X_{k+1})$ are contractible in $\pi_G(X_k)$ for every $k\in \Z$  by Lemma \ref{lemma:llc_1}.

\emph{Special case:} Assume that $m=0$ and $y\in \pi_G(\R^3\setminus X_\infty)$; so $y\in H^\diff$ for some $H\in \cC(X_k)$. We consider this case in two parts.

By the uniform quasisimilarity of the modular embedding $\theta$, there exist constants $C_0=C_0(d_\lambda)>1$ and $C_1=C_1(d_\lambda)>0$ with the property: if $0<r<C_1 \lambda^k$ there exists a $3$-cell $E\subset \pi_G(X_{k-1}\setminus X_{k+2})$ satisfying
\[
B_{d_\lambda}(y,r) \subset E\subset B_{d_\lambda}(y,C_0 r).
\]
Hence $B_{d_\lambda}(y,r)$ contracts in $B_{d_\lambda}(y,C_0 r)$ if $r<C_1\lambda^k$.

Suppose now that $r\ge C_1 \lambda^k$. We fix cubes-with-handles $H',H''\in \cC(\cX)$ satisfying $H\subset H'\subset H''$,
\[
\lvl(H')=\min\{\lvl(K)\colon K\in \cC(\cX),\ B_{d_\lambda}(x,r)\subset \pi_G(K))\},
\]
and $\lvl(H'')=\lvl(H')-1$. Then $B_{d_\lambda}(y,r)\subset \pi_G(H')$ and, by Lemma \ref{lemma:llc_1}, $\pi_G(H')$ contracts in $\pi_G(H'')$. By Remark \ref{rmk:mp1} and Lemma \ref{lemma:measure1}, there exists $C_2=C_2(d_\lambda)\ge 1$ so that $\diam_{d_\lambda}(\pi_G(H'')) \le C_2 r$. Thus $\pi_G(H'') \subset B_{d_\lambda}(y, C_2r)$, and $B_{d_\lambda}(y,r)$ contracts in $B_{d_\lambda}(y,C_2r)$. This concludes the proof of this special case.

\emph{General case:} Let $x=(y,v)\in \pi_G(\R^3) \times \R^m$, where $m \ge 0$. Let $r>0$. To show that there exists $C=C(d_{\lambda,m})>1$ so that every ball $B_{d_{\lambda,m}}(x,r)$ is contractible in $B_{d_{\lambda,m}}(x,Cr)$, we consider two cases  (a) $x\in \pi_G(\R^3\setminus X_\infty)\times \R^m$ and (b) $x\in \pi_G(X_\infty)\times \R^m$.

We consider first case (a), that is $x=(y,v) \in \pi_G(\R^3\setminus X_\infty)\times \R^m$ with $y\in \pi_G(H^\diff)$ and $H\in C(X_k)$. Then $B_{d_{\lambda,m}}(x,r)$ contracts in $B_{d_\lambda}(y,Cr) \times (v+[-r,r]^m)$.
Thus $B_{d_{\lambda,m}}(x,r)$ is contractible in $B_{d_{\lambda,m}}(x,(C+\sqrt{m})r)$ and the claim follows.

Case (b) follows from (a). Indeed, there exists $z\in \pi_G(\R^3\setminus X_\infty)\times \R^m$ so that $d_{\lambda,m}(x,z)<r/2$. Hence $B_{d_{\lambda,m}}(x,r)$ is contained in a ball $ B_{d_{\lambda,m}}(z,2r)$ that is  contractible in $B_{d_{\lambda,m}}(z,2Cr) \subset B_{d_{\lambda,m}}(x,4Cr)$, where $C=C(d_{\lambda,m})$ is as in case (a).

\end{proof}

\subsection{Loewner property}
\label{sec:Loewner}

In this section, we briefly list some other analytical properties of Semmes spaces. We refer to \cite{SemmesS:Goomsw}, \cite{SemmesS:Fincgs}, and \cite{HeinonenJ:Quamms} for definitions and background. Assume in what follows that $(\R^3/G,d_\lambda)$ is Ahlfors $3$-regular.

From the proof of the Modular Embedding Theorem (Theorem \ref{thm:FT_embedding}) we see that any pair of points $x,y\in \R^3/G$ is contained in a uniformly bilipschitz image of the Euclidean ball $B^3(|x-y|)$. This property is the same as \cite[Lemma 3.70]{SemmesS:Goomsw} for self-similar spaces. The argument of \cite[Proposition 10.8]{SemmesS:Goomsw} can now be applied almost verbatim to show that $(\R^3/G,d_\lambda)$ supports a $(1,1)$-Poincar\'e inequality as formulated in \cite[(10.9)]{SemmesS:Goomsw}. Since the space $\R^3/G$ is PL outside $\pi_G(X_\infty)$, the Poincar\'e inequality can be formulated in terms of \emph{generalized gradients} (\emph{upper gradients}). We refer to \cite[Appendix C]{SemmesS:Fincgs} for a detailed treatment.

Ahlfors $3$-regularity, quasiconvexity, and the $(1,1)$-Poincar\'e inequality together imply that $(\R^3/G,d_\lambda)$ is a Loewner space in the sense of Heinonen and Koskela; see \cite[Theorem 5.7]{HeinonenJ:Quamms}. A metric measure space $(X,d,\mu)$ of Hausdorff-dimension $Q$ is a \emph{Loewner space} if there exists a function $\phi \colon (0,\infty) \to (0,\infty)$ so that
\[
\Mod_Q(E,F) \ge \phi(\Delta(E,F,X))
\]
whenever $E$ and $F$ are disjoint continua in $X$, where
\[
\Delta(E,F,X) = \frac{\dist(E,F)}{\min\{ \diam E, \diam F\}},
\]
and $\Mod_Q(E,F)$ is the $Q$-modulus of the family of paths connecting $E$ and $F$ in $X$.

Suppose now that the space $(\R^3/G\times \R^m,d_{\lambda,m})$ is Ahlfors $(3+m)$-regular and homeomorphic to $\R^{3+m}$ for some $m\ge 0$. Then $(\R^3/G\times \R^m,d_{\lambda,m})$ supports $(1,1)$-Poincar\'e inequality by a theorem of Semmes for manifolds \cite[Theorem B.10(b)]{SemmesS:Fincgs}. Thus $(\R^3/G\times \R^m,d_{\lambda,m})$ is a Loewner space by the aforementioned theorem of Heinonen and Koskela. Recall that the Loewner property is a necessary condition for the quasisymmetric parametrizability; see Tyson \cite{TysonJ:Quaqmm}.

\subsection{Quasisymmetric equivalence of Semmes metrics}
\label{sec:R3G_qs_equivalence}

In this section we prove the quasisymmetric equivalence of Semmes metrics on $(\R^3/G,\cX)$ associated to different welding structures and scaling factors.

\begin{proposition}
\label{prop:R3G_qs_type}
Let $(\R^3/G,\cX, (\fC_i,\fA_i,\fW_i),\theta_i, d_{\lambda_i})$ be two Semmes spaces with $i=1,2$ and  $\lambda_1,\lambda_2\in (0,1)$. Suppose that $(\fC_1,\fA_1,\fW_1)$ and $(\fC_2,\fA_2,\fW_2)$ have compatible atlases. Then $\id \colon (\R^3/G,d_{\lambda_1}) \to (\R^3/G,d_{\lambda_2})$ is quasisymmetric.
\end{proposition}

\begin{proof}
Assume, as we may, that $X_0=\bB^3$ and define $X_{-j}=B^3(0,\lambda^{-j})$ for $j> 0$.

Since $\pi_G(\R^3\setminus X_\infty)$ is dense in $\R^3/G$ and metrics $d_{\lambda_i}$ are bilipschitz equivalent to modular metrics $d_{\theta_i}$ for $i=1,2$, respectively, it suffices to show that there exists a homeomorphism $\eta \colon [0,\infty) \to [0,\infty)$ so that
\begin{equation}
\label{eq:tau_qs}
\frac{|\theta_2(x)-\theta_2(y)|}{|\theta_2(x)-\theta_2(z)|} \le \eta\left( \frac{|\theta_1(x)-\theta_1(y)|}{|\theta_1(x)-\theta_1(z)|}\right)
\end{equation}
for all distinct points $x$, $y$, and $z$ in $\pi_G(\R^3\setminus X_\infty)$.

We divide the proof to different cases depending on relative distances between points $x$, $y$, and $z$. For brevity, say that points $x$ and $y$ in $\pi_G(\R^3\setminus X_\infty)$ are \emph{near} if there exists $H,H'\in \cC(\cX)$ so that $\{x,y\}\subset \pi_G(H^\diff\cup H'^\diff)$ and the common boundary $H^\diff \cap H'^\diff \ne \emptyset$. Otherwise, we say that points $x$ and $y$ are \emph{far}.

Let $x$, $y$, and $z$ be distinct points in $\pi_G(\R^3\setminus X_\infty)$.

\emph{Case I:} Suppose that there are at least two pairs of points in the set $\{x,y,z\}$ are \emph{near}.

Then there exists $H,H',H''\in \cC(\cX)$ so that $H^\diff \cap H'^\diff \ne \emptyset$, $H'^\diff \cap H''^\diff \ne \emptyset$ and $\{x,y,z\}\subset \pi_G(H^\diff \cup H'^\diff \cup H''^\diff)$. Then, by quasiconvexity of metrics $d_{\lambda_i}$, compatibility of atlases, and modularity of embeddings $\theta_i$, there exists $C_1=C_1(\theta_1,\theta_2)>0$ so that \eqref{eq:tau_qs} holds with $\eta=\eta_1$, where $\eta_1(t)=C_1 t$.

\emph{Case II:} Suppose that all the points $x$, $y$, and $z$ are \emph{far} from each other. Then, by Remark \ref{rmk:comparison}, there exists $C_2=C_2(\theta_1,\theta_2)>0$ so that \eqref{eq:tau_qs} holds with $\eta = \eta_2$, where $\eta_2(t) = C_2 t^p$ and $p=\log \lambda_2 /\log \lambda_1$.

\emph{Case III:} Suppose now that there exists only one pair in $\{x,y,z\}$ where the points are near and that points in the other two pairs are \emph{far}. We have three subcases.

\emph{Case III.1:} Suppose that $y$ and $z$ are near. Then $x$ and $y$ are far and $x$ and $z$ are far. So there exists $C=C(\theta_1,\theta_2)>0$ so that
\[
\frac{1}{C} \le \frac{|\theta_i(x)-\theta_i(y)|}{|\theta_i(x)-\theta_i(z)|} \le C
\]
for $i=1,2$. Thus \eqref{eq:tau_qs} holds with $\eta=\eta_3$, where $\eta_3(t) = C_3t$ with $C_3=C_3(\theta_1,\theta_2)>0$.

\emph{Case III.2:} Suppose now that $x$ and $z$ are \emph{near} and let $H,H'\in \cC(\cX)$ be such that $\{x,z\}\subset \pi_G(H^\diff\cup H'^\diff)$ and $H^\diff \cap H'^\diff\ne \emptyset$. Then, by modularity of embeddings $\theta_1$ and $\theta_2$, there exist $C=C(\theta_1,\theta_2)>1$ and $w\in \pi_G(H^\diff\cup H'^\diff)$ so that
\[
\min\{ |\theta_i(x)-\theta_i(w)|, |\theta_i(z)-\theta_i(w)|\} \ge \frac{1}{C} \diam \theta_i(\pi_G(H^\diff\cup H'^\diff))
\]
for $i=1,2$.

Following the argument for cases I and II, there exists $C_4=C_4(\theta_1,\theta_2)>0$ so that
\[
\frac{|\theta_2(x)-\theta_2(w)|}{|\theta_2(x)-\theta_2(z)|} \le C_1\eta_1\left( \frac{|\theta_1(x)-\theta_1(w)|}{|\theta_1(x)-\theta_1(z)|}\right)
\]
and
\[
\frac{|\theta_2(x)-\theta_2(y)|}{|\theta_2(x)-\theta_2(w)|} \le C_4 \eta_2\left( \frac{|\theta_1(x)-\theta_1(y)|}{|\theta_1(x)-\theta_1(w)|}\right)
\]
where homeomorphisms $\eta_1$ and $\eta_2$ are as in cases I and II.

Thus
\begin{eqnarray*}
\frac{|\theta_2(x)-\theta_2(y)|}{|\theta_2(x)-\theta_2(z)|} &=& \frac{|\theta_2(x)-\theta_2(y)|}{|\theta_2(x)-\theta_2(w)|} \frac{|\theta_2(x)-\theta_2(w)|}{|\theta_2(x)-\theta_2(z)|}\\
&\le& C_1C_4 \eta_2\left( \frac{|\theta_1(x)-\theta_1(y)|}{|\theta_1(x)-\theta_1(w)|}\right) \eta_1\left( \frac{|\theta_1(x)-\theta_1(w)|}{|\theta_1(x)-\theta_1(z)|}\right)\\
&\le& C_1C_4 \eta_2\left( C \frac{|\theta_1(x)-\theta_1(y)|}{|\theta_1(x)-\theta_1(z)|}\right) \eta_1\left( C \frac{|\theta_1(x)-\theta_1(w)|}{|\theta_1(x)-\theta_1(z)|}\right),
\end{eqnarray*}
Thus \eqref{eq:tau_qs} holds with $\eta=\eta_3$, where $\eta_3(t) = C_1C_4\eta_1(Ct)\eta_2(Ct)$.

\emph{Case III.3.} The remaining case is that $x$ and $y$ are \emph{near}. Let $H,H'\in \cC(\cX)$ be such that $\{x,y\}\subset \pi_G(H^\diff \cup H'^\diff)$ and $H^\diff \cap H'^\diff \ne \emptyset$. As in Case III.2, there exist $C=C(\theta_1,\theta_2)>1$ and $w\in \pi_G(H^\diff\cup H'^\diff)$ so that
\begin{equation}
\label{eq:w_2}
\min\{ |\theta_i(x)-\theta_i(w)|, |\theta_i(y)-\theta_i(w)|\} \ge \frac{1}{C} \diam \theta_i(\pi_G(H^\diff\cup H'^\diff)).
\end{equation}
Furthermore, there exists $C_5 = C_5(\theta_1,\theta_2)>0$ so that
\[
\frac{|\theta_2(x)-\theta_2(y)|}{|\theta_2(x)-\theta_2(z)|} \le C_5 \eta_1\left( \frac{|\theta_1(x)-\theta_1(y)|}{|\theta_1(x)-\theta_1(w)|}\right) \eta_2\left( \frac{|\theta_1(x)-\theta_1(w)|}{|\theta_1(x)-\theta_1(z)|}\right).
\]

By \eqref{eq:w_2} and assumptions on $\{x,y,z\}$, we have that
\[
\max\left \{
\frac{|\theta_1(x)-\theta_1(y)|}{|\theta_1(x)-\theta_1(z)|}, \frac{|\theta_1(x)-\theta_1(y)|}{|\theta_1(x)-\theta_1(w)|}, \frac{|\theta_1(x)-\theta_1(w)|}{|\theta_1(x)-\theta_1(z)|}\right\} \le C'
\]
where $C'=C'(\theta_1,\theta_2)$. Assume first that
\[
\frac{|\theta_1(x)-\theta_1(y)|}{|\theta_1(x)-\theta_1(w)|} \le \left( \frac{|\theta_1(x)-\theta_1(y)|}{|\theta_1(x)-\theta_1(z)|}\right)^{1/2}.
\]
Then
\[
\frac{|\theta_2(x)-\theta_2(y)|}{|\theta_2(x)-\theta_2(z)|} \le C_5 \eta_1\left( \frac{|\theta_1(x)-\theta_1(y)|}{|\theta_1(x)-\theta_1(z)|}\right)^{1/2} \eta_2( C' ).
\]
The case
\[
\frac{|\theta_1(x)-\theta_1(w)|}{|\theta_1(x)-\theta_1(z)|} \le \left( \frac{|\theta_1(x)-\theta_1(y)|}{|\theta_1(x)-\theta_1(z)|}\right)^{1/2}
\]
is similar. So \eqref{eq:tau_qs} holds with $\eta(t) = C_5 \max\{ \eta_1(t^{1/2})\eta_2(C'), \eta_1(C') \eta_2(t^{1/2})\}$. This concludes case III.2 and the proof.
\end{proof}

\section{A sufficient condition for quasisymmetric parametrization}
\label{sec:positive}

In this section we consider the existence of Semmes metric $d_\lambda$ on $\R^3/G$ such that  $(\R^3/G,d_\lambda)$ is quasisymmetrically equivalent to $\R^3$. A sufficient condition for the parametrizability is the existence of a \emph{flat welding structure}.

\begin{definition}
We say that $(\fC,\fA,\fW)$ is \emph{a flat welding structure} if $\fC$ is finite, condensers $\fC$ are in $\R^3$, and weldings $\fW$ are similarities.
\end{definition}

The existence of a flat welding structure leads to a modular embedding $\theta$ of the decomposition space $\R^3/G$ into $\R^4$, which in turn shows a strong form of the  quasisymmetric parametrizability of $\theta(\R^3/G)$, in particular, $\theta(\R^3/G)$ is a $3$-dimensional quasiplane in $\R^4$.

\begin{theorem}
\label{thm:F}
Suppose that $(\R^3/G,\cX) $ is a decomposition space of finite type whose defining sequence $\cX$ has a flat welding structure $(\fC,\fA,\fW)$. Suppose also that components of $X_{k+1}$ are contractible in $X_k$ for every $k\ge 0$. Then there exists $\lambda_0\in (0,1)$ depending on $(\fC,\fA,\fW)$ for the following. For each $\lambda \in (0,\lambda_0)$, there is a $\lambda$-modular embedding $\theta \colon \R^3/G \to \R^4$ with respect to this structure, so that the embedded set $\theta(\R^3/G)\subset \R^4$
is Ahlfors $3$-regular, linearly locally contractible, and quasisymmetric to $\R^3$. Furthermore, there exists a quasisymmetric map $f\colon \R^4\to \R^4$ so that $f(\R^3)= \theta(\R^3/G)$.
\end{theorem}

It is easy to see that every flat welding structure on a defining sequence of finite type induces a rigid welding structure with a compatible atlas, in the sense of \eqref{eq:compatible}. The converse is not always true; an obvious criterion can be given as follows.

\begin{lemma}\label{lemma:compatible}
Let $\cX$ be a defining sequence of finite type and $(\fC,\fA,\fW)$ a rigid welding structure on $\cX$ in $\R^3$. Suppose that for every $\cxi = (A,B)\in \fC$ there exists an PL embedding $h_\cxi \colon A \to \R^3$ so that $h_\cxi|\partial A$ and $h_\cxi|D$ are (Euclidean) similarities for each component $D$ of $B$. Then $\cX$ admits a flat welding structure.
\end{lemma}

All but the last claim in Theorem \ref{thm:F} can be proved by appealing to a rigid welding structure
$(\fC',\fA',\fW')$ compatible to the given flat structure $(\fC,\fA,\fW)$. However, in order to extend the quasisymmetric map $\R^3 \to \theta(\R^3/G)$ to a quasisymmetric homeomorphism of $\R^4$, we will need to repack the condensers in
$(\fC,\fA,\fW)$.  The idea of repacking is adapted from Semmes' excellent packages for the self-similar decomposition spaces (\cite[Definition 3.2]{SemmesS:Goomsw}).

Let $d_\theta$ be a $\lambda$-modular metric induced by an embedding $\theta \colon \R^3/G \to \R^4$ as in Theorem \ref{thm:F}. By Lemma \ref{lemma:compatible}, the path metric $\hat d_\theta$ associated to $d_\theta$  is bilipschitz equivalent to the path metric associated to the Semmes metric (with the same scaling $\lambda$) derived from a compatible  rigid structure $(\fC',\fA',\fW')$. Thus $d_\theta$ is a Semmes metric and $(\R^3/G, \cX,(\fC,\fA,\fW), \theta,d_\theta)$ a Semmes space; we write $d_\lambda$ for $d_\theta$.

In view of Theorem \ref{thm:F}, there exist defining sequences which do not admit flat welding structures. Indeed, by Theorem \ref{thm:F}, the existence of a flat welding structure yields quasisymmetric parametrizability. Thus, for example, the standard defining sequences associated to the Whitehead continuum and to the Bing double do not admit flat welding structures.

\subsection{Unlinking and repacking}
As a preliminary step for the proof of Theorem \ref{thm:F}, we discuss homeomorphisms of $\R^4$ that unlink and repack condensers in $\R^3$.

Let $K\subset \R^3$ be a cube-with-handles. We define
\[
K^* = K \times [-\diam K, \diam K]\subset \R^4,
\]
where $\diam K$ is the Euclidean diameter of $K$. If $B$ is a pair-wise disjoint union of cubes-with-handles, we set $B^* = \bigcup_{K \in \cC(B)} K^*$. Suppose $(A,B)$ is a condenser in $\R^3$, we will also call \emph{$(A^*,B^*)$ a ($4$-dimensional) condenser}.

Let $\cxi=(A,B)$ be a condenser in $\R^3$ with $\diam A=1$, and $\lambda\in (0,1)$. We say a PL-embedding $p_\cxi \colon (\R^3\setminus A) \cup B\to \R^3$ is a
\emph{$\lambda$-repacking} of $\cxi$ if there are pair-wise disjoint Euclidean balls $\{b_D \subset \interior A \colon D\in \cC(B)\}$ such that
\begin{enumerate}
\item [(i)] $p_\cxi|\R^3\setminus A = \id$,
\item [(ii)] $p_\cxi|D$ is an orientation preserving similarity, and
\item [(iii)] $p_\cxi(D)\subset \interior \, b_D$ and $\diam p_\cxi(D)=\lambda$,
\end{enumerate}
for each component $D$ of $B$.

Let $\cxi$ be a condenser in $\R^3$ with $\diam A=1$. We denote by $\lambda_\cxi$ the supremum of $\lambda>0$ so that $\cxi$ admits a $\lambda$-repacking. Note that $\lambda_\cxi>0$, since repackings exist for all sufficiently small $\lambda>0$.

Let $\fC$ be a finite collection of condensers in $\R^3$, with  $\diam A=1$ for every $\cxi=(A,B)\in \fC$.
We denote by $\lambda_\fC$ the supremum of $\lambda>0$ so that every $\cxi\in \fC$ admits a $\lambda$-repacking. We call $\lambda_{\fC}$ the \emph{repacking constant of $\fC$}.

\begin{definition}
Let $\cxi=(A,B)$ be a condenser in $\R^3$ with $\diam A=1$. We say that a PL-homeomorphism $P_\cxi \colon \R^4 \to \R^4$ is a \emph{*-stable $\lambda$-repacking of $\cxi$ in $\R^3$ $($or of condenser $(A^*,B^*)$ $)$} if
\begin{enumerate}
\item $P_\cxi|(\R^3\setminus A) \cup B$ is a $\lambda$-repacking of $\cxi$,
\item  $P_\cxi|\R^4\setminus A^* = \id$,
\item $P_\cxi|D^*$ is an orientation preserving similarity for each component $D$ of $B$, in particular
\item $P_\cxi(B^*) = P_\cxi(B)^*$.
\end{enumerate}
\end{definition}

\begin{lemma}
\label{lemma:repacking}
Let $\cxi=(A,B)$ be a condenser in $\R^3$ with $\diam A=1$. If the components of $B$ are contractible in $A$, then, for every $0< \lambda < \lambda_{\cxi}$, there exists a *-stable $\lambda$-repacking $P_\cxi\colon \R^4 \to \R^4$ of $\cxi$.
\end{lemma}

\begin{proof}
Let $p_\cxi \colon (\R^3\setminus A) \cup B\to \R^3$ be a $\lambda$-repacking of the condenser $\cxi=(A,B)$. Fix $d \in (0,1)$ so that
\[
B^* \cup (p_\cxi(B))^* \subset \interior (A\times [-d,d]),
\]
and set $I=[-d,d]$.

As a preliminary step, we construct for every $D\in \cC(B)$ a PL-homeo\-morphism $f_D \colon \R^4\to \R^4$ with the properties that $f_D|\R^4\setminus (A\times I) = \id$, and that $f_D|D^*$ is an orientation preserving  similarity satisfying $f_D|D = p_\cxi|D$  and $f_D(D^*) = p_\cxi(D)^*$.

Given $D\in \cC(B)$, let $b_D=B^3(x_D,r_D)\subset A$ be the Euclidean ball containing $p_\cxi(D)$ as in (iii). Since $p_\cxi|D$ is a similarity, it extends to a similarity $p_D \colon \R^3\to \R^3$ with a scaling constant $\rho_D$.
Denote again by $p_D \colon \R^4\to \R^4$ the extension $(x,t) \mapsto (p_D(x),\rho_D t)$.

Fix a core $\mathcal R_D$ of $D$. Since $D^*$ is a regular neighborhood of $\mathcal R_D$ and $\mathcal R_D$ is contractible in $A$, there exist, by the Penrose--Whitehead--Zeeman lemma (see Section \ref{sec:top_lc}), PL $4$-cells $E_D$ and $E'_D$ in $A\times I$ so that
\[
D^* \subset \interior E_D \subset E_D\subset \interior E'_D \subset E'_D \subset \interior (A\times I).
\]
We fix $z_D\in \interior E_D$ and $\varepsilon_D >0$  so that $B^4(z_D,2 \varepsilon_D) \subset E_D$, and choose a number $ r_D' \in (r_D, \lambda_\cxi)$. Thus
\[
p_D(D^*) \subset b_D^* \subset B^3(x_D,r_D')^* \subset \interior (A\times I).
\]

By standard isotopy results, we may fix two PL self-homeomorphisms $h_D$ and  $\tau_D$ of $\R^4$ with the following properties. Since $E_D$ and $E'_D$ are $4$-cells, there exists an orientation preserving PL homeomorphism $h_D\colon \R^4\to \R^4$ so that  $h_D|\R^4\setminus (A \times I)= \id$, $h_D(E'_D)=B^4(z_D,2\varepsilon_D)$, $h_D(E_D)=B^4(z_D,\varepsilon_D)$ and  $h_D(D^*)\subset B^4(z_D,\varepsilon_D/2)$, and that $h_D|D^*$ is a scaling followed by a translation.
Furthermore, there exists a PL-homeomorphism $\tau_D\colon \R^4\to \R^4$ such that
$\tau_D|\R^4 \setminus (A\times I) = \id$,
$\tau_D(B^4(z_D,2\varepsilon_D))= B^3(x_D,r'_D)^*$ and $\tau_D(B^4(z_D,\varepsilon_D))= b_D^*$, and that
$\tau_D|h_D(D^*)$ is an orientation preserving similarity which maps $h_D(D^*)$ onto $p_D(D^*)$.

Therefore, the PL homeomorphism $f_D=\tau_D\circ h_D$ satisfies $f_D|\R^4\setminus (A\times I) = \id$ and $f_D|D^*$ is the similarity $p_D|D^*$.

We will combine the homeomorphisms $f_D, D\in \cC(B)$ defined above as follows.
First, components of $B$ are raised to different levels in $\R^3\times (d,1)\subset \R^4$ by a homeomorphism $g_1$ of $\R^4$. Next, one component at a time, each $D\in \cC(B)$ is lowered to $\{x_4=0\}$ upon which  the homeomorphism $f_D$ may be applied and then the image $p_D(D)$ is raised to the previous level. The composition of these maps is a homeomorphism $g_2$ of $\R^4$. Finally, all raised $p_D(D)$ are descended to $\{x_4=0\}$  by a homeomorphism $g_3^{-1}$ of $\R^4$. The *-stable $\lambda$-repacking of $\cxi$ is defined by
\[
P_\cxi=g_3^{-1}\circ g_2 \circ g_1.
\]
We now give the details.

For every $D\in \cC(B)$, fix $d_D\in (d,1)$ so that $d_D\ne d_{D'}$ for different components $D$ and $D'$ in $\cC(B)$;  fix also $\delta>0$ so that intervals $[d_D-\delta,d_D+\delta]$ are pair-wise disjoint and contained in $(d,1)$. Let $\rho = \delta/(4d)$, and  $J_D = [d_D-\delta/4,d_D+\delta/4]$ for every $D\in \cC(B)$.
We  fix a PL-homeomorphism $g_1\colon \R^3\times \R \to \R^3\times \R$ so that $g_1|\R^4\setminus A^* = \id$ and $g_1(x,t)=(x, \rho t + d_D)$ for $(x,t)\in D\times I$ and $D\in \cC(B )$.
In particular,
\[
g_1(D^*) \subset g_1(D\times I) = D\times J_D
\]
for every $D \in \cC(B)$.

The homeomorphism $g_3$ is defined similarly as $g_1$, with $(A,\bigcup_{D\in \cC(B)} p_D(D))$ in place of $(A,B)$ and with $d_{p_D(D)} = d_D$, so that the PL-homeomorphism $g_3\colon \R^4\to \R^4$ satisfies $g_3|\R^4\setminus A^*=\id$ and
\[
g_3(p_D(D)^*) \subset g_3(p_D(D)\times I) \subset p_D(D)\times J_D
\]
for $D\in \cC(B)$.

Having $g_1$ and $g_3$ at our disposal, we construct a PL-homeomorphism $g_2$ as follows. For every $D\in \cC(B)$, let $\zeta_D\colon \R\to \R$, be a piece-wise linear increasing function so that $\zeta_D(t) = \rho t + d_D$ for $t\in I$, and $\zeta_D(t) = t$ for $|t|>1$. Let also $\xi_D \colon \R^4 \to \R^4$ be the PL map $(x,t)\mapsto (x,\zeta_D(t))$. Then $\xi_D|D^* = g_1|D^*$ and $\xi_D|p_D(D)^* = g_3|p_D(D)^*$ for every $D\in \cC(B)$.

Since $f_D|\R^4\setminus (A\times I) = \id$, we have
\[
\xi_D \circ f_D \circ \xi_D^{-1}|(\R^4\setminus ( A\times (d_D-\delta/4, d_D+\delta/4))) = \id
\]
for every $D\in \cC(B)$. Thus the mapping $g_2 \colon \R^4 \to \R^4$ defined by taking the composition (in any fixed order) of $\xi_D\circ f_D \circ \xi^{-1}_D$ for all $D\in \cC(B)$, is a well-defined PL-homeomorphism satisfying $g_2|\R^4\setminus A^* = \id$ . Moreover,
\[
g_3^{-1} \circ g_2 \circ g_1|D^* = (g_3^{-1} \circ \xi_D) \circ f_D \circ (\xi_D^{-1} \circ g_1)|D^* = f_D|D^* = p_D|D^*
\]
is a similarity. Since $p_\cxi|D=p_D|D$, $P_\cxi = g_3^{-1} \circ g_2 \circ g_1$ is a *-stable repacking of $\cxi$.
\end{proof}

\subsection{Proof of Theorem \ref{thm:F}}
Let $(\fC,\fA,\fW)$ be a flat welding structure on the defining sequence $\cX$, and assume that  $\diam A=1$ for all  $\cxi=(A,B)\in \fC$. We also assume that $X_0 =H_0= A_{H_0}$ where $\cxi_{H_0}=(A_{H_0}, B_{H_0})$ is the condenser associated to $H_0$, and that  the corresponding chart satisfies  $\varphi_{H_0}|\partial H_0=\id$.

We enumerate cubes-with-handles in $\cC(\cX)$ by $H_0, H_1,\ldots$ so that if $H_j\in \cC(X_k)$ then $H_{j+1}\in \cC(X_k)\cup \cC(X_{k+1})$. This enumeration provides a natural ordering for condensers, charts and the weldings as well. Denote by $\cxi_j= (A_j,B_j)= \cxi_{H_j}$ for condensers in $\fC$ and by $\varphi_j= \varphi_{H_j}\colon H_j^\diff \to \cxi_j^\diff$ the charts in $\fA$ for $j\ge 0$.

Let $k_j = \lvl(H_j)$, and  $q(j)$ be the index of the parent of $H_j$, that is, $\lvl(H_{q(j)})=k_j -1$ and $H_j\in  \cC(H_{q(j)}\cap X_{k_j})$.

Let $w_j=\varphi_{q(j)}\circ \varphi_j^{-1}$ be the welding of $(A_j,B_j)$ to its parent $(A_{q(j)},B_{q(j)})$, for $j\ge 1$. Since $w_j$ is a similarity and $B_{q(j)} \subset \R^3$, $w_j(A_j)$ is a component of $B_{q(j)}$. We extend $w_j$ to a similarity $w_j \colon \R^4\to \R^4$ by $(x,t) \mapsto (w_j(x),\lambda_j t)$, where $\lambda_j$ is the scaling factor of $w_j$. We call the extended $w_j$ a welding of $(A_j^*,B_j^*)$ to $(A_{q(j)}^*,B_{q(j)}^*)$, and note that $w_j(A_j^*)=w_j(A_j)^*$ is a component of  $B_{q(j)}^*$.

We construct now cumulative welding maps and repacked cumulative welding maps. We define the \emph{cumulative welding maps} $\hat w_j$ by $\hat w_0 = \id$ and
\[
\hat w_j = \hat w_{q(j)} \circ w_j
\]
for $j\ge 1$.

Since $w_j|\partial A_j = \varphi_{q(j)}\circ \varphi_j^{-1}|\partial A_j$, we have
\begin{equation}
\label{eq:hat_w_i_check}
\begin{split}
\hat w_j \circ \varphi_j|\partial H_j &= \hat w_{q(j)} \circ w_j \circ \varphi_j|\partial H_j \\
&= \hat w_{q(j)} \circ \varphi_{q(j)}|\partial H_j
\end{split}
\end{equation}
for $j\ge 1$, and $\hat w_0 \circ \varphi_0|\partial H_0 = \id$.

Since $w_j$ is a similarity, $\hat w_j(A_j)$ is a component of $\hat w_{q(j)}(B_{q(j)})$ and
\[
\hat w_j(A_j^*) = \hat w_{q(j)}(w_j(A_j^*)) \subset \hat w_{q(j)}(B^*_{q(j)}).
\]
It follows by induction that the images $\hat w_j(A_j^*\setminus B_j^*)$ are pair-wise disjoint for $j\ge 0$. Then
\begin{equation}
\label{eq:hat_F}
\R^4\setminus \hat F=(\R^4\setminus X_0^*)\cup \bigcup_{j=1}^\infty \hat w_j(A_j^*\setminus B_j^*)
\end{equation}
is a disjoint union, where
\[
\hat F=\bigcap_{k\ge 0}\left( \bigcup \{\hat w_j(A_j) \colon (A_j,B_j)\in \fC, \varphi_j^{-1}(A_j)=H_j \in \cC(X_k)\}\right).
\]
Since $\diam \hat w_j(A_j) \to 0$ as $j \to \infty$, the components of $\hat F$ are points.

We define now \emph{repacked cumulative welding maps $\tilde w_j$}. Let $0<\lambda < \lambda_{\fC}$. We show first that components of $B_j$ are contractible in $A_j$. Let $D\in \cC(B_j)$. Since $\varphi_j^{-1}(\partial D)$ is a boundary of a component of $H_j \cap X_{\lvl(H_j)+1}$, $\varphi_j^{-1}(\partial D)$ is contractible in $H_j$. Thus $\partial D$ is contractible in $A_j$. Let $\mathcal R_D$ be a core of $D$ that is contained in a collar $\Omega_D$ of $\partial D$ in $D$. Since $\Omega_D$ retracts to $\partial D$, $\mathcal R_D$ is contractible in $A_j$. Thus $D$ is contractible in $A_j$.

Using Lemma \ref{lemma:repacking}, we fix a collection of *-stable $\lambda$-repackings $\{P_\cxi\colon \R^4 \to \R^4 \colon \cxi\in \fC\}$. For simplicity, denote the *-stable repacking for $\cxi_j=(A_j,B_j)$  by $P_j = P_{\cxi_j}$ for $j\ge 0$; note that there are only finitely many distinct mappings in $\{P_j\colon j\geq 0\}$.

Associated to the welding maps $w_j \colon \R^4 \to \R^4$ for $j\ge 1$, and the *-stable repackings $P_j$ for $j\ge 0$, we define $\tilde w_j\colon \R^4\to \R^4$ by
\[
\tilde w_j = \tilde w_{q(j)} \circ w_j \circ P_j
\]
for $j\ge 1$ and set $\tilde w_0 = P_0$.

Since the *-stable repacking $P_j$ is a similarity on $D^*$ for each $D\in \cC(B_j)$ and $P_j(A_j^*)= A_j^*$, we know that $w_j \circ P_j(A_j^*)\subset B_{q(j)}^*$ and that $\tilde w_j|D^*$ is a similarity for every $D\in \cC(B_j)$. Therefore  $\tilde w_{q(j)} \circ w_j|A_j^*$ is a similarity, and  $\tilde w_j|A_j^*\setminus B_j^*$ is a composition of an $L$-bilipschitz map $P_j$ with a similarity for every $j\ge 0$.
Indeed, the mapping $\tilde w_j|A_j$ is $L \lambda^j$-Lipschitz for every $j\ge 0$, where $L$ is the maximum of the Lipschitz constants of *-stable repackings $\{P_\cxi \colon \cxi\in \fC\}$.

Since $P_j|\partial A_j = \id$, we have, as in \eqref{eq:hat_w_i_check},
\begin{equation}
\label{eq:i_q(i)}
\tilde w_j \circ \varphi_j|\partial H_j = \tilde w_{q(j)} \circ \varphi_{q(j)}|\partial H_j
\end{equation}
for $j\ge 1$, and  $\tilde w_0 \circ \varphi_0|\partial H_0 = P_0 \circ \varphi_0|\partial H_0 = \id$.

Since $\tilde w_j(A_j)$ is a component of $\tilde w_{q(j)}(B_{q(j)})$ for each $j\ge 1$ and the image sets $\tilde w_j(A_j^*\setminus B_j^*)$ are pair-wise disjoint for $j\ge 0$, we obtain a disjoint union
\begin{equation}
\label{eq:tilde_F}
\R^4\setminus \tilde F=(\R^4\setminus X_0^*)\cup \bigcup_{j=1}^\infty \tilde w_j(A_j^*\setminus B_j^*),
\end{equation}
with
\[
\tilde F=\bigcap_{k\ge 0}\left( \bigcup \{\tilde w_j(A_j) \colon (A_j,B_j)\in \fC, \varphi_j^{-1}(A_j)=H_j \in \cC(X_k)\}\right).
\]
As in the case of $\hat F$, the set $\tilde F$ is totally disconnected.

Having \eqref{eq:i_q(i)} and \eqref{eq:tilde_F} at our disposal, we define an embedding $\theta_\infty \colon \R^3\setminus X_\infty \to \R^4$ by $\theta_\infty|\R^3\setminus X_0 = \id$ and $\theta_\infty|H_j^\diff = \tilde w_j \circ \varphi_j$ for $j \ge 1$. Furthermore, $\theta_\infty$ descends (and then extends) to an embedding $\theta \colon \R^3/G \to \R^4$ so that $\theta(\pi_G(X_\infty)) = \tilde F$. The $\lambda$-modularity of $\theta$ with respect to $(\fC,\fA,\fW)$ follows directly from the uniform quasisimilarity of cumulative repacked welding maps $\tilde w_j$. The space $(\R^3/G,\cX,(\fC,\fA,\fW),\theta,d_\theta)$ is linearly locally connected, and Ahlfors $3$-regular for sufficiently small $\lambda$.

It remains to construct a quasisymmetric map $f\colon \R^4 \to \R^4$ so that $f(\R^3)=\theta(\R^3/G)$. Since $P_j|\partial A_j^* = \id$, we have
\[
\tilde w_j \circ \hat w_j^{-1}|\hat w_j(\partial  A_j^*) = (\tilde w_{q(j)} \circ w_j \circ P_j) \circ (w_j^{-1} \circ \hat w_{q(j)}^{-1})|\hat w_j(\partial A_j^*) = \tilde w_{q(j)} \circ \hat w_{q(j)}^{-1}|\hat w_j(\partial A_j^*)
\]
for every $j\ge 1$. Thus the map $f_\infty \colon \R^4\setminus \hat F \to \R^4\setminus \tilde F$, defined by
\[
f_\infty|\hat w_j(A_j^*\setminus B_j^*) = \tilde w_j \circ \hat w_j^{-1}|\hat w_j(A_j^*\setminus B_j^*)
\]
for $j\geq 1$ and $f_\infty|\R^4\setminus A_0^* = \id$, is a well-defined homeomorphism. Since $\hat F$ and $\tilde F$ are totally disconnected, $f_\infty$ extends to a homeomorphism $f\colon \R^4\to \R^4$.

Since $f|\R^3\setminus A_0 = \theta_\infty|\R^3\setminus X_0$ and
\[
f \circ \hat w_j \circ \varphi_j|X_j^\diff = \tilde w_j \circ \hat w_j^{-1} \circ \hat w_j \circ \varphi_j|X_j^\diff = \tilde w_j \circ \varphi_j|X_j^\diff = \theta_\infty|X_j^\diff
\]
for every $j\ge 0$, we have
\[
f(\R^3) = \theta(\R^3/G)
\]
by continuity.

Since $\tilde w_j \circ \hat w_j^{-1}|\hat w_j(A_j^*)$ is a $(L,\mu_j)$-quasisimilarity for some $\mu_j>0$ and for every $j\ge 0$, the homeomorphism $f_\infty \colon \R^4\setminus \hat F \to \R^4\setminus \tilde F$ is quasiconformal. Moreover, the homeomorphisms $f_j \colon \R^4\to\R^4$, defined by
\[
f_j|\R^4\setminus \hat w_j(A^*_j) = f_\infty|\R^4\setminus \hat w_j(A_j^*)
\]
and
\[
f_j|\hat w_j(A^*_j) = \tilde w_j \circ \hat w_j^{-1}|\hat w_j(A_j^*),
\]
are uniformly quasiconformal. Therefore there exists a homeomorphism $\eta\colon [0,\infty) \to [0,\infty)$ so that homeomorphisms $f_j$ are $\eta$-quasisymmetric. Since $f = \lim_{j\to \infty} f_j$, $f$ is $\eta$-quasisymmetric. This completes the proof of Theorem \ref{thm:F}.

\bigskip

\begin{remark}
\label{rmk:quasispheres}
The quasisymmetric map $f\colon \R^4 \to \R^4$ in Theorem \ref{thm:F} can be taken to be identity in a neighborhood of the infinity, that is, there exists $R>0$ so that $f|\R^4\setminus B^4(R) = \id$. Thus the quasisymmetric map $f\colon \R^4\to \R^4$ extends naturally to a quasiconformal map $f\colon \bS^4\to\bS^4$ and $f(\bS^3)$ is the one point compactification of $f(\R^3)$. Thus
the embedding $\theta \colon \R^3/G \to \R^4$ extends to an embedding $\bS^3/G \to \bS^4$. So $\theta(\bS^3/G)$ is a \emph{quasisphere}, that is, $\theta(\bS^3/G) = f(\bS^3)$, where $f\colon \bS^4\to\bS^4$ is a quasiconformal map.

In view of Theorem \ref{thm:F}, geometrically different quasispheres built this way exist in abundance.

\end{remark}

\section{Circulation}
\label{sec:meridians}
In this section we introduce the notion of \emph{circulation} of a union of cubes-with-handles  based on longitudes and meridians. This concept of circulation will be used in estimating conformal modulus of surface families.

\subsection{Meridians and longitudes}
\label{sec:meri_and_longi}

Recall that a simple closed curve $\bS^1 \to \partial \bB^2\times \bS^1$ on the boundary of a  torus $\bB^2\times \bS^1$ is called a \emph{meridian of $\bB^2\times \bS^1$} if it is homotopic to the loop $e^{i\theta} \mapsto (e^{i\theta},1)$, on $\partial \bB^2\times \bS^1$. In particular, a meridian is contractible in $\bB^2\times \bS^1$ but not in $\partial \bB^2\times \bS^1$.

A non-contractible loop in the solid torus $\bB^2\times \bS^1$ is called a \emph{longitude of } $\bB^2\times \bS^1$. Longitudes are non-trivially linked with all meridians, that is, given a longitude $\sigma$ and a meridian $\alpha$ of $\bB^2\times \bS^1$ then $\sigma (\bS^1)\cap \phi (\bB^2) \ne \emptyset$ for every $\phi\colon \bB^2 \to \bB^2\times \bS^1$ satisfying $\phi|\partial \bB^2 = \alpha$.

Let $X$ be a disjoint union of cubes-with-handles.
We call a simple closed PL curve $\alpha \colon \bS^1 \to \partial X$ a \emph{meridian of $X$} if $[\alpha] \ne 0$ in $\pi_1(\partial X)$ and $[\alpha]=0$ in $\pi_1(X)$; that is, $\alpha$ is not contractible on $\partial X$ but there exists a map $\phi\colon \bB^2 \to X$ so that $\phi|\partial \bB^2 = \alpha$.

Suppose $\alpha \colon \bS^1 \to \partial X$ is a meridian of $X$. Departing slightly from the standard notion of mapping of pairs $(C,D)\to (E,F)$, we denote by $\phi \colon (\bB^2,\partial \bB^2) \to (X,\alpha)$ a mapping $\phi\colon \bB^2 \to X$ that satisfies $\phi|\partial \bB^2 = \alpha$. Let
\[
\cE(X,\alpha) = \text{the collection of all maps}\ \phi \colon (\bB^2,\partial \bB^2) \to (X,\alpha).
\]

Let $\sigma= \sigma_1 + \cdots + \sigma_k$ be a PL $1$-chain in a union of cubes-with-handles $X$, where
$\sigma_i \colon \bS^1 \to X$ are PL maps for $i=1,\ldots, k$ ; and denote $|\sigma| =\bigcup_{i=1}^k \sigma_i(\bS^1)$ its carrier.
We say that $\sigma$ is a \emph{longitude in $X$} if $|\sigma| \cap \phi(\bB^2)  \ne \emptyset$ for all $\phi\colon (\bB^2,\partial \bB^2) \to (X,\alpha)$ and all meridians $\alpha$ of $X$. Heuristically, longitudes are the $1$-cycles in $X$ that are linked with all meridians of $X$. We denote by
\begin{equation}\label{eq:longi1}
\longi (X)= {\text{the family of all longitudes of}}\, X.
\end{equation}
Suppose that $H_1,\ldots ,H_d$ are  pair-wise disjoint cubes-with-handles, then
\begin{equation}\label{eq:longi-sum}
\longi\left(\bigcup_{i=1}^d H_i\right)=\{ \sigma_1+\ldots +\sigma_d \colon \sigma_i\in \longi(H_i),1\le i\le d \}.
\end{equation}

\subsection{Circulation with respect to meridians}
\label{sec:circulation}

Let $H$ be a cube-with-handles, $X$ a finite union of cubes-with-handles in $\interior H$, and $\alpha\colon \bS^1 \to \partial H$ a meridian of $H$. The \emph{circulation of $X$ in $H$ with respect to meridian $\alpha$} is defined to be
\begin{equation}
\wind(X,\alpha,H) = \min_{\phi\in \cE(H,\alpha)} \min_{\sigma\in \longi(X)} \# (|\sigma| \cap \phi(\bB^2)).
\end{equation}

\begin{figure}[h!]
\includegraphics[scale=0.60]{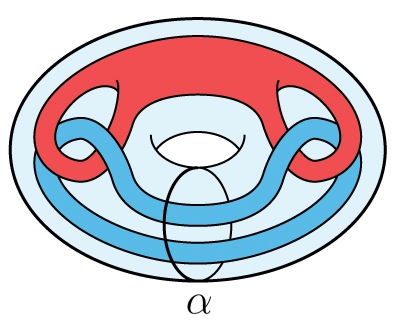}
\qquad \qquad
\includegraphics[scale=0.60]{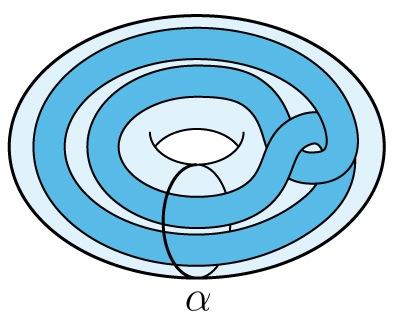}
\caption{Circulations having values $1$ and $2$  with respect to $\alpha$.}
\end{figure}

Let $(\R^3/G,\cX,(\fC,\fA,\fW))$ be a decomposition space of finite type. We call meridians of any cube-with-handles in $\cX$ \emph{meridians of $\cX$}. We denote this family of meridians by $\meri(\cX)$.

\begin{definition}
\label{def:mra}
Let $(\R^3/G,\cX,(\fC,\fA,\fW))$ be a decomposition space of finite type.
Meridians $\alpha\colon \bS^1\to \partial H$ and $\alpha'\colon \bS^1\to \partial H'$ in $\meri(\cX)$ are \emph{related by atlas $\fA$} if  the charts $\varphi_H$ and $\varphi_{H'}$ have the same target condenser in $\fC$ and $\varphi_H^{-1} \circ \varphi_{H'} \circ \alpha'$ is homotopic to $\alpha$ on $\partial H$.
\end{definition}

Given a meridian $\alpha$ in $\cX$,
\begin{equation}\label{eq:meridainA}
\meri_{\fA}(\cX;\alpha) = \{ \alpha' \in \meri(\cX) \colon \alpha\ \mathrm{and}\ \alpha'\ \mathrm{are\ related\ by\ atlas}\ \fA\}
\end{equation}
is called the \emph{meridians in $\cX$ related to $\alpha$ by $\fA$}.

\begin{definition}
\label{def:circulation}
Let $(\R^3/G,\cX,(\fC,\fA,\fW))$ be a decomposition space of finite type. We say that the \emph{order of circulation of $\cX$ is at least $\omega$}, $\omega \ge 0$, if there exist a meridian $\alpha_0\in \meri(\cX)$ and a constant $C>0$ such that for every $\ell\ge 0$ there exist $k'>k\ge 0$ with $k'-k>\ell$, a cube-with-handles $H\in \cC(X_k)$, and a meridian $\alpha \colon \bS^1 \to \partial H$ in $\meri_\fA(\cX;\alpha_0)$ which satisfy
\begin{equation}
\label{eq:circulation}
\wind(X_{k'}\cap H,\alpha,H) \ge C \omega^{k'-k}.
\end{equation}
\end{definition}

The homotopy property of a collection of meridians translates to geometric finiteness in the corresponding Semmes space after fixing a simple PL-representative for each homotopy class.

\begin{lemma}
\label{lemma:qs_meridians}
Let $(\R^3/G,\cX,(\fC,\fA,\fW),\theta,d_\lambda)$ be a Semmes space and $\alpha_0\in \meri(\cX)$. Then there exists $L=L(d_\lambda, \alpha_0)\ge 1$ so that for every meridian $\alpha \colon \bS^1 \to \partial H$ in $\meri_\fA(\cX;\alpha_0)$ there exists a meridian $\beta$ of $H$ homotopic to $\alpha$ in $\partial H$ so that $\pi_G \circ \beta\colon \bS^1 \to (\R^3/G,d_\lambda)$ is a $(\lambda^{\lvl(H)},L)$-quasisimilarity.
\end{lemma}

We record the observation that a quasisimilar meridian has a quasisimilar collar. The claim follows directly from properties of the metric $d_{\lambda,m}$, and the definition of the constant $\epsilon_\lambda$ in Remark \ref{rmk:epsilon}.
\begin{lemma}
\label{lemma:collection_A}
Let $(\R^3/G,\cX, (\fC,\fA,\fW), \theta, d_\lambda)$ be a Semmes space and $H\in \cC(\cX)$. Suppose $\alpha \colon \bS^1 \to \partial H$ is a  $(\lambda^{\lvl(H)},L)$-quasisimilar meridian on $H$. Then, for each $m\ge 0$, there exists a $(\lambda^{\lvl(H)},L')$-quasisimilarity
\[
\varkappa_\alpha \colon (\bB^{2+m} \times \bS^1,\{0\}\times  \bS^1) \to  (N_{d_{\lambda,m}}(\pi_G(\partial H),\epsilon_\lambda \lambda^{\lvl(H)}), (\pi_G\circ \alpha)\times\{0\})),
\]
where $L' >1$ is a constant depending only on $m, d_\lambda$ and $L$.
\end{lemma}

Given a defining sequence $\cX=(X_k)$ and a union $Y$ of a nonempty subcollection of cubes-with-handles in $\cC(X_k)$, we call
\begin{equation}\label{eq:longi2}
\longi(Y,\cX) = \{ \sigma\in \longi(Y)\colon |\sigma| \subset X_k\setminus X_{k+1} \}
\end{equation}
\emph{longitudes of $Y$ relative to $\cX$}. This subfamily $\longi(X_{k'}\cap H,\cX)$ of $\longi(X_{k'}\cap H)$ may be used to determine the circulation $\wind(X_{k'}\cap H,\alpha,H)$.

\begin{lemma}
\label{lemma:circulation_reduction}
Let $k\ge 0$, $H\in \cC(X_k)$, and $\alpha$ a meridian of $H$. Then
\[
\wind(X_{k'}\cap H,\alpha,H) = \min_{\phi\in \cE(H,\alpha)} \min_{\sigma\in \longi(X_{k'}\cap H,\cX)} \# (|\sigma|\cap \phi(\bB^2))
\]
for $k'>k$.
\end{lemma}

\begin{proof}
Let $\sigma= \sigma_1+\cdots + \sigma_\ell \in \longi(X_{k'}\cap H)$ and $\phi\in \cE(H,\alpha)$  be chosen so that
\[
\# ( |\sigma|\cap \phi(\bB^2)) = \wind(X_{k'}\cap H,\alpha,H).
\]
We claim that there exists a homeomorphism $h$ of $X_{k'}\cap H$, identity on $\partial (X_{k'}\cap H)$, so that $h\circ \sigma = h\circ \sigma_1 + \cdots + h\circ \sigma_\ell \in \longi(X_{k'}\cap H,\cX)$.

For every component $H_1,\ldots, H_d$ of $X_{k'}\cap H$ let $g_i$ be the genus of $H_i$ and let $\rho_i \colon \bigvee^{g_i} \bS^1 \to H_i$ be a core of $H_i$. Let $\mathcal{R}=\bigcup_i \rho_i(\bigvee^{g_i} \bS^1)$. By considering an isotopy of $X_{k'}\cap H$ if necessary, we may assume that $\mathcal{R} \cap |\sigma|= \emptyset$. Then there exists a regular neighborhood $X$ of $\mathcal{R}$ so that $(X_{k'}\cap H)\setminus X$ is homeomorphic to $\partial (X_{k'}\cap H) \times [0,1)$ and that $|\sigma|\subset (X_{k'}\cap H)\setminus X$. Then there exists a homeomorphism $h$ of $X_{k'}\cap H$, isotopic to the identity, so that $h((X_{k'}\cap H)\setminus X)\cap X_{k'+1}=\emptyset$ and that $h$ is identity on $\partial(X_{k'}\cap H)$. Hence $h\circ \sigma\in \longi(X_{k'}\cap H,\cX)$.

We extend the homeomorphism $h$ by identity on $H\setminus X_{k'}$. Then $h\circ \phi\in \cE(H,\alpha)$ and
\[
\# (|h\circ \sigma| \cap h(\phi(\bB^2)))=\# ( |\sigma|\cap \phi(\bB^2)) = \wind(X_{k'}\cap H,\alpha,H).
\]
The claim follows from  $\longi(X_{k'}\cap H,\cX)\subset \longi(X_{k'}\cap H)$.
\end{proof}

\subsection{Intersections in decomposition spaces}
\label{sec:int_ds}

When $\R^3/G$ is a manifold factor, circulation of cubes-with-handles in $\R^3$ can be estimated from above by the intersection number of longitudes and interior essential components of maps in the decomposition space $\R^3/G$, instead of $\R^3$.
The following proposition deals with this subtle, technical point.

In the following, $\Pi\colon \R^3/G \times \R^m \to \R^3/G$ is the projection map $(x,v)\mapsto x$.

\begin{proposition}
\label{prop:lifting}
Let $(\R^3/G,\cX)$ be a decomposition space, $\alpha \colon \bS^1 \to \partial H$ a meridian of $H\in \cC(\cX)$, and $\zeta \colon \bB^2 \to \pi_G H$ be a map satisfying $\zeta|\partial \bB^2 = \pi_G \circ \alpha$. Suppose that $\R^3/G\times \R^m$ is homeomorphic to $\R^{3+m}$ for some $m\ge 0$. Then
\[
\# (\pi_G(|\sigma|)\cap \zeta (\bB^2)) \ge \wind(X_{k'}\cap H,\alpha,H)
\]
for all $k'>\lvl(H)$ and every longitude $\sigma \in \longi(X_{k'}\cap H,\cX)$.
\end{proposition}

The proof is based on the following approximation lemma.

\begin{lemma}
\label{lemma:HW}
Under the hypotheses of the proposition, for every $k'>k$ there exists a map $\phi \colon \bB^2 \to H$ so that $ \pi_G \circ \phi|\Omega = \zeta|\Omega $, where $\Omega $ is the component of $ \zeta^{-1}(\pi_G H\setminus \pi_G( X_{k'}))$ that contains $\partial \bB^2$.
\end{lemma}

\begin{proof}
If $\zeta(\bB^2)\cap \pi_G(X_K) = \emptyset$ for some $K>0$ then we may take $\phi = \pi_G^{-1} \circ \zeta$, since $\pi_G|\R^3\setminus X_K$ is a homeomorphism. The conclusion follows. Thus we may assume that $\zeta(\bB^2)\cap \pi_G(X_K)\ne \emptyset$ for all $K>0$.

We fix a homeomorphism $f \colon \R^3/G\times \R^m \to \R^{3+m}$ and a number $R>0$ so that $f(\zeta(\bB^2)) \subset B^{3+m}(R)$. Let $\mathcal B'=B^{3+m}(R+1)$, $\mathcal B=B^{3+m}(R+2)$, and
\[
\varepsilon= \frac{1}{4} \min\{ 1, \dist(f(\pi_G(\partial X_{k'+1})\times \R^m)\cap \mathcal B, f(\pi_G(X_{k'+2})\times \R^m)\cap \mathcal B)\}.
\]
Since $\zeta$ and $f|f^{-1} \mathcal B$ are uniformly continuous, we may fix $\delta>0$ so that $|f(\zeta(x))-f(\zeta(y))|<\varepsilon/5$ for all $x,y\in\bB^2$ satisfying $|x-y|<\delta$.

We fix $K>k'+2$ so that the diameters of components of $\zeta^{-1}(\pi_G(X_K))$ are at most $\delta/2$. Let $\Omega_K$ be the component of $\zeta^{-1}(\R^3/G\setminus \pi_G(X_K))$ that contains $\partial \bB^2$. Then $\Omega \subset \Omega_K$.

Since $\pi_G$ is a homeomorphism near the boundary of $X_K$, we may use the transversality and the PL-structure in $\R^3$ to modify $\zeta$ in a neighborhood of $\pi_G(\partial X_K)$ in $\pi_G(X_{k+2})$ in such a way that the components of $\zeta^{-1}(\pi_G(\partial X_K))$ are topological circles, that $|f(\zeta(x))-f(\zeta(y))|<\varepsilon/4$ for all $|x-y|<\delta$, and that $f(\zeta(\bB^2))\subset B^{3+m}(R+\varepsilon)$.

For each component $C$ of $\partial \Omega_K$, except for the outermost boundary $\partial \bB^2$, we denote by $\omega$ the $2$-cell  in $\bB^2$ enclosed by $C$, thus $C=\partial \omega$, and define a map $\tilde \phi_\omega \colon \omega \to \R^{3+m}$ extending  $f\circ \zeta|\partial \omega$ as follows.

Let $\tau \colon \omega\to \bB^2$ be a homeomorphism and
fix a point $y_0\in f(\zeta(\partial \omega))$. Define $\tilde \phi_\omega \colon \omega \to \R^{3+m}$ so that $ \tilde \phi_\omega(\tau^{-1}(0))=y_0$ and
\[
\tilde \phi_\omega(x) = (1-|\tau(x)|) y_0 + |\tau(x)| f \circ \zeta \circ \tau^{-1}\left( \frac{\tau(x)}{|\tau(x)|}\right),\quad x\neq \tau^{-1}(0).
\]
Then  $\tilde \phi_\omega|\partial \omega= f\circ \zeta|\partial \omega  $.  Since $\diam \partial \omega < \delta$ we have $\diam f(\zeta(\partial \omega))\le \varepsilon/4$, $\diam( \tilde \phi_\omega(\omega))<\varepsilon$, and $\tilde \phi_\omega(\omega) \subset \mathcal B'$;
since $\zeta(\partial \omega)\subset \pi_G( \partial  X_K)$, we have $\tilde \phi_\omega(\partial \omega))\subset f(\pi_G (X_{k'+2})\times \R^m)\cap \mathcal B'$. Therefore
\begin{eqnarray*}
&& \dist( \tilde \phi_\omega(\omega), f(\pi_G(\partial X_{k'+1})\times \R^m)) \\
&& \quad \ge \dist( \tilde \phi_\omega(\partial \omega), f(\pi_G(\partial X_{k'+1})\times \R^m)) - \diam(\tilde \phi_\omega(\omega))\\
&& \quad \ge \min\{1, \dist(\tilde \phi_\omega(\partial \omega), f(\pi_G(\partial X_{k'+1})\times \R^m)\cap \mathcal B)\} \\
&& \qquad - \diam(\tilde \phi_\omega(\omega)) \,> 3 \varepsilon.
\end{eqnarray*}
Thus $\tilde \phi_\omega(\omega) \subset f(\pi_G(X_{k'+1})\times \R^m)$.

We define map $\phi \colon \bB^2 \to \R^3$ by $\phi|\Omega_K = \pi_G^{-1} \circ \zeta|\Omega_K$, and $\phi|\omega = \Pi \circ f^{-1} \circ \tilde \phi_\omega$ on every $2$-cell $\omega$ bounded by a component of $\partial \Omega_K\setminus \partial \bB^2$. Since $\pi_G|\R^3\setminus X_K$ is a homeomorphism and
$\tilde \phi_\omega|\partial \omega =f\circ \zeta|\partial \omega$, the map is well-defined and continuous.

Since $\phi(\bB^2)$ is connected, $\phi(\bB^2)\subset H$, and since $\Omega \subset \Omega_K$, $\phi|\Omega= \pi_G^{-1} \circ \zeta|\Omega$. The claim follows.
\end{proof}

\begin{proof}[Proof of Proposition \ref{prop:lifting}]
The map $\phi \colon \bB^2 \to H$ constructed in Lemma \ref{lemma:HW} belongs to $\cE(H;\alpha)$ and satisfies $\pi_G\circ \phi|\Omega=\zeta|\Omega $. Then, by Lemma \ref{lemma:circulation_reduction},
\[
\# (\zeta (\bB^2) \cap  \pi_G(|\sigma|)) \ge \# (\phi (\bB^2) \cap |\sigma|) \ge \wind(X_{k'}\cap H,\alpha,H)
\]
for every $\sigma \in \longi(X_{k'}\cap H,\cX)$. The claim follows.
\end{proof}

\subsection{Virtually interior essential components}
\label{sec:VIE}

Let $\Omega$ be a $2$-manifold with boundary, $M$ an $n$-manifold with boundary, and $\phi\colon (\Omega,\partial \Omega) \to (M,\partial M)$ a map. Following Daverman \cite[pp. 73-74]{DavermanR:Decm}, we
 say that $\phi$ is \emph{interior inessential} if there exists a map $\phi'\colon \Omega \to \partial M$ so that $\phi'|\partial \Omega = \phi|\partial \Omega$; if no such map exists, we say that $\phi$ is \emph{interior essential}.

If $\phi\colon (\Omega,\partial \Omega) \to (M,\partial M)$ is interior essential and $\Omega$ is a submanifold of a $2$-cell $D$ so that $\partial D\subset \partial \Omega$, we say that $\phi$ is \emph{virtually interior essential} if there exists a map $\Phi\colon D\to M$ so that $\Phi|\Omega = \phi$ and $\Phi(D\setminus \Omega) \subset \partial M$.

Let $\cX$ be a defining sequence for a decomposition space. Given $H\in \cC(X_k)$, a meridian $\alpha$ of $H$, and $k'>k$, denote by $\cE(H,\alpha; X_{k'})$ the collection of  maps  $\phi \colon (\bB^2,\partial \bB^2) \to (H,\alpha)$ such that $\phi(\bB^2)$ is transverse to $\partial X_{k'}$. Given $\phi\in \cE(H,\alpha; X_{k'})$, we say that a component $\omega$ of $\phi^{-1}X_{k'}$ is \emph{virtually interior essential with respect to $X_{k'}$} if $\phi|\omega \colon (\omega,\partial \omega) \to (X_{k'},\partial X_{k'})$ is virtually interior essential. We denote by $\Gamma(\phi,X_{k'})$ the set of \emph{virtually interior essential components of $\phi^{-1}X_{k'}$}.

\begin{remark}
\label{rmk:vie}
The circulation $\wind(X_{k'}\cap H,\alpha,H)$ is closely related  to the minimal number of essential components among all maps in $ \cE(H,\alpha; X_{k'})$.
In fact,
\[
\wind(X_{k'}\cap H,\alpha,H) \ge \min_{\phi\in \cE(H,\alpha; X_{k'})} \# \Gamma(\phi,X_{k'}).
\]
\end{remark}

 Indeed, given $\phi\in \cE(H,\alpha; X_{k'})$ and $\sigma \in \longi(X_{k'}\cap H)$, it follows from the definition of longitudes that $|\sigma|\cap \phi(\omega)  \ne \emptyset$ for every $\omega \in \Gamma(\phi,X_{k'})$. Thus $\# (|\sigma|\cap \phi(\bB^2) ) \ge \# \Gamma(\phi,X_{k'})$.

\section{Circulation and a modulus estimate for walls}
\label{sec:winding}

Suppose that $(\R^3/G,\cX, (\fC,\fA,\fW), \theta, d_\lambda)$ is a Semmes space. Let $Y$ be the union of a nonempty subcollection of cubes-with-handles in $\cC(X_k)$, $m\ge 0$, and $a>0$.
We call carriers of $(1+m)$-chains in the collection
\[
\longi^m(Y,\cX,a) = \{ |\sigma| \times [-a,a]^m  \colon \sigma \in \longi(Y,\cX)\}
\]
\emph{$m$-walls over $Y$ of height $a$ relative to $\cX$}. Note that these walls do not meet $X_\infty\times \R^m$ and that $\pi_G|\R^3\setminus X_\infty$ is a homeomorphism. We denote by
\[\hat \longi^m(Y,\cX,a) = (\pi_G\times \id)(\longi^m(Y,\cX,a))
\]
the corresponding collection of $m$-walls in the decomposition space $ \R^3/G\times \R^m$.

The main result of this section is an upper estimate for the conformal modulus of an $m$-wall family in terms of circulation. This, together with a lower estimate in terms of growth, yields a necessary condition for the existence of quasisymmetric parametrization. Our result extends the second part of \cite[Proposition 4.5]{HeinonenJ:Quansa}.

\begin{theorem}
\label{thm:key}
Let  $(\R^3/G,\cX, (\fC,\fA,\fW), \theta, d_\lambda)$ be a Semmes space, $\alpha \colon \bS^1 \to \partial H$ be an $(\lambda^k,L)$-quasisimilar meridian of $H\in \cC(X_k)$, and  $m\ge 0$. Suppose that $f\colon (\R^3/G\times \R^m,d_{\lambda,m}) \to \R^{3+m}$ is $\eta$-quasisymmetric. Then there exist $A = A(\eta, d_\lambda, m,L)>0$ and $C=C(\eta,d_\lambda,m,L)$ so that
\[
\Mod_{\frac{3+m}{1+m}}\left( f(\hat \longi^m(X_{k'}\cap H, \cX, A\lambda^k))\right) \le C \left( \frac{1}{\wind(X_{k'}\cap H,\alpha, H)}\right)^{\frac{3+m}{1+m}}
\]
for all $k'>k+1$.
\end{theorem}

We begin with an intersection lemma which contains the gist of the proof; the number $\epsilon_\lambda$ in the statement is the constant defined in Remark \ref{rmk:epsilon}.

\begin{lemma}
\label{lemma:hits}
Suppose $g\colon (\R^3/G\times \R^m, d_{\lambda,m}) \to \R^{3+m}$ is $\eta$-quasisymmetric. Let $\alpha \colon \bS^1 \to \partial H$ be a meridian of $H\in \cC(\cX)$, and  $\beta \colon \bS^1 \to \R^3/G\times \R^m$ be a map homotopic to $\pi_G \circ \alpha$ in $N_{d_{\lambda,m}}(\pi_G(\partial H), \epsilon_\lambda \lambda^{\lvl(H)}/3)$ with the property that $g ( \beta(\bS^1))=\partial \bB^2 \times \{0\}\subset \R^2\times \R^{1+m}$. Then there exist $\delta = \delta(\eta,d_{\lambda,m})>0$
and $A=A(\eta_{g^{-1}})\ge 1$ so that
\begin{equation}
\label{eq:hits}
\# (g( w) \cap (\bB^2+j)) \ge \wind(X_{k'}\cap H,\alpha,H)
\end{equation}
for every $k'>\lvl(H)$, $m$-wall $w \in \hat \longi^m(X_{k'}\cap H,\cX,A\lambda^{\lvl(H)})$, and $j\in \{0\}\times B^{1+m}(\delta)\subset \R^2\times \R^{1+m} $.
\end{lemma}

\begin{figure}[h!]
\includegraphics[scale=0.60]{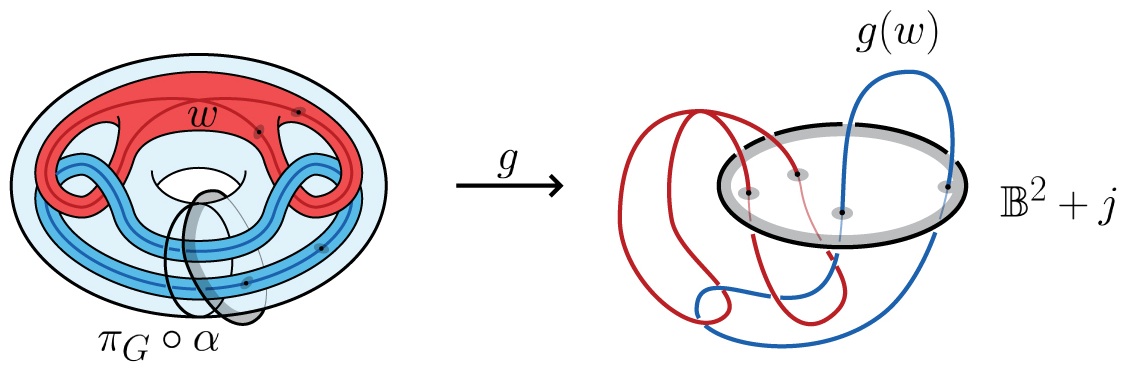}
\caption{An example of \eqref{eq:hits}, where $m=0, k=0,$ and $k'=1$}
\end{figure}

\begin{proof}
Let $k=\lvl(H)$. We show first that $\dist(\partial \bB^2, g(\pi_G(X_{k+1})\times \R^m))$ is bounded from below by a positive constant depending only on $\eta$ and $\lambda$.

Since $ \beta(\bS^1) \subset N_{d_{\lambda,m}}(\pi_G(\partial H),\epsilon_\lambda \lambda^k/3) \subset \pi_G(\R^3\setminus X_{k+1})\times \R^m$   and $g( \beta(\bS^1)) = \partial \bB^2$, we may fix $x \in \beta(\bS^1)$ and $z\in \pi_G(X_{k+1})\times \R^m$ so that
\[
|g(x)-g(z)| = \dist(\partial \bB^2,g(\pi_G(X_{k+1})\times \R^m)).
\]
We also fix $y \in \beta(\bS^1)$ so that $d_{\lambda,m}(x,y)=d_\lambda(x,y)=\max_{y'\in \beta(\bS^1)} d_\lambda(x,y')$. Since $x,y \in N_{d_{\lambda,m}}(\pi_G(\partial H),\epsilon_\lambda \lambda^k)$ and  projection $\Pi(z)\in \pi_G(X_{k+1})$, we have, by quasisymmetry and Remark \ref{rmk:mp1},
\begin{eqnarray*}
|g(x)-g(y)| &\le& \eta\left(\frac{d_{\lambda,m}(x,y)}{d_{\lambda,m}(x,z)}\right)|g(x)-g(z)| \\
&\le& \eta\left( \frac{ \diam_{d_{\lambda,m}}(N_{d_{\lambda,m}}(\pi_G(\partial H),\epsilon_\lambda \lambda^k))}
{\dist_{d_{\lambda,m}}(N_{d_{\lambda,m}}(\pi_G(\partial H), \pi_G(X_{k+1}))} \right)
|g(x)-g(z)| \\
&\le& \eta(C(d_{\lambda,m})) |g(x)-g(z)|.
\end{eqnarray*}
Choose $x'\in \beta(\bS^1)$ so that $g(x')$ and $g(x)$ are antipodal on $\partial \bB^2$. Then,
\[
|g(x)-g(x')| \le \eta\left( \frac{d_{\lambda,m}(x,x')}{d_{\lambda,m}(x,y)}\right) |g(x)-g(y)| \le \eta(1) |g(x)-g(y)|.
\]
Thus
\begin{equation}\label{eq:circular_nbhd}
 \dist(\partial \bB^2,g(\pi_G(X_{k+1})\times \R^m)) \ge \frac{2}{\eta(1)\eta(C(d_{\lambda,m}))}.
\end{equation}

Since $g$ is $\eta$-quasisymmetric, we may fix $\delta \in (0, (\eta(1)\eta(C(d_{\lambda,m})))^{-1})$, independent of $k$, so that
\begin{equation}\label{eq:preimage_circular_nbhd}
g^{-1}(\partial \bB^2 \times B^{1+m}(\delta) )\subset N_{d_{\lambda,m}}(\pi_G(\partial H), \epsilon_\lambda \lambda^k) .
\end{equation}

We prove now \eqref{eq:hits}. Let $k'>k$ and $j \in \{0\}\times B^{1+m}(\delta)$; define $\phi_j \colon (\bB^2,\partial \bB^2) \to \R^3/G\times \R^m$ to be the map $\phi_j(x) = g^{-1}(x+j)$.

By the definition of $\epsilon_\lambda$, $N_{d_{\lambda}}(\pi_G(\partial H),\epsilon_\lambda \lambda^k)\subset \R^3/G$ is contained in a regular neighborhood of $\pi_G(\partial H)$.
In view of \eqref{eq:circular_nbhd} and \eqref{eq:preimage_circular_nbhd}, the projection $\Pi \circ \phi_j|\partial \bB^2$ is homotopic to $\pi_G\circ \alpha$ in  $N_{d_{\lambda}}(\pi_G(\partial H),\epsilon_\lambda \lambda^k)\subset \R^3/G$, and there exists a map $\zeta\colon (\bB^2,\partial \bB^2) \to (\pi_G H,\pi_G\circ\alpha)$ so that $\zeta|\Omega = \Pi\circ \phi_j|\Omega$, where $\Omega = (\Pi\circ \phi_j)^{-1}(\pi_G(X_{k'}))$.

Note, from Proposition \ref{prop:lifting}, that
\[
\# ( \pi_G(|\sigma|)\cap \zeta(\bB^2)) \ge \wind(X_{k'}\cap H,\alpha,H)
\]
for every $\sigma \in \longi(X_{k'}\cap H,\cX)$.
Thus
\begin{eqnarray*}
\# (g(\pi_G(|\sigma|)\times \R^m) \cap (\bB^2+j)) &=& \# ((\pi_G(|\sigma|)\times \R^m)\cap g^{-1}(\bB^2+j)) \\
&=& \# ((\pi_G(|\sigma|) \times \R^m) \cap \phi_j (\bB^2 )) \\
&\ge& \# (\pi_G(|\sigma|) \cap \Pi \phi_j (\bB^2) ) \\
&=& \# ( \pi_G(|\sigma|)\cap \zeta(\bB^2)) \\
&\ge& \wind(X_{k'}\cap H,\alpha,H)
\end{eqnarray*}
for all $\sigma \in \longi(X_{k'}\cap H,\cX)$. This concludes the proof in the case $m=0$.

Suppose now $m\ge 1$. It suffices to find $A=A(\eta)\ge 1$ so that
\begin{equation}
\label{eq:A}
\phi_j(\bB^2) \subset \R^3/G \times [-A\lambda^k,  A\lambda^k]^m.
\end{equation}
Let $x\in \partial \bB^2$ and $y\in \bB^2$. By quasisymmetry of $g^{-1}$, we have
\begin{eqnarray*}
\left| \phi_j(y) - \phi_j(x) \right| &=& \left| g^{-1}(y+j) - g^{-1}(x+j) \right| \\
&\le& \eta_{g^{-1}}\left( \frac{|y-x|}{|(-x)-(x+j)|}\right) \left| g^{-1}(-x) - g^{-1}(x+j)\right| \\
&\le& \eta_{g^{-1}}\left( \frac{|y-x|}{2-|j|} \right) \eta_{g^{-1}} \left( \frac{2+|j|}{2} \right) |g^{-1}(-x)-g^{-1}(x)| \\
&\le& \eta_{g^{-1}} \left( 1 \right) \eta_{g^{-1}}\left( 2 \right) \diam_{d_{\lambda,m}} g^{-1}(\partial \bB^2).
\end{eqnarray*}
Since
\[
g^{-1}(\partial \bB^2)=|\beta|\subset N_{d_{\lambda,m}}(\pi_G(\partial H),\epsilon_\lambda \lambda^k),
\]
\eqref{eq:A} holds with $A=C(\lambda) \eta_{g^{-1}}(1) \eta_{g^{-1}}(2)$.  The claim now follows.
\end{proof}

The proof of Theorem \ref{thm:key} is based on Lemma \ref{lemma:hits} and unknotting properties of quasisymmetric tubes; see Propositions \ref{prop:chi} and \ref{prop:chi2}.

\begin{proof}[Proof of Theorem \ref{thm:key}]
We first consider the case $m\ge 1$. Let $\alpha \colon \bS^1\to\partial H$ be the $(\lambda^k,L)$-quasisimilar meridian in the statement of the theorem, and $k'>k+1$.
We assume, as we may, that $\wind(X_{k'}\cap H,\alpha,H)>0$. So
\[
\#  (|\sigma|\cap \phi(\bB^2) )\ge \wind(X_{k'}\cap H,\alpha,H)
\]
for  all $\sigma \in \longi(X_{k'}\cap H,\cX)$ and all maps $\phi \colon (\bB^2,\partial \bB^2) \to (H, \alpha)$.

By Lemma \ref{lemma:collection_A}, there exists an $(\lambda^k,L')$-quasisimilar, thus $\eta'$-quasisymmetric, embedding $\varkappa \colon \bB^{2+m} \times \bS^1 \to  N_{d_{\lambda,m}}(\pi_G(\partial  H),\epsilon_\lambda \lambda^k/3)$ so that $\varkappa(0,x)=(\pi_G\circ \alpha(x),0)$ for $x\in \bS^1$, where the constant $L'$ and the homeomorphism $\eta' \colon [0,\infty) \to [0,\infty)$ depend only on $d_\lambda, m$ and $L$. Recall that $N_{d_{\lambda,m}}(\pi_G(\partial  H),\epsilon_\lambda \lambda^k)\subset \pi_G(\R^3\setminus X_{k+1}) \times \R^m$.

Set $T=f\circ \varkappa(\bB^{2+m}\times \bS^1)\subset \R^{3+m}$ and $h = f\circ \varkappa$. By Proposition \ref{prop:chi}, there exist an $\eta''$-quasisymmetric, $\eta''=\eta''(m,\eta, \eta')$, map $\chi \colon \R^{3+m} \to \R^{3+m}$ and a constant $\delta_0=\delta_0(m,\eta,\eta')>0$ so that $\chi(T)$ contains the tubular neighborhood $N^{3+m}(\partial \bB^2, 2 \delta_0)$ of $\partial \bB^2$ in $\R^{3+m}$ and that $\chi \circ f\circ \pi_G \circ \alpha \colon \bS^1 \to \R^{3+m}$ is homotopic in $\chi(T)$ to the identity map $\id \colon \partial \bB^2 \to \R^2\times\R^{1+m}$. Set $\beta = f^{-1} \circ \chi^{-1} \circ \id_{\R^{3+m}}|\bS^1 \colon \bS^1 \to \R^3/G \times \R^m$.  Thus we have
\[
\xymatrix{
& \bB^{2+m}\times \bS^1 \ar[dr]^\varkappa & \\
\bS^1 \ar@{_{(}->}[dr] \ar@{^{(}->}[ur] \ar@<0.5ex>[rr]^{\pi_G\circ \alpha}
\ar@<-0.5ex>[rr]_\beta
& & \R^3/G \times \R^m \ar[d]^{f}  \\
& \R^{3+m} & \R^{3+m} \ar[l]_\chi }
\]
where both diagrams commute and maps $\pi_G \circ \alpha$ and $\beta$ are homotopic in $\varkappa(\bB^{2+m}\times \bS^1)$.

Note that $\chi\circ f \colon \R^3/G \times \R^m \to \R^{3+m}$ is $\eta'''$-quasisymmetric for some $\eta'''=\eta'''(\eta,\eta'')$.

Since $\beta$ is homotopic to $\pi_G\circ \alpha$ in $\varkappa(\bB^{2+m}\times \bS^1) \subset  N_{d_{\lambda,m}}( \pi_G (\partial H),\epsilon_\lambda \lambda^k/3)$ and $\dist(\partial \bB^2, \chi\circ f(\pi_G(X_{k+1})\times \R^m))>2\delta_0$,
we can find, by applying Lemma \ref{lemma:hits} to $g=\chi \circ f$ and setting $\delta=\delta_0$,
a constant $A=A(\eta''', m)$ so that
\begin{equation}
\label{eq:tilde_f_intersection}
\# \left( \chi ( f ( w) )\cap \left( \bB^2+j \right) \right) \ge \wind(X_{k'}\cap H,\alpha,H),
\end{equation}
for every $w \in \hat \longi^m(X_{k'},\cX,A\lambda^k)$ and $j\in \{0\}\times B^{1+m}(\delta_0)$.

Using \eqref{eq:tilde_f_intersection} we estimate the conformal modulus of the $m$-wall family $\chi\circ f  (\hat \longi^m(X_{k'},\cX, A \lambda^k))$ in $\R^{3+m}$. Set $J=\{0\} \times B^{1+m}(\delta_0)$.
By the co-area formula, we have, for $w\in \hat\longi^m(X_{k'},\cX, A\lambda^k)$, that
\begin{eqnarray*}
\haus^{1+m}(\chi( f (w))) &\ge& \haus^{1+m}\left(  \chi(f ( w))  \cap \left( \bB^2\times B^{1+m}(\delta_0)\right) \right) \\
&\ge& \int_J \# \left(  \chi(f ( w)) \cap \left( \bB^2+j \right)\right) \dhaus^{1+m}(j) \\
&\ge& \wind(X_{k'}\cap H,\alpha,H) \haus^{1+m}(J).
\end{eqnarray*}
Thus
\[
\rho = \frac{1}{\wind(X_{k'}\cap H,\alpha,H)\,\,
\haus^{1+m}(J)} \chi_{\bB^2\times J},
\]
is an admissible function for the family $\chi( f (\hat \longi^m(X_{k'},\cX,A \lambda^k)))$, and
\begin{eqnarray*}
&& \Mod_{\frac{3+m}{1+m}}(\chi (f \hat \longi^m(X_{k'},\cX, A \lambda^k)))
\le \int_{\R^{3+m}} \rho^{\frac{3+m}{1+m}} \dhaus^{3+m} \\
&&\qquad = \frac{\haus^2(\bB^2) \haus^{m+1}(J)}{\haus^{m+1}(J)^{\frac{3+m}{1+m}}} \left( \frac{1}{\wind(X_{k'}\cap H,\alpha,H)} \right)^{\frac{3+m}{1+m}} \\
&&\qquad \le C(\delta_0, m) \left( \frac{1}{\wind(X_{k'}\cap H,\alpha,H)} \right)^{\frac{3+m}{1+m}}.
\end{eqnarray*}
This concludes the proof for $m\ge 1$.

In the case $m=0$, we apply Proposition \ref{prop:chi2} to the mapping $h=f\circ \varkappa$ and the $3$-manifold $M= f(\pi_G(H))$. Otherwise the proof is the same.
\end{proof}

\section{Quasisymmetric tubes}
\label{sec:tubes}

In Proposition \ref{prop:chi} we quantify Zeeman's unknotting theorem to provide a quasisymmetrical unknotting of quasisymmetric tubes in $\R^n, n\ge 4$.
In Proposition \ref{prop:chi2} we apply Dehn's Lemma to treat the unknotting in $\R^3$.

\begin{proposition}
\label{prop:chi}
Let $m\ge 1$, $h \colon \bB^{2+m}\times \bS^1 \to \R^{3+m}$ be an $\eta$-quasisymmetric embedding and $T=h(\bB^{2+m}\times \bS^1)$. Then there exist  an $\eta'$-quasisymmetric homeomorphism $\chi\colon \R^{3+m} \to \R^{3+m}$, $\eta'=\eta'(m,\eta)$, and a constant $\delta_0=\delta_0(m,\eta)>0$ so that
\begin{enumerate}
\item [(1)]$\chi (T)$ contains the tube $N^{3+m}(\partial \bB^2, \delta_0)$ in $\R^{3+m}$, in particular
\item[(1)$'$] $\partial \bB^2 + j \subset \chi (T)$ for $j\in\{0\} \times B^{1+m}(\delta_0)\subset \R^2 \times \R^{1+m}$, and
\item[(2)] $\chi \circ h|(\{0\}\times \bS^1)$ is homotopic to the identity map $\id|\partial \bB^2$ in $\chi (T)$.
\end{enumerate}
\end{proposition}

Here $\bB^{2+m}\times \bS^1$ has the natural Euclidean metric inherited from $\R^{2+m}\times \R^2 = \R^{3+m}$.

We first state a bilipschitz version of Zeeman's theorem on unknotting PL $1$-sphere in $\bS^q$ for $q\ge 4$. Since the claim follows from \cite[Theorem 5.6, Corollary 5.9]{RourkeC:Intplt} almost directly, we omit the details.

For the statement, let $\ell\in \Z_+$, $m \ge 1$. Given $ w_1,\ldots, w_n \in (1/\ell)\Z^{3+m}$, we set $w = (w_1,\ldots, w_n)$, and set $\gamma_w$ to be the piecewise linear curve $ [w_0,w_1]\cup [w_1,w_2] \cup \cdots \cup [w_{n-1},w_0]$ in $\R^{3+m}$. Given $R>0$, we also denote by $\cJ(R,\ell,m; n)$ the collection of Jordan curves in $\{\gamma_w \subset B^{3+m}(R) \colon w\in ((1/\ell)\Z^{3+m})^n\}$.

\begin{lemma}
\label{lemma:Z}
Let $R\ge 1$, $\ell\in \Z_+$, $m\ge 1$, and $n\ge 3$. Then there exists $L_0=L_0(R, \ell, m, n)$ so that given $\gamma \in \cJ(R,\ell,m;n)$ there exists an $L_0$-bilipschitz map $\chi\colon \R^{3+m} \to \R^{3+m}$ satisfying $\chi(\gamma) = \partial B^2(\diam \gamma) \times \{0\} \subset \R^2\times \R^{1+m}$.
\end{lemma}

\begin{proof}[Proof of Proposition \ref{prop:chi}]
Set $\bS^1=\{0\}\times \bS^1$. Then, by quasisymmetry,
\[
\diam h(\bS^1)\le \eta(5)\dist( h(\bS^1), \partial T) .
\]

Indeed, set $\kappa = \dist( h(\bS^1), \partial T)$ and choose $x \in \bS^1$ and $y\in \partial (\bB^{2+m}\times \bS^1)$ be so that $|h(x)-h(y)|=\kappa$. Then
\[
|h(x')-h(x)| \le \eta\left( \frac{|x'-x|}{|y-x|} \right) |h(y)-h(x)| \le  \eta(5) \kappa
\]
for all $x' \in \bS^1$.

We fix an orientation of $h(\bS^1)$ and choose points  $z_0, z_1,\ldots, z_n = z_0$ on $h(\bS^1)$ as follows. Let $z_0$ be any point on $h(\bS^1)$. After $z_i$ has been chosen, let $z_{i+1}$ be the last point $z$ on the subarc of $h(\bS^1)$ starting at $z_i$ and ending at $z_0$ according to the orientation, so that $|z - z_i|= \kappa/100$  if such a point exists; otherwise, we have $|z_0 - z_i|< \kappa/100$ and in this case we remove the already defined value of $z_i$ and set $n=i$ and $z_n=z_0$. We show next that $n \le n_0$ for some $n_0 = n_0(\eta)>0$.

Let $s_i = h^{-1}(z_i)\in \bS^1$ for $0 \le i \le n-1$. Then there exists an $i$ so that $|s_i-s_{i+1}|\le 2\pi/n$; for this particular $i$,
\[
\kappa/100 \le |z_i - z_{i+1}| \le \eta\left( \frac{|s_i - s_{i+1}|}{|(-s_i)-s_i|} \right) |h(-s_i)-h(s_i)| \le \eta( 1/n ) \eta(5)\kappa.
\]
Hence $|s_i-s_{i+1}|\ge C_0$, where $C_0$ depends on $\eta$, and $n \le 2\pi/C_0$.

We next fix points $w_i \in (\kappa/(1000\sqrt{m}n_0))\Z^{3+m}$ so that $|w_i-z_i|<\kappa/500$ and let $\gamma$ be the polygonal path $[w_0,w_1]\cup [w_1,w_2]\cup \cdots \cup [w_{n-1},w_0]$.

By replacing the points $w_i$ with points in $(\kappa/(1000\sqrt{3+m}n_0))\Z^{3+m}\cap B^{3+m}(z_i,\kappa/500)$, we may assume that $\gamma$ is a Jordan curve. Indeed, if $\gamma$ is not a Jordan curve, then there exist indices $i$ and $j$, $i>j$, so that $(w_i,w_{i+1})\cap (w_j,w_{j+1})\ne \emptyset$. Since $B^{3+m}(z_i,\kappa/500)$ contains more than $n_0^{3+m}$ points in $(\kappa/(1000\sqrt{3+m}n_0))\Z^{3+m}$ and there are at most $n_0(n_0-1)/2$ directions between the points $w_1,\ldots, w_n$, there exists $w' \in B^{3+m}(z_i,\kappa/500)$ so that $(w_i,w')\cap (w_k,w_{k+1})=\emptyset$ for all $k < i$. We remove all the intersections inductively on $i$.

Since $\max_{w\in[w_i,w_{i+1}]}\dist(w, z_i) \le \kappa/40$, we have $\max_{w\in \gamma}\dist(w, h(\bS^1)) \le \kappa/40$. Thus $\dist(\gamma, \partial h(T)) \ge 39\kappa/40$.

Let $\iota \colon \R^{3+m} \to \R^{3+m}$ be a linear transformation $\iota(x)= -w_0 + x/\kappa$ which
maps $\gamma$ into $B^{3+m}(\eta(5))$. Then
\[
\iota(\gamma) \in \cJ(2\eta(5),1/(1000\sqrt{3+m}n_0), m;n).
\]
By Lemma \ref{lemma:Z}, there exists an $L_0$-bilipschitz, therefore $\eta'$-quasisymmetric, homeomorphism $\chi'$ of $\R^{3+m}$ so that $\chi'(\iota(\gamma)) = \partial B^2(\diam \iota(\gamma))\subset \R^2\times \R^{1+m}$,  where $\eta'$ depends only on $\eta(5)$, $m$, and $n_0$. Then $\chi = (\diam \iota(\gamma))^{-1} \chi' \circ \iota$ is also $\eta'$-quasisymmetric. Since $n\le n_0$ and $n_0$ depends only on $\eta$, we have that $\eta'=\eta'(m,\eta)$. The existence of the constant $\delta_0$ follows from quasisymmetry of $\chi$ and geometry of $\bB^{2+m}\times \bS^1$.
\end{proof}

\begin{proposition}
\label{prop:chi2}
Let $M$ be a PL $3$-manifold with boundary in $\R^3$. Suppose $h \colon \bB^2\times \bS^1 \to \R^3$ is an $\eta$-quasisymmetric embedding with properties that $h$ embeds $\{0\}\times \bS^1$ into $\partial M$ and  $h|(\{0\}\times \bS^1)$  is null-homotopic in $M$. Let $T=h(\bB^2\times \bS^1)$. Then there exist  an $\eta'$-quasisymmetric homeomorphism $\chi\colon \R^3 \to \R^3$, $\eta'=\eta'(\eta)$, and a constant $\delta_0=\delta_0(\eta)>0$ so that
\begin{enumerate}
\item[(1)]$\chi (T)$ contains the tube $N^3(\partial \bB^2, \delta_0)$ in $\R^3$, in particular
\item[(1)$'$] $\partial \bB^2 + j \subset \chi (T)$ for $j\in\{0\} \times [-\delta_0,\delta_0]\subset \R^2\times \R $, and
\item[(2)] $\chi \circ h|(\{0\}\times \bS^1)$ is homotopic to the identity map $\id|\partial \bB^2$ in $\chi (T)$.
\end{enumerate}
\end{proposition}

\begin{proof}
Let $\alpha \colon \bS^1 \to M$ be the map $x\mapsto h(0,x)$. Then, by assumption, $\alpha$ is simple and null-homtopic in $M$.
We show first that $\alpha$ is an unknot. It suffices to show that there exists an embedding $\tau \colon \bB^2 \to M$ for which $\tau|\partial \bB^2 = \alpha$.

Since $\alpha$ is null-homotopic, there exists an extension $\hat \alpha \colon \bB^2 \to M$ of $\alpha$. Since $M$ is a PL manifold with boundary, $\partial M$ has a collar in $M$; see \cite[Corollary 2.26]{RourkeC:Intplt}. Thus we may assume that $\partial \bB^2$ has a neighborhood $A$ in $\bB^2$ for which $\hat \alpha|A$ is an embedding and $\hat\alpha^{-1}(\hat\alpha(A))=A$. Thus conditions of Dehn's Lemma (see e.g.\;\cite[Chapter 4]{HempelJ:3man}) are satisfied and there exists an embedding $\tau \colon \bB^2 \to M$ so that $\tau|\partial \bB^2 =\alpha$.

To unknot quantitatively, we follow the proof of Proposition \ref{prop:chi} almost verbatim. Let $\kappa = \dist( h(\bS^1), \partial T)$ and $n_0 (\eta)$ be as in the proof of Proposition \ref{prop:chi}. Then there exists a polygonal Jordan path $\gamma =[w_0,w_1]\cup [w_1,w_2]\cup \cdots \cup [w_{n-1},w_0]$ with  vertices $w_i \in (\kappa/(1000\sqrt{3+m}n_0))\Z^{3+m}$ so that $\max_{w\in \gamma}\dist(w, h(\bS^1)) \le \kappa/20$ and $\dist(\gamma, \partial h(T)) \ge 19\kappa/20$. Therefore $\gamma$ is PL-isotopic to $h(\{0\}\times \bS^1)$ in $h(T)$. We may now fix a scaled $L_0=L_0(\eta,m,n_0)$-bilipschitz, therefore $\eta'$-quasisymmetric homeomorphism $\chi\colon \R^3 \to \R^3$ so that $\chi(h(\bS^1))=\partial \bB^2$ as in the proof of Proposition \ref{prop:chi}. Conditions (1) and (2) in the statement now follow by quasisymmetry.
\end{proof}

\section{Growth and a modulus estimate for walls}
\label{sec:modulus}

The main result in this section is a lower estimation of the conformal modulus of a  $m$-wall family, which corresponds partly to the first claim of \cite[Proposition 4.5]{HeinonenJ:Quansa}.

\begin{proposition}
\label{prop:lower_mod_estimate}
Suppose  $(\R^3/G,\cX, (\fC,\fA,\fW), \theta, d_\lambda)$ is a Semmes space.
Let $k\ge 0$, $\mathcal{Y}$ be a collection of cubes-with-handles in $\cC(X_k)$ of positive genus, and $Y$ be their union. Let $m\ge 0$, then the conformal modulus of $m$-walls
\begin{equation}
\label{eq:lower_mod_estimate}
\Mod_{\frac{3+m}{1+m}}\left( \hat \longi^m(Y,\cX ,a)\right) \ge C \left( \left( \# \mathcal{Y}\right) \left(\frac{a}{\lambda^k}\right)^m  \right)^{1-\frac{3+m}{1+m}}
\end{equation}
for every $a>0$ and a constant $C=C(\fC,\fW,\fA,m)>0$.
\end{proposition}

To obtain the estimate, we first fix a collection of cubes-with-handles $\mathscr{H}= \{\mathsf H_0,\mathsf H_1,\ldots \}$, one for each genus, and a special family of longitudes in each $\mathsf H_g$ as follows.

Let $Q(x,r) = [x_1-r,x_1+r]\times [x_2-r,x_2+r] \subset \R^2$ for $x=(x_1,x_2)\in \R^2$ and $r>0$ and denote the origin of $\R^2$ by $\mathrm{O}$.
Set $\mathsf H_0=Q(\mathrm{O},1)\times [0,1]$.

For each $g>0$, fix points $\{p_1,\ldots, p_g\}$ in $Q(\mathrm{O},1-1/(10 g))$ having pair-wise distance at least $1/(20 g)$. Let $\Omega_g = Q(\mathrm{O},1)\setminus \cup_i (\interior Q(p_i,1/(100 g)))$ and $\mathsf H_g = \Omega_g \times [0,1]$. Then $\mathsf H_g$ is a cube-with-$g$-handles.
For every $0\le t \le 1/(100g)$ and every $0<s<1$, fix a PL $1$-cycle $\sigma^g_{t,s}$ in $\mathsf H_g$ having
\[
\left( \partial Q(\mathrm{O},1-t) \cup \partial Q(p_1,\frac{1}{100g}+t)\cup \cdots \cup \partial Q(p_g,\frac{1}{100g}+t)\right)\times \{s\}
\]
as a carrier.

\begin{lemma}
\label{lemma:horizontal_cycles}
Given $g>0$, the $1$-cycles $\sigma^g_{t,s}$ defined above are longitudes of $\mathsf H_g$ for all $0\le t \le 1/(100g)$ and $0<s< 1$. Moreover, if $\omega $  is a $2$-manifold in $\bB^2$ and
 $\zeta \colon (\omega,\partial \omega)\to (\mathsf H_g,\partial \mathsf H_g)$ is virtually interior essential, then $\zeta(\omega) \cap |\sigma^g_{t,s}| \ne \emptyset$.
\end{lemma}

\begin{proof}
We denote $\Omega =\Omega_g$, $\mathsf H=\mathsf H_g=\Omega_g \times [0,1]$, and $\sigma_{t,s}=\sigma_{t,s}^g$.

To show that $\sigma_{t,s}$ is a longitude, let $\alpha\colon \bS^1 \to \partial \mathsf H$ be a meridian of $\mathsf H$ and $\phi\colon (\bB^2,\partial \bB^2) \to (\mathsf H,\alpha)$ a map.  We claim that $\phi(\bB^2) \cap |\sigma_{t,s}| \ne \emptyset$.

Consider first the case $t=0$. Suppose toward contradiction that there is an $s\in (0,1)$ so that $\phi(\bB^2) \cap |\sigma_{0,s}|= \emptyset$. After postcomposing $\phi$ with a homeomorphism from $\mathsf H\setminus |\sigma_{0,s}|$ onto $\mathsf H \setminus (\partial \Omega\times [0,1])$, we may assume that $\phi \colon (\bB^2, \partial \bB^2) \to (\mathsf H, \Omega\times \{0,1\})$. Suppose that $\phi(\partial \bB^2)\subset \Omega \times \{1\}$. Since $\phi$ is interior essential, $\phi(\partial \bB^2)$ is not trivial in $\pi_1(\Omega\times \{1\})$. Hence $\phi(\partial \bB^2)$ is not trivial in $\pi_1(\Omega \times [0,1]) = \pi_1(\mathsf H)$. Since $\phi(\bB^2)\subset \mathsf H$, this is a contradiction.

We next prove the second statement in the lemma for $t=0$. Let $\zeta\colon (\omega,\partial \omega) \to (\mathsf H,\partial \mathsf H)$ be the map given. Since $\zeta$ is virtually interior essential, it has an extension $\zeta' \colon D_\omega \to \mathsf H$ satisfying $\zeta'(D_\omega \setminus \omega)\subset \partial \mathsf H$, where $D_\omega$ is the $2$-cell in $\bB^2$ with $ \omega \subset D_\omega $ and  $\partial D_\omega \subset \partial \omega$; see Section \ref{sec:VIE}. After applying a homotopy to $\zeta'$ which leaves $\zeta'|\partial D_\omega$ fixed, we may assume that $\zeta'( D_\omega) \cap \partial \mathsf H=\zeta'(\partial D_\omega) \cap \partial \mathsf H$. Since $\zeta'$ is interior essential, $\zeta'(D_\omega) \cap |\sigma_{0,s}| \ne \emptyset$ for all $s\in (0,1)$. Since $\zeta'(\partial D_\omega) \subset\zeta(\partial\omega)$, $\zeta(\partial\omega) \cap |\sigma_{0,s}| \ne \emptyset$. Since $\zeta|\partial \omega  = \zeta'|\partial \omega$, the claim follows.

We now verify $\phi(\bB^2) \cap |\sigma_{t,s}|\ne \emptyset$ in the case $0< t \le 1/(100g)$ for a given $s\in (0,1)$. Let $\Omega_t$ be the planar closed region with boundary $|\sigma_{t,0}|$, and
$\mathsf H_{t,s}=\Omega_t\times [s/2,(1+s)/2]$ a cube-with-$g$-handles contained in $\mathsf H$. Note that $|\sigma_{t,s}| \subset \partial \mathsf H_{t,s} $. Since
$\mathsf H \setminus \mathsf H_{t,s}$ is a regular neighborhood of $\partial \mathsf H$ in $\mathsf H$,
$\phi^{-1}\mathsf H_{t,s}$ contains a component, say $\omega'$, on which $\phi|\omega'\colon (\omega',\partial \omega')\to (\mathsf H_{t,s},\partial \mathsf H_{t,s})$ is virtually interior essential. Then, by the argument above, $\phi (\omega') \cap |\sigma_{t,s}| \ne \emptyset$ and hence
$\phi(\bB^2)\cap |\sigma_{t,s}| \ne \emptyset$. This proves the claim.

The second statement in the case $t>0$ follows from the same argument for $t=0$.
\end{proof}

\begin{proof}[Proof of Proposition \ref{prop:lower_mod_estimate}]

By passing to a bilipschitz equivalent metric if necessary, we may assume that $d_\lambda = d_\theta$, where $\theta$ is a $\lambda$-modular embedding $\R^3/G \to \R^n$.

As a preliminary step, we fix for every $\cxi=(A_\cxi,B_\cxi) \in \fC$, a PL-homeo\-morphism $\xi_\cxi \colon A_\cxi \to \mathsf{H}_{g_\cxi}$, where $g_\cxi$ is the genus of $A_\cxi$. Since $\fC$ is finite, the mappings  $\xi_\cxi $ are uniformly bilipschitz, and there exists $t_{\fC}\in (0,1/(100g))$ so that
\[
\xi_\cxi(B_\cxi)\cap |\sigma^{g_\cxi}_{t,s}| = \emptyset
\]
for every $0\le t \le t_{\fC}$, every $0<s<1$, and $\cxi\in \fC$.

We fix a special family of longitudes for each $\mathsf H_g$ in $\mathscr{H}$ and an induced family of longitudes on $\cX$ as follows.
For each $g>0$, let
\[
\longi(\mathsf H_g, \mathscr{H} ) = \{ \sigma^g_{t,s} \colon 0\le t \le r_\fC,\ 0< s< 1\};
\]
and for $g=0$, define $\longi(\mathsf H_0, \mathscr{H} ) =\emptyset$.

By Lemma \ref{lemma:horizontal_cycles}, these $1$-cycles are longitudes of $\mathsf H_g$.
Define for every $H\in \cC(\cX)$ an induced family of longitudes of $H$ by
\[
\longi(H,\cX, \mathscr{H}) = \{ \varphi_H^{-1}\circ \xi_{\cxi_H}^{-1} (\sigma) \colon \sigma \in \longi(\mathsf H_{g_H}, \mathscr{H} )\},
\]
where $g_H$ is the genus of $H$ and $\varphi_H \colon H^\diff \to \cxi_H^\diff$ is the chart map in $\fA$.

By \eqref{eq:longi-sum}, every $1$-cycle in $Y$ of the form
\[
\tau_{t,s} = \sum_{H \in\mathcal Y} \varphi_H^{-1}\circ \xi_{\mathsf c_{H}}^{-1}(\sigma^{g_H}_{t,s}),
\]
$0\le t \le t_{\fC}$ and $0< s< 1$, is a longitude of $Y$. Set
\[
\longi(Y,\cX, \mathscr{H})=\{\tau_{t,s}\colon 0\le t \le t_{\fC} \,\text{and}\, 0< s< 1  \},
\]
and
\[
\longi^m(Y,\cX, \mathscr{H};a)=\{|\tau|\times [-a,a]^m \colon \tau \in \longi(Y,\cX, \mathscr{H})\}
\]  
the collection of corresponding $m$-walls over $Y$ of height $a$.

Since $\longi^m(Y,\cX, \mathscr{H};a)\subset \longi^m(Y,\cX;a)$, it suffices to show that the estimate \eqref{eq:lower_mod_estimate} holds for the surface family $\hat \longi^m(Y,\cX, \mathscr{H};a)=(\pi_G\times \id)\left(\longi^m(Y,\cX,\mathscr H; a)\right)$.

Before continuing, we observe that, since the embedding $\theta:\R^3/G\to\R^n$ is $\lambda$-modular, there exists $L=L(\fC,\fA,\fW)\ge 1$ so that  for every $k\ge 0$ and every $H\in \cC(X_k)$, the map
\[
\zeta_H=\pi_G \circ \varphi_H^{-1} \circ \xi_{\cxi_H}^{-1}|\xi_{\cxi_H}(\cxi_H^\diff) \colon \xi_{\cxi_H}(\cxi_H^\diff) \to \pi_G( H^\diff)
\]
and its extension $\, \xi_{\cxi_H}(\cxi_H^\diff)\times \R^m \to (\pi_G(H)\times\R^m,d_{\lambda,m})$ defined by
\[
\zeta_H\colon (x,z)\mapsto (\pi_G \circ \varphi_H^{-1} \circ \xi_{\cxi_H}^{-1}(x),\lambda^k z)
\]
are $(\lambda^k,L)$-quasisimilarities.

In the following estimation of the modulus of surface families, we denote by $\haus^\beta_\delta$ and by $\haus^\beta_e$ the $\beta$-dimensional Hausdorff measures with respect to $d_{\lambda,m}$ and the Euclidean metric, respectively.

Suppose that $\rho$ is an admissible Borel function for $\hat \longi^m(Y,\cX, \mathscr{H};a)$ on $\R^3/G\times \R^m$, that is,
\[
\int_{\pi_G(|\tau_{t,s}|)\times [-a,a]^m} \rho \dhaus^{1+m}_\delta \ge 1
\]
for every $\tau_{t,s}\in \longi(Y,\cX,\mathscr H)$. We assume as we may that $\rho$ is supported in $\pi_G(Y\setminus X_{k+1})\times [-a,a]^m$.

We have, for every $0\le t\le t_\fC$ and every  $0< s< 1$, that
\begin{eqnarray*}
&& \sum_{H\in\mathcal{Y}}(L\lambda^k)^{1+m}\int_{ |\sigma^{g_H}_{t,s}| \times [-\lambda^{-k} a, \lambda^{-k} a]^m} \rho\circ \zeta_H \dhaus^{1+m}_e \\
&&\qquad \ge \sum_{H\in\mathcal{Y}} \int_{\zeta_H
\left(|\sigma^{g_H}_{t,s}| \times  [-\lambda^{-k} a, \lambda^{-k} a]^m\right)} \rho \dhaus^{1+m}_\delta \\
&&\qquad = \int_{\pi_G(|\tau_{t,s}|)\times [-a,a]^m} \rho \dhaus^{1+m}_\delta \ge 1.
\end{eqnarray*}
Thus
\begin{equation}
\label{eq:H_i}
\begin{split}
& \sum_{H\in\mathcal{Y}} \int_{\mathsf{H}_{g_H}\times [-\lambda^{-k}a,\lambda^{-k}a]^m} \rho \circ \zeta_H \dhaus^{3+m}_e \\
& \quad \ge C  \int_{[0,t_\fC]\times [0,1]} \left(\sum_{H\in \mathcal{Y}}\int_{|\sigma^{g_H}_{t,s}|\times [-\lambda^{-k}a, \lambda^{-k}a]^m} \rho \circ \zeta_H \dhaus^{1+m}_e \right)\dhaus^2_e \\
& \quad \ge C t_\fC \lambda^{-k(1+m)},
\end{split}
\end{equation}
where $C$ depends only $(\fC,\fA,\fW)$.

Let $p=(3+m)/(1+m)$. Then, by \eqref{eq:H_i},
\begin{equation}
\label{eq:cH_i}
\begin{split}
&\sum_{H\in\mathcal{Y}} \int_{\mathsf{H}_{g_H}\times [-\lambda^{-k}a,\lambda^{-k}a]^m} (\rho \circ \zeta_H)^p \dhaus^{3+m}_e \\
&\quad \ge \left(\sum_{H\in \mathcal{Y}} \haus^{3+m}_e(\mathsf{H}_{g_H}\times [-\lambda^{-k}a,\lambda^{-k}a]^m )\right)^{1-p} \\
&\qquad \times  \left( \sum_{H\in\mathcal{Y}}  \int_{\mathsf{H}_{g_H}\times [-\lambda^{-k}a,\lambda^{-k}a]^m} \rho \circ \zeta_H \dhaus^{3+m}_e \right)^p \\
&\quad \ge C (\# \mathcal{Y})^{1-p} (\lambda^{-k}a)^{m(1-p)} \lambda^{-k(1+m)p} \\
&\quad = C (\# \mathcal{Y})^{1-p} \lambda^{-k(m+p)}a^{m(1-p)},
\end{split}
\end{equation}
where $C>0$ depends only on $m$ and $(\fC,\fA,\fW)$.

Since $\zeta_H$ is a $(\lambda^k,L)$-quasisimilarity, $\zeta_H^{-1}$ is $L\lambda^{-k}$-Lipschitz. By a change of variables,
\begin{eqnarray*}
&& \int_{\pi_G(H)\times [-a,a]^m} \rho^p \dhaus^{3+m}_\delta = \int_{\pi_G(H)\times [-a,a]^m} (\rho\circ \zeta_H)^p \circ \zeta_H^{-1} \dhaus^{3+m}_\delta \\
&& \qquad\ge \left(\frac{\lambda^k}{L}\right)^{3+m}  \int_{\mathsf{H}_{g_H}\times [-\lambda^{-k}a,\lambda^{-k}a]^m} (\rho\circ \zeta_H)^p \dhaus^{3+m}_e,
\end{eqnarray*}
for every $H\in \mathcal{Y}$. Since $\rho$ is supported in $\pi_G(Y\setminus X_{k+1})\times [-a,a]^m$, we have
\begin{eqnarray*}
\label{eq:X_k}
&&\int_{\R^3/G\times \R^m} \rho^p \dhaus^{3+m}_\delta
=\int_{\pi_G(Y\setminus X_{k+1})\times [-a,a]^m} \rho^p \dhaus^{3+m}_\delta \\
&&\qquad \ge (\lambda^k/L)^{3+m} \sum_{H\in\mathcal{Y}} \int_{\mathsf{H}_{g_H}\times [-\lambda^{-k}a,\lambda^{-k}a]^m} (\rho\circ \zeta_H)^p \dhaus^{3+m}_e \\
&&\qquad \ge C (\# \mathcal{Y})^{1-p} \lambda^{k(3+m)}a^{m(1-p)} \lambda^{-k(m+p)} \\
&&\qquad = C \left( (\# \mathcal{Y}) (a/\lambda^k)^m \right)^{1-p},
\end{eqnarray*}
where $C$ depends only on $m$ and  $(\fC,\fA,\fW)$. The claim follows.
\end{proof}

\section{A necessary condition for quasisymmetric parametrization}
\label{sec:abstract_main_thm}

The existence of quasisymmetric parametrization of  $(\R^3/G\times \R^m,d_{\lambda,m})$ by $\R^{3+m}$ requires a balance among the growth, circulation and the scaling factor of the Semmes space. We prove this result in this section.

\begin{theorem}
\label{thm:general}
Let  $(\R^3/G,\cX, (\fC,\fA,\fW), \theta, d_\lambda)$ be a Semmes space, and let $m\ge 0$.
Assume that $\cX$ has order of growth at most $\growth$ and order of circulation at least $\omega$. Suppose that there exists a quasisymmetric homeomorphism $(\R^3/G \times \R^m,d_{\lambda,m})  \to \R^{3+m}$. Then
\[
\lambda^m \omega^{\frac{3+m}{2}} \le \growth.
\]
\end{theorem}

We obtain now Theorem \ref{thm:G_first} as a corollary.
\begin{proof}[Proof of Theorem \ref{thm:G_first}]
Since $\omega^3 > \gamma^2 \ge 1$, we may fix $\lambda$ so that $\omega^{-1/2} <  \lambda < \gamma^{-1/3}$. On the one hand, $\lambda \gamma^3 < 1$, so $(\R^3/G\times \R^m, d_{\lambda,m})$ is Ahlfors $(3+m)$-regular for all $m\ge 0$. On the other hand,
\[
\lambda^m \omega^{\frac{3+m}{2}} > \gamma,
\]
so there are no quasisymmetric homeomorphisms $(\R^3/G\times \R^m,d_{\lambda,m}) \to \R^{3+m}$ for any $m\ge 0$. The linear local contractibility follows from Proposition \ref{prop:LLC}.
\end{proof}

To combine modulus estimates in Sections \ref{sec:winding} and \ref{sec:modulus}, we need a one-sided comparison between the modulus of a wall family and the modulus of a  quasisymmetric image of the same family. The proof in \cite[Proposition 4.1]{HeinonenJ:Quansa} for the case of the Whitehead continuum applies almost verbatim  to Semmes spaces $\R^3/G\times \R^m$; we omit the details.

\begin{proposition}
\label{prop:quasi-invariance}
Suppose $f\colon \R^3/G\times \R^m \to \R^{3+m}$ is an $\eta$-quasisymmetric homeomorphism, and $Y$ is the union of a nonempty subcollection of cubes-with-handles in $\cC(X_k)$ for some $k \geq 1$.
Then there exists $C=C(\eta)>0$ so that
\[
\Mod_{\frac{3+m}{1+m}} (\hat\longi^m(Y,\cX ,a)) \le C\Mod_{\frac{3+m}{1+m}} f(\hat \longi^m(Y,\cX ,a))
\]
for $a>0$.
\end{proposition}

\begin{proof}[Proof of Theorem \ref{thm:general}]
Let $\alpha_0\in \meri(\cX)$ be a meridian as in Definition \ref{def:circulation}. In view of Lemma \ref{lemma:qs_meridians}, when considering a lower bound of the circulation, we may restrict to a subcollection $\meri$ of $ \meri_\fA(\cX;\alpha_0)$ consisting of uniformly quasisimilar meridians.

Since order of circulation of $\cX$ is at least $\omega$, there exists $C>0$ so that for each $\ell>1$, there exist $k,k'\ge 0$, $k'-k\ge \ell$, $H\in \cC(X_k)$, and $\alpha \colon \bS^1\to \partial H$ in $\meri$ so that
\[
\wind(X_{k'}\cap H,\alpha,H)\ge C\omega^{k'-k}.
\]

Let $f$ be an $\eta$-quasisymmetric mapping $(\R^3/G \times \R^m,d_{\lambda,m})  \to \R^{3+m}$. From Lemma \ref{lemma:qs_meridians}, Theorem \ref{thm:key}, Proposition \ref{prop:lower_mod_estimate}, and Proposition \ref{prop:quasi-invariance} it follows that
\[
\left(\#\cC(X_{k'}\cap H) (A\lambda^k)^m  \lambda^{-k' m} \right)^{1-p} \le C \left( \frac{1}{\wind(X_{k'}\cap H,\alpha,H)}\right)^p,
\]
where $p=(3+m)/(1+m)$, $C>0$ depends only on $(\fC,\fA,\fW)$, $\lambda$ and $m$, and $A$ is the constant defined in Theorem \ref{thm:key}. Since the order of growth of  $\cX$ is at most $\growth$, there exists $C=C(\fC,\fW,\fA, m,\eta, \alpha_0)\ge 1$ such that
\begin{eqnarray*}
\omega^{(k'-k) p} &\le& C \left( \wind(X_{k'}\cap H,\alpha,H)\right)^p \\
&\le& C \left( \#\cC(X_{k'}\cap H) (A\lambda^k)^m \lambda^{-k' m} \right)^{p-1} \\
&\le& C \left({\growth}^{k'-k} \lambda^{km}  \lambda^{-k' m} \right)^{p-1} \\
&\le& C \lambda^{(k-k') m(p-1)} {\growth}^{(k'-k)(p-1)},
\end{eqnarray*}
as $\ell\to \infty$. Thus
\[
\lambda^m \omega^{\frac{p}{p-1}} \le C^{\frac{1}{k'-k}} \growth \le C^{1/\ell} \growth,
\]
as $\ell\to \infty$. The claim now follows.
\end{proof}

\section{Local parametrizability}
\label{sec:local_theory}

In this section we consider a local version of Theorem  \ref{thm:general} that compares growth and circulation in parallel along a sequence of blocks of $(X_k)$ targeting at a point $x\in \pi_G(X_\infty)$. Theorem \ref{thm:General_local} below may be used to detect the quasisymmetric non-parametrizability of some Semmes spaces, unnoticed by Theorem \ref{thm:general}.

Let $(\R^3/G,\cX, (\fC,\fA,\fW))$ be a decomposition space of finite type. Given $x\in \pi_G(X_\infty)$, we denote by $(H_k(x))$ the unique sequence in $\cC(\cX)$ for which $x\in \pi_G(H_k(x))$ and $H_k(x)\in \cC(X_k)$ for every $k\ge 0$. We call $(H_k(x))$ the \emph{branch of $\cX$ at $x$}.

We denote by $\meri(\cX,x)$ the collection of all meridians on the branch $(H_k(x))$ of $\cX$ at $x$. Given a meridian $\alpha_0\in \meri(\cX,x)$, we denote
\[
\meri_\fA(\cX,x;\alpha_0) = \meri_\fA(\cX;\alpha_0)\cap \meri(\cX,x),
\]
where $\meri_\fA(\cX;\alpha_0)$ is the collection of meridians in $\cX$ related to $\alpha_0$ by $\fA$ defined in \eqref{eq:meridainA}.

\begin{definition}
\label{def:growth+circulation_local}
At a point $x\in \pi_G(X_\infty)$, we say that  \emph{the order of circulation of $\cX$ is at least $\omega\ge 0$ and the order of growth of $\cX$  is at most $\gamma \leq \infty$ concurrently}
if the following holds.
There exists a meridian $\alpha_0\in \meri(\cX,x)$ and constants $C_1, C_2 >0$ such that for every $\ell\ge 0$ there exist $k'>k\ge 0$ with $k'-k>\ell$ and a meridian $\alpha\colon \mathbb S^1\to \partial H_k(x)$ in $\meri_\fA(\cX,x;\alpha_0)$ satisfying
\[
\wind(X_{k'}\cap H_k(x),\alpha,H_k(x)) \ge C_1 \,\omega^{k'-k},
\]
and
\[
\# \cC(X_{k'} \cap H_k(x)) \leq C_2 \, \gamma^{k'-k}.
\]
\end{definition}

\begin{remark}
By mixing the steps in the construction of the Whitehead continuum and of an Antoine's necklace, we may build a defining sequence $\cX$ having the following property. Sequence $\cX$ has the order of growth at most $\gamma$ and the order of circulation at least $\omega$  for which $\omega^3<\gamma^2$. Nevertheless, at each point $x\in \pi_G(X_\infty)$,
 a concurrent pair $(\omega(x),\gamma(x))$, as defined in  \ref{def:growth+circulation_local}, may be chosen so that $\omega^3(x)>\gamma^2(x)$.
\end{remark}

\begin{theorem}
\label{thm:General_local}
Let  $(\R^3/G,\cX, (\fC,\fA,\fW), \theta, d_\lambda)$ be a Semmes space and $x\in \pi_G(X_\infty)$. Suppose at $x$, $\cX$ has the order of growth at most $\gamma(x)$ and the order of circulation at least $\omega(x)$  concurrently. If for some $\delta>0$ and a neighborhood $U$ of $x$ there exists a quasisymmetric embedding $(U\times (-\delta,\delta)^m,d_{\lambda,m}) \to \R^{3+m}$ then
\begin{equation}\label{eq:omega_lambda_local}
\lambda^m \omega(x)^{\frac{3+m}{2}} \le \gamma(x).
\end{equation}
\end{theorem}

\begin{proof}[Sketch of the proof]
The only essential modification to the proof of Theorem \ref{thm:general} is related to the application of a local version of Theorem \ref{thm:key}.

Let $U\subset \R^3/G$ be an open set containing $x$ and $f\colon (U\times (-\delta,\delta)^m,d_{\lambda,m}) \to \R^{3+m}$ be a quasisymmetric embedding. We may fix a ball $B^{3+m}(f(x),r_0)$ in $f(U\times (-\delta,\delta)^m)$, and an integer $k_0>0$ so that $\pi_G(H_{k_0}(x))\times [-\lambda^{k_0},\lambda^{k_0}]^m \subset U\times (-\delta,\delta)^m$.
Under these choices of parameters, the quasisymmetric unknotting of images of meridians (Proposition \ref{prop:chi} and Proposition \ref{prop:chi2}) can be performed in $B^{3+m}(f(x),r_0)$. Thus the proof of Theorem \ref{thm:key} can be carried over to the defining sequence $(X_k)_{k\ge k_0}$. We omit the straightforward modifications of Theorem \ref{thm:key} and the related lemmas in Sections \ref{sec:winding}, \ref{sec:tubes}, and \ref{sec:modulus}.
\end{proof}

\section{Singular fibers of Semmes spaces}
\label{sec:fibers}

In this section, we consider an application of Theorem \ref{thm:General_local} to a question on the quasisymmetric equivalence of product spaces $(\R^3/G\times \R^1,d_{\lambda,1})$ for $0<\lambda <1$.

Let $(\R^3/G,(X_k), d_\lambda)$ be a Semmes space and $m\geq 0$. A point $x\in \R^3/G$ is said to be (\emph{quasisymmetrically}) \emph{$\lambda$-singular} of index $m$ if there is no quasisymmetric homeomorphism from any neighborhood  of $(x,0)$ in $\R^3/G\times \R^m$ to a subset of $\R^{3+m}$; in this case,  $\{x\}\times \R^m\subset \R^3/G\times \R^m$ is called a \emph{singular fiber}.  We denote by $\sing_{\lambda,m}(\R^3/G)$ the set of $\lambda$-singular points of index $m$ and note that $\sing_{\lambda,m}(\R^3/G)$ is a closed subset of $\pi_G(X_\infty)$.

A quasisymmetric map $(\R^3/G\times \R^m, d_{\lambda,m})\to (\R^3/G'\times \R^m, d_{\mu,m})$ between two Semmes spaces induces a homeomorphism from $\sing_{\lambda,m}(\R^3/G)\times \R^m$ to $\sing_{\mu,m}(\R^3/G')\times \R^m$. For $m=1$, the induced map is bilipschitz on non-isolated fibers.
\begin{theorem}
\label{theorem:cantor-bilip}
Let $(\R^3/G,(X_k), d_\lambda)$ and $(\R^3/G',(Y_k), d_\mu)$ be two Semmes spaces, $m\ge 1$, and let $f\colon (\R^3/G\times \R,d_{\lambda,m}) \to (\R^3/G' \times \R,d_{\mu,m})$ be an $\eta$-quasisymmetric map. Then
\[
f(\sing_{\lambda,m}(\R^n/G)\times \R) = \sing_{\mu,m}(\R^3/G')\times \R.
\]
Furthermore, if $m=1$ and  $A$ is the collection of accumulation points in $\sing_{\lambda,1}(\R^3/G)$, then $f|A \times \R$ is $L_0$-bilipschitz for some $L_0\geq 1$.
\end{theorem}

By quasisymmetry, the bilipschitz rigidity of the singular fibers yields the nesting of corresponding branches. We formalize this observation in the next theorem. As an application of this result, we obtain the quasisymmetric  inequivalence between $(\R^3/\Bd\times \R^1, d_{\lambda,1})$ and $(\R^3/\Bd\times \R^1,d_{\lambda',1})$ for $\lambda\ne \lambda'$ and $1/2<\lambda'<1$; see Theorem \ref{thm:Bing_stab} in the Introduction. We postpone this discussion to Section \ref{sec:Bing_Double}.

\begin{theorem}
\label{thm:sb_point}
Let $0<\lambda<\mu<1$ and $(\R^3/G,(X_k), d_\lambda)$ and $(\R^3/G',(Y_k), d_\mu)$ be Semmes spaces and $f\colon (\R^3/G\times \R,d_{\lambda,1})\to (\R^3/G'\times \R,d_{\mu,1})$ be an $\eta$-quasisymmetric map. Let $x\in \sing_{\lambda,1}(\R^3/G)$ be an accumulation point. Then, for any $\ell > 0$, there exists  $k_0=k_0(\eta,d_\lambda,d_\mu,\ell)>0$ so that
\[
f(\pi_G(H_k(x))\times \R) \subset \pi_{G'}(H_{k+\ell}(y))\times \R
\]
for all $k\ge k_0$, where $y= \mathrm{proj}\,f(x)$ is the image of $f(x)$ under the projection $\mathrm{proj} \colon \R^3/G'\times \R \to \R^3/G'$, and $(H_k(x))$ and $(H_k(y))$ are the branches of $\cX=(X_k)$ and $\mathcal Y=(Y_k)$ at $x$ and $y$ respectively.
\end{theorem}

We begin with some auxiliary results on lines in metric spaces. Let $(X,d)$ be a metric space. We say that $L\subset X$ is a \emph{line} if $L$ is isometric to $\R$. We say that a line $L$ is \emph{parallel to a line $L'$} if there exists $a>0$ so that $\dist(p,L')=a$ for every $p\in L$; in this case, $L'$ is also parallel to $L$ and $\dist(L,L')=a$.

\begin{lemma}
\label{lemma:parallel}
Suppose $f\colon X\to Y$ is an $\eta$-quasisymmetric map between two metric spaces which maps two given parallel lines $L$ and $L'$ to parallel lines $fL$ and $fL'$. Then there exists $C=C(\eta)>1$ so that
\[
\frac{1}{C} \frac{\dist_Y(fL,fL')}{\dist_X(L,L')} \le \frac{\dist_Y(f(p),f(q))}{\dist_X(p,q)} \le C \frac{\dist_Y(fL,fL')}{\dist_X(L,L')}
\]
for all $p,q \in L$ with $\dist_X(p,q)\ge \dist_X(L,L')$.
\end{lemma}

\begin{proof}
Suppose that points $p, q \in L$ have distance $\dist_X(p,q)\geq \dist_X(L,L')$. Since $L$ is a line, there exist points $p=p_0,\ldots, p_k = q$ on $L$ so that
\[
 \dist_X(L,L')\leq \dist_X(p_i,p_{i-1})\leq 2\dist_X(L,L')
\]
for all $1\le i \le k$. Since lines $L$ and $L'$ are parallel, lines $fL$ and $fL'$ are parallel, and $f$ is $\eta$-quasisymmetric, there exists $C_0=C_0(\eta)>1$ so that
\[
\frac{1}{C_0}\, \dist_Y(fL,fL')\leq   \dist_Y(f(p_i),f(p_{i-1}))  \leq C_0\,\dist_Y(fL,fL')
\]
for all $1\le i \le k$. Since $L$ and $fL$ are lines, the claim follows by summing.
\end{proof}

\begin{corollary}
\label{cor:bilip}
Let $f\colon X\to Y$ be an $\eta$-quasisymmetric map between two metric spaces which maps two given parallel lines $L$ and $L'$ to two parallel lines $fL$ and $fL'$.
Suppose, in addition, $f$ maps a sequence $(L_i)$ of lines parallel to $L$ tending to $L$ to a sequence $(fL_i)$ of lines parallel to $fL$ tending to $fL$. Then there exists
$C=C(\eta)>1$ so that
\[
\frac{1}{C} \frac{\dist_Y(fL,fL')}{\dist_X(L,L')} \le \frac{\dist_Y(f(p),f(q))}{\dist_X(p,q)}\le C \frac{\dist_Y(fL,fL')}{\dist_X(L,L')}
\]
for $p,q\in L$.
\end{corollary}

\begin{proof}
Calculations using Lemma \ref{lemma:parallel} show that there exists a constant $C >1$ so that
\[
\frac{1}{C}\frac{\dist_Y(fL,fL')}{\dist_X(L,L')} \le \frac{\dist_Y(fL,fL_i)}{\dist_X(L,L_i)} \le C\frac{\dist_Y(fL,fL')}{\dist_X(L,L')}
\]
for every $i\ge 0$.

Given $p,q\in L$, we fix a line $L_i$ so that $\dist_X(L,L_i)<\dist_X(p,q)$. The claim now follows by applying Lemma \ref{lemma:parallel} again.
\end{proof}

\begin{proof}[Proof of Theorem \ref{theorem:cantor-bilip}]
The first claim is clear.

Suppose next that  $f\colon \R^3/G\times \R \to \R^3/G' \times \R$ is quasisymmetric and $x\in \sing_{\lambda,1}(\R^3/G)$ is an accumulation point.
We choose a point $x'$ in $\sing_{\lambda,1}(\R^3/G)$ so that
$d_{\lambda,1}(x,x')\ge \frac{1}{2}\diam \sing_{\lambda,1}(\R^3/G)$, and let $L$ and $L'$ be the singular fibers $\{x\}\times \R$ and $\{x'\}\times \R$, respectively. In view of Corollary \ref{cor:bilip},
\[
\frac{1}{C_0} \le \frac{\dist_{\mu,1}(f(p),f(q))}{\dist_{\lambda,1}(p,q)}\le C_0
\]
for all $p=(x,s)$ and $q=(x,t)$ in the singular fiber $\{x\}\times \R$, where $C_0> 1$ depends only on the data and not on $x$.

Let $p=(x,s), w=(y,r)\in \sing_{\lambda,1}(\R^3/G)\times \R$, and set $q=(x,t)$, where $t$ is defined by $t=r+d_{\lambda,1}(x,y)$ if $r\ge s$ and by $t=r-d_{\lambda,1}(x,y)$ if $r<s$. So, $\dist_{\lambda,1}(p,w)= \dist_{\lambda,1}(p,q)$.  By $\eta$-quasisymmetry,
\[
\frac{1}{\eta(1)} \le \frac{\dist_{\mu,1}(f(p),f(w))}{\dist_{\mu,1}(f(p),f(q))} \le \eta(1).
\]
Hence
\[
\frac{1}{C_0\eta(1)} \le \frac{\dist_{\mu,1}(f(p),f(w))}{\dist_{\lambda,1}(p,w)} \le C_0\eta(1).
\]
The second claim now follows.
\end{proof}

\begin{proof}[Proof of Theorem \ref{thm:sb_point}]

By properties of the Semmes metric (see Section \ref{sec:metric properties}), there exists $C_1>1$ and $C_2>1$ so that
\[
\diam_{d_{\lambda,1}}(\pi_G(T)) \le C_1 \lambda^{\lvl(T)}
\]
and
\[
\dist_{d_{\mu,1}} (\sing_{\mu,1}(\R^3/G')\cap \pi_{G'}(T'), \partial \pi_{G'}(T')) \ge \frac{1}{C_2} \mu^{\lvl(T')}
\]
for every $T\in \cC(\cX)$ and $T'\in\cC(\mathcal Y)$.

Let $L_0\geq 1$ be the constant in Theorem \ref{theorem:cantor-bilip}.
Since $\lambda < \mu$, we may fix $\ell>0$ and $k_0>0$ so that
\[
L_0 C_1 C_2 \eta(1) \lambda^k < \mu^{k+\ell}
\]
for $k\ge k_0$.

Since  $f(\{x\}\times \R)= \{y\} \times \R$,  $f(\{x\}\times \R) \subset f(\pi_G(H_k(x))\times \R) \cap (\pi_{G'}(H_{k'}(y))\times \R)$ for $k, k'\geq 1$. From Theorem \ref{theorem:cantor-bilip} and the $\eta$-quasisymmetry, it follows that for any $k\ge k_0$,
\begin{eqnarray*}
&& \dist_{\haus}(f(\partial \pi_G(H_k(x))\times \R), f(\{x\}\times \R)) \\
&&\qquad \le L_0 \eta(1) \diam_{d_{\lambda,1}} \pi_G(H_k(x)) \\
&&\qquad \le L_0 C_1 \eta(1) \lambda^k \\
&&\qquad \le \mu^{k+\ell}/C_2 \\
&&\qquad \le \dist_{d_{\lambda,1}}(\partial \pi_{G'}(H_{k+\ell}(y))\times \R, \{y\}\times \R),
\end{eqnarray*}
where $\dist_{\haus}(f(\partial\pi_G(H_k(x))\times \R), f(\{x\}\times \R))$ is the Hausdorff distance of $f(\partial \pi_G(H_k(x))\times \R)$ and $f(\{x\}\times \R)$ in Semmes metric $d_{\mu,1}$. Thus
\[
f(\pi_G(H_k(x))\times \R) \subset \pi_{G'}(H_{k+\ell}(y))\times \R.
\]
This concludes the proof.
\end{proof}

\section{Necklaces}
\label{sec:necklaces}

As an application of Theorem \ref{thm:F} we prove the existence of quasisymmetric parametrization for decomposition spaces associated with \emph{Antoine's necklaces} when the chains are long. For the statement, we introduce some terminology.

Let $I\ge 3$, a union $\bigcup_{i=1}^I T_i$ of pair-wise disjoint tori $T_1,\ldots, T_I$ in $\R^3$ is called a \emph{chain} if $T_i \cup T_j$ is a \emph{Hopf link} if $|i-j|=1$ or $\{i,j\} = \{1,I\}$, and an \emph{unlink} otherwise.

Suppose $T$ a torus in $\R^3$, and $\bigcup_{i=1}^I T_i$ is a torus chain contained in $\interior T$ in such a way that there is a homeomorphism $h\colon T \to \bB^2\times \bS^1$ satisfying $h(\partial T) =\partial \bB^2\times \bS^1$ and having the property that arguments of $p(h(T_i))$ are contained in $[\frac{2\pi i}{I} , \frac{2\pi (i+4/3)}{I} ]$ for each $i=1,\ldots,I$. Here $p\colon \bB^2\times \bS^1 \to \bS^1$ is the projection map $(x,s)\mapsto s$.
In this case, we say $\bigcup_{i=1}^I T_i$ is a \emph{necklace chain in $T$}.

Let $\phi_i\colon U\to U_i$ be PL-homeomorphisms from a neighborhood $U$ of $T$ onto  mutually disjoint neighborhoods $U_i$ of $T_i , 1\le i\le I,$ satisfying $T_i\subset U_i\subset T \subset U$.
The initial package $(T, T_1,\ldots, T_I; \phi_1,\ldots, \phi_I)$ yields a defining sequence $\cX=(X_k)$ and a
decomposition space, called an \emph{Antoine's $I$-necklace space}, $\R^3/G$; see Section \ref{sec:ft_ip}.  It is easy to see that the diameters of components of $X_k$ can be arranged to tend to zero. Thus the components of $X_\infty$ are singletons and $\R^3/G$ is homeomorphic to $\R^3$.

As discussed in Section \ref{sec:ft_ip}, the initial package induces a welding structure for the $I$-necklace space $\R^3/G$, therefore for each $\lambda>0$, a modular embedding of $\R^3/G$ and a Semmes metric $d_\lambda$ on $\R^3/G$. Semmes spaces $(\R^3/G,d_\lambda)$ associated with necklaces are linearly locally contractible because tori $T_i$'s are contractible in $T$, and these spaces are Ahlfors $3$-regular when $\lambda^3 I<1$.

The existence of quasisymmetric parametrization is proved in the following.

\begin{theorem}
\label{thm:necklace_theorem}
For every $I\ge 10$, there exists a Semmes metric $d$ on the decomposition space $\R^3/G$ associated to Antoine's $I$-necklace so that $(\R^3/G,d)$ is quasisymmetric to $\R^3$.
\end{theorem}

The proof of Theorem \ref{thm:necklace_theorem} relies on the possibility of fitting a necklace chain of length $I$ in a torus, using only tori all similar to the larger one; we find it easier to fit a rectangular chain in a rectangular torus than to fit a round chain in a round torus.

\subsubsection{{\bf{Rectangular necklaces}}}

Let $0<\lambda<b<a$. We define
\begin{eqnarray*}
R_+(a,b,\lambda) &=& [-\frac{\lambda}{2}, a +\frac{\lambda}{2}]\times
[-\frac{\lambda}{2}, b +\frac{\lambda}{2}], \\
R_-(a,b,\lambda) &=& (\frac{\lambda}{2}, a -\frac{\lambda}{2})\times
(\frac{\lambda}{2}, b -\frac{\lambda}{2}),
\end{eqnarray*}
and
\[
T(a,b,\lambda) = \left( R_+(a,b,\lambda)\setminus R_-(a,b,\lambda) \right) \times [-\frac{\lambda}{2}, \frac{\lambda}{2}].
\]

Let $L(a,b)=\partial([0,a]\times [0,b]) \times \{0\}$ be the 1-dim boundary of the rectangle $[0,a]\times [0, b]\times \{0\} $.  We say $T(a,b,\lambda)$ is a torus with \emph{length} $a+\lambda$, \emph{width} $b+\lambda$, \emph{thickness} $\lambda$ and \emph{core} $L(a,b)$.

Let $T=T(a,b,\lambda)$. We say that components of
\[
\partial T \cap \left( \R\times \{-\lambda/2, b+\lambda/2\} \times \R\right)\ \mathrm{and}\ \partial T \cap \left( \{-\lambda/2,a+\lambda/2\} \times \R^2 \right)
\]
are the \emph{long and short faces of $T$}, respectively. We call also the components of
\[
\partial T \cap \left( \R^2 \times \{-\lambda/2,\lambda/2\} \right)
\]
as the \emph{boundary annuli of $T$}.

We call the $3$-cells $[-\frac{\lambda}{2},a+\frac{\lambda}{2}]\times [-\frac{\lambda}{2},\frac{\lambda}{2}]\times [ -\frac{\lambda}{2},\frac{\lambda}{2}]$ and  $[-\frac{\lambda}{2},a+\frac{\lambda}{2}]\times [b-\frac{\lambda}{2}, b+ \frac{\lambda}{2}]\times [-\frac{\lambda}{2}, \frac{\lambda}{2}]$ as \emph{two long sides (front and back)} of $T$, and similarly $[-\frac{\lambda}{2},\frac{\lambda}{2}]\times [-\frac{\lambda}{2}, b+\frac{\lambda}{2}]\times [ -\frac{\lambda}{2}, \frac{\lambda}{2}]$ and $[a-\frac{\lambda}{2}, a+\frac{\lambda}{2}]\times [-\frac{\lambda}{2}, b+\frac{\lambda}{2}]\times [ -\frac{\lambda}{2}, \frac{\lambda}{2}]$ \emph{two short sides (left and right)} of $T$.

We say that a torus $T$ in $\R^3$ is a \emph{rectangular torus} if there exist a similarity map $g\colon \R^3 \to \R^3$ and $0<\lambda < b <a$ so that $T=g(T(a,b,\lambda))$. Furthermore, $T$ is \emph{$(p,q,r)$-oriented} if $g=h \circ O$, where $h$ is a similarity of the form $x\mapsto \mu x + v\, ,\mu>0,\,$ and $O$ is an orthogonal transformation taking the standard basis $(e_1,e_2,e_3)$ to $(e_p,e_q,e_r)$. We call the images of the long (resp. short) sides (resp. faces) of $T(a,b,\lambda)$ as the \emph{long} (resp. \emph{short}) sides (resp. faces) of $T$.

In what follows we use the following three types of \emph{tightly fitted torus pairs}. Let $T=T(A,B,1)$ and let $T'=g(T(a,b,\lambda))$ be an oriented torus contained in $T$. We say that $T'$ is \emph{tightly fitted into $T$} if one of the following conditions hold:
\begin{enumerate}
\item[(1)] $T'$ is a $(1,2,3)$-oriented torus contained in a long side of $T$, so that each long face of $T'$ intersects $\partial T$;

\item[(2)] $T'$ is a $(1,3,2)$-oriented torus contained in a long side of $T$, so that the long faces of $T'$ are contained in the boundary annuli of $T$;

\item[(3)] $T'$ is a $(2,1,3)$-oriented torus contained in a short side of $T$, so that each long face of $T'$ intersects $\partial T$ and the short faces of $T'$ are contained in the long faces of $T$. \
\end{enumerate}
If $T'$ is either $(1,2,3)$- or $(1,3,2)$-oriented torus, we have the following relations
\begin{equation}
\label{eq:relations_1}
a+\lambda\le A+1,\quad b+\lambda = 1,\quad \mathrm{and}\quad 2\lambda <1.
\end{equation}
If $T'$ is $(2,1,3)$-oriented,
\begin{equation}
\label{eq:relations_2}
a+\lambda = B+1,\quad b+\lambda= 1,\quad \mathrm{and}\quad 2\lambda <1.
\end{equation}

\begin{proposition}
\label{prop:rectangular}
Suppose $I\ge 10$. There exist $A>B>1$ and $a_i>b_i>\lambda_i,  (1\le i \le I)$ satisfying
\[
\frac{a_i}{A}= \frac{b_i}{B}= \frac{\lambda_i}{1},
\]
and tori $T$ and $T_i$, $1\le i \le I$, congruent to $T(A,B,1)$ and $T(a_i,b_i,\lambda_i)$, respectively, such that the union $\bigcup_{1\le i \le I}T_i$ is a necklace chain in $T$.
\end{proposition}

Theorem \ref{thm:necklace_theorem} now readily follows from this proposition and Theorem \ref{thm:F}.

\begin{proof}[Proof of Theorem \ref{thm:necklace_theorem}]
Let $I\ge 10$ and $T,T_1,\ldots, T_I$ be the tori constructed in Proposition \ref{prop:rectangular}. Let $\phi_i\colon \R^3\to  \R^3$ be similarity maps $x\mapsto \lambda_i x+ v_i$ so that $\phi_i(T)=T_i$ for $1\le i\le I$. Then the initial package $(T,T_1,\ldots, T_I,\phi_1,\ldots, \phi_I)$ gives rise to a natural self-similar welding structure in $\R^3$ as in Section \ref{sec:ft_ip}. The claim now follows from Theorem \ref{thm:F}.
\end{proof}

\begin{proof}
We construct for each $I\ge 10$, a torus $T=T(A,B,1)$ and a chain $\bigcup_{1\le i \le I}T_i$ which consists of tori all similar to $T$ and is tightly fitted in $T$.

Since the tori in the chain are pair-wise disjoint, there exist similarity maps $h_i \colon \R^3\to \R^3,\,\, x\mapsto \mu x+v_i$, with $\mu\in (0,1)$ and $v_i\in \R^3$, so that the new chain
$\bigcup_{i=1}^I h_i(T_i)$ is contained in the interior of $T$.  Hence tori $h_1(T_1),\ldots, h_I(T_I)$ satisfy the claims of the proposition.

It remains to construct tori $T,T_1,\ldots, T_I$ with aforementioned properties.

{\bf Case I.} Suppose $I=4k \ge 12$. Then $I=2K+2$ for some odd integer $K\ge 5$. We search for $A >B>2$ and $a>b>2\lambda $ satisfying
\begin{equation}
\label{ratio}
\frac{a}{A}= \frac{b}{B}= \frac{\lambda}{1},
\end{equation}
and mutually disjoint tori $T_i$, $1\le i \le 2K+2$, which are congruent to $T(a,b,\lambda)$ and contained in $T(A,B,1)$ so that $\bigcup T_i$ forms a necklace-chain positioned as follows.

Tori $T_1$ and $T_{K+2}$ are $(2,1,3)$-oriented torus tightly fitted in the two short sides (left and right) of $T(A,B,1)$ with cores lying on the plane $\{x_3=0\}$. Tori $T_i$ are $(1,3,2)$-oriented for even $i$, and
$T_i$ are $(1,2,3)$-oriented for odd indices $i\neq 1, K+2$.

Tori $T_2, T_3,\ldots , T_{K+1}$ are tightly fitted in the front side of $T(A,B,1)$, with the cores of $T_2,T_4,\ldots,T_{K+1}$ lying on the plane $\{x_2=0\}$ and the cores of $T_3, T_5,\ldots, T_K$ lying on $\{x_3=0\}$. Tori $T_{K+3},T_{K+4},\ldots,T_{2K+2}$ are tightly fitted in the back side of $T(A,B,1)$, with the cores of  $T_{K+3},T_{K+5},\ldots,T_{2K+1}$ lying on the plane $\{x_2=B\}$ and the cores of  $T_{K+4},T_{K+4},\ldots,T_{2K}$ lying on the plane $\{x_3=0\}$ .

Since the necklace-chain $\bigcup_{i=1}^I T_i$ is tightly fitted in $T(A,B,1)$,
\begin{equation}
\label{fit-bothsides}
a+\lambda =B+1\quad \mathrm{and}\quad b+\lambda = 1.
\end{equation}
Since tori $T_1$ and $T_2$, of thickness $\lambda$, are linked,
\begin{equation}
\label{fit-linking}
3\lambda <1.
\end{equation}
In order to fit the linked chain $T_1\cup T_2\cup \ldots \cup T_{K+2}$ in a long side of  $T(A,B,1)$, we seek for $ 0< \epsilon,\delta< 1/10$ so that
\begin{equation}
\label{fit-longsides}
A+1=K(a+\lambda)-(K-1)(2+\epsilon)\lambda+2(1+\delta)\lambda,
\end{equation}
\begin{equation}
\label{epsilon}
a+\lambda>2(2+\epsilon)\lambda,
\end{equation}
and
\begin{equation}\label{delta}
1>(3+\delta)\lambda.
\end{equation}
Note that $K(a+\lambda)-(K-1)(2+\epsilon)\lambda$ is the total length of the union $T_2\cup T_3\cup \cdots \cup T_{K+1}$, with $(K-1)(2+\epsilon)\lambda$ measuring the $K-1$ overlaps and $(1+\delta)\lambda$ measuring the distance from the chain to either short face of $T(A,B,1)$. Conditions \eqref{epsilon} and \eqref{delta}  are imposed to allow room for linking between  consecutive tori in the union $T_1\cup T_2\cup \cdots \cup T_{K+2}$.

We now check that \eqref{ratio} to \eqref{delta} can be realized with proper choices of $A,B,a,b,\lambda,\epsilon$, and $\delta$. By \eqref{ratio} and \eqref{fit-bothsides}, we have relations
\begin{equation}
\label{fit-more}
A+1=\lambda^{-2} \quad \text{and}\quad  B+1=a+\lambda =\lambda^{-1}.
\end{equation}
Furthermore, by \eqref{fit-longsides} and \eqref{fit-more},
\begin{equation}
\label{polynomial}
2(K-2)\lambda^3-(2\delta-(K-1)\epsilon)\lambda^3-K\lambda+1=0.
\end{equation}
Let $0<\epsilon<1/(5K)$ to be fixed later and fix $\delta=(K-1)\epsilon/2$.
Then \eqref{polynomial} reads as
\begin{equation}
\label{polynomialr}
2(K-2)\lambda^3-K\lambda+1=0.
\end{equation}
It is now easy to check that \eqref{polynomialr} admits a solution $\lambda\in (0,3/10)$. We now choose $\epsilon$ small enough so that \eqref{epsilon} and \eqref{delta} hold. The parameters $A,B,a$, and $b$ are now uniquely determined by \eqref{fit-more} and \eqref{ratio}.

{\bf Case II.} Suppose that $I=4k+2\ge10$. Then $I=2K+2$ for an even $K\ge 4$. Again we will fit a necklace-chain $\bigcup_{i=1}^{2K+2} T_i$, consisting of tori all similar to $T(A,B,1)$, in the torus $T(A,B,1)$.

Since $K$ is even, the linking condition forces $T_1$ and $T_{K+2}$ to have different $(p,q,r)$-orientations and unequal sizes. Let $T_1$  be a $(2,1,3)$-oriented torus tightly fitted in the left side of $T(A,B,1)$ with the core lying on the plane $\{x_3=0\}$, and let $T_{K+2}$ be a smaller $(2,3,1)$-oriented torus (not tightly fitted)  in the right side of $T(A,B,1)$ with its core lying on the $2$-plane $\{x_1=A\}$.

As in Case I, we choose $T_i$ to be a $(1,3,2)$-oriented when $i\neq K+2$ is even and $T_i$ to be a $(1,2,3)$-oriented when $i\neq 1$ is odd. Tori $T_2, T_3,\ldots, T_{K+1}$ shall be tightly fitted in the front side of $T(A,B,1)$ with the cores of $T_2,T_4,\ldots,T_K$ lying on the plane $\{x_2=0\}$ and the cores of $T_3, T_5,\ldots, T_{K+1}$ lying on $\{x_3=0\}$. Tori $T_{K+3},T_{K+4},\ldots,T_{2K+2}$ shall be tightly fitted in the back side of $T(A,B,1)$ with the cores of  $T_{K+3},T_{K+5},\ldots,T_{2K+1}$ lying on the plane $\{x_2=B\}$ and cores of  $T_{K+4},T_{K+6},\ldots,T_{2K+2}$ lying on the plane $\{x_3=0\}$. Furthermore, one short face of $T_{K+1}$ and one short face of $T_{K+3}$ are placed in a common short face of $T(A,B,1)$.

Tori $T_i, 1\le i\le 2K+2$ and $ i\neq K+2,$ are congruent to $T(a,b,\lambda)$ and torus $T_{K+2}$ is congruent to a smaller $T(a',b',\lambda')$; all are similar to $T(A,B,1)$.

It is straightforward to check that numbers $A >B>1$, $a>b>\lambda>0 $ and $a'>b'>\lambda'>0 $  can be found so that $\bigcup_{i=1}^I T_i$ is a chain tightly fitted in $T$. We omit the details.

{\bf Case III.} Suppose that $I\ge 11$ is odd. Then $I=2K+3$ for some $K\ge 4$.
For $K$ even, there exist, by Case I, numbers $A,B,a,b,\lambda$ and tightly fitted tori $T_1,\ldots, T_{2K+2}$ in $T=T(A,B,1)$ so that tori $T_1,\ldots, T_{2K+2}$ are congruent to $T(a,b,\lambda)$. For $K$ odd, we have, in addition, parameters $a',b'$, and $\lambda'$ so that tori $T_1,\ldots, T_{2K+2}$ are congruent to either $T(a,b,\lambda)$ or $T(a',b',\lambda')$. Let $\epsilon>0$ and $\delta>0$ be the parameters appearing in these constructions. We rename the first torus $T_1$ to $T_0$.

The plan is to replace tori $T_2,T_3,T_4$ congruent to $T(a,b,\lambda)$ by four tori $t_1, t_2, t_3, t_4$ congruent to a smaller torus $T(a'',b'',\lambda'')$ which is similar to $T(a,b,\lambda)$. The new collection  $T_0, t_1, t_2, t_3, t_4, T_5,\ldots, T_{2K+2}$ forms the necklace chain for the case $I=2K+3$.

Denote by $F_{\theta}$ the rotation in $\R^3$ about the $x_1$-axis by an angle $\theta$, so that $F_\theta(\R^2\times \{0\}) = P_\theta$, where $P_{\theta}$ is the plane $\{x_3=x_2 \tan \theta \}$ in $\R^3$. Recall that $T_0$ is a $(2,1,3)$-torus and $T_5$ is a $(1,2,3)$-torus with  cores lying on the plane $P_0$.

For $j=1,\ldots, 4$, let $t_j$ be a translate of $F_{2j\pi/5}(T(a'',b'',\lambda''))$ in the direction of $x_1$, where the translation will be fixed later. Then the core of $t_j$ lies on the plane $P_{2j\pi/5}$; and the planes containing the cores of two consecutive tori in  $\{T_0, t_1, t_2, t_3, t_4, T_5\}$ form an angle $2\pi/5$.

Numbers $A >B>1$, $a>b>\lambda>0 $ and $a'>b'>\lambda'>0 $  are retained from the previous cases. To realize the plan, we need to choose $a''>b''>\lambda''>0$ satisfying
$\frac{a''}{A}= \frac{b''}{B}= \frac{\lambda''}{1}<\lambda$, so that
the $4$ new tori $t_1, t_2, t_3, t_4$ can be fitted lengthwise into the space vacated by $T_2,T_3,T_4$ to form a necklace chain. The calculations leading to these choices are routine however tedious, we omit the details. This completes the proof.
\end{proof}

\section{The Bing double and the Whitehead continuum revisited}
\label{sec:Bing_Double}

The construction of the space $\R^3/\Bd$ associated to the Bing double is illustrated and discussed in Daverman's book \cite[Example 1, pp. 62-63]{DavermanR:Decm} and in an article of Freedman and Skora \cite{FreedmanM:Strags}. See the original article \cite{BingR:Homb3s} or \cite{BingR:Shrwl} for a highly nontrivial shrinking procedure that leads to a homeomorphism $\R^3/\Bd\approx \R^3$.

We fix an initial package consisting of three tori $T, T_1,T_2$ in $\R^3$ so that $T_1$ and $T_2$ are linked in $T$ but not in $\R^3$ as in \cite[Figure 9-1]{DavermanR:Decm}, and the homeomorphisms $\phi_i \colon T\to T_i$.
Denote by $\cX=(X_k)$ the defining sequence induced by the initial package as described in Section \ref{sec:ft_ip}, by $\Bd$ the (cellular) decomposition, and by $\R^3/\Bd$  the decomposition space.

Semmes showed \cite[Theorem 1.12c]{SemmesS:Goomsw} that $\R^3/\Bd$ admits a metric $d$ so that the space $(\R^3/\Bd,d)$ is quasiconvex, Ahlfors $3$-regular and linearly locally contractible and it supports certain Sobolev and Poincar\'e inequalities that are crucial for analysis, but this space is not quasisymmetric to $\R^3$. Semmes' construction of the metric $d$ has served as a model for modular metrics defined in Section \ref{sec:Semmes}; it is easy to verify that $d$ is bilipschitz equivalent to a modular metric.

The non-existence of a quasisymmetric homeomorphism $(\R^3/\Bd,d) \to \R^3$ is based on a lemma of Freedman and Skora on essential intersections (Lemma 2.4 in \cite{FreedmanM:Strags}).

We state their lemma in the following.

\begin{lemma}\label{lemma:FS}
Let $T_1$ and $T_2$ be two solid tori embedded in $\bB^2\times \bS^1$ as in the Bing double construction. Let $(P,\partial P)\subset (\bB^2\times \bS^1, \partial \bB^2\times \bS^1 )$ be an embedded connected planar surface representing the generator of the relative homology group $H_2(\bB^2\times \bS^1 ,\partial, \bZ)$. Suppose $P$ and $T_1\cup T_2$ meet in transverse general position. Then for $i=1$ or $2$, $P\cap T_i$ must contain at least two surfaces which represent generators of $H_2(T_i ,\partial, \bZ)$.
\end{lemma}

Using the notion of circulation Lemma \ref{lemma:FS} can be interpreted as follows.

\begin{lemma}
\label{lemma:FS-2}
Let $\R^3/\Bd$ the decomposition space associated to the Bing double $Bd$ and $\cX=(X_k)$ the defining sequence associated to the initial package $(T,T_1,T_2,\phi_1,\phi_2)$. Then
\[
\wind(X_k,T;\alpha) \ge 2^k
\]
for every $k\ge 0$ and every meridian $\alpha$ on $T$.
\end{lemma}

The Freedman--Skora lemma yields that the defining sequence $\cX$ of the Bing double has order of circulation at least $2$; in fact the order of growth of $\cX$ is exactly $2$. Theorem \ref{thm:BD_first} now follows from Theorem \ref{thm:general}.

By the Freedman--Skora lemma, the concurrent pair $(\gamma(x),\omega(x))$ defined in Section \ref{sec:local_theory} can be taken to be $(2,2)$ for every $x$ in $\pi_\Bd(\Bd_\infty)$. From Theorem \ref{thm:General_local}, it follows that
\begin{enumerate}
\item any open subset of $(\R^3/\Bd, d_\lambda)$ which intersects  $\pi_\Bd(\Bd_\infty)$ is not quasisymmetrically embeddable in $\R^3$ for any $0< \lambda <1$;
\item every point in $\pi_\Bd(\Bd_\infty)$ is $\lambda$-singular (of index $1$) for $1/2 <\lambda<1$.
\end{enumerate}

\smallskip

We finish the proof of Theorem \ref{thm:Bing_stab} using the second fact.
\begin{proof}[Proof of Theorem \ref{thm:Bing_stab}]
Suppose there is a quasisymmetric homeomorphism $f\colon (\R^3/\Bd\times \R,d_{\lambda,1}) \to (\R^3/\Bd\times \R,d_{\lambda',1})$ for some $\lambda'\in (1/2,1)$ and $\lambda \in (0,\lambda')$.
 Let $\cX=(\Bd_k)$ be the standard defining sequence for the Bing double, and denote $\Bd_\infty = \bigcap_k \Bd_k$. By Theorems \ref{thm:General_local} and \ref{theorem:cantor-bilip},
\begin{equation}\label{eq:singular}
\pi_\Bd(\Bd_\infty)  = \sing_{\lambda,1}(\R^3/\Bd) = \sing_{\lambda',1}(\R^3/\Bd).
\end{equation}

Since $\pi_\Bd(\Bd_\infty)$ is a Cantor set, every point is an accumulation point. By Theorem \ref{thm:sb_point}, given $\ell>0$ there exists $k_0>0$ so that
\[
f(\pi_\Bd(\Bd_k)\times \R) \subset \pi_\Bd(\Bd_{k+\ell})\times \R
\]
for all $k\geq k_0$. Since $\pi_\Bd(\Bd_k)$ has $2^k$ components and $\pi_\Bd(\Bd_{k+\ell})$ has $2^{k+\ell}$ components, we conclude that there exists a component $T$ of $\Bd_{k+\ell}$ so that
\[
\pi_\Bd(T)\times \R \cap f(\pi_\Bd(\Bd_\infty)\times \R) = \emptyset.
\]
This contradicts  \eqref{eq:singular}.
\end{proof}

We refer to \cite{HeinonenJ:Quansa} for the non-existence of the quasisymmetric parametrization of $\R^3/\Wh\times \R^m$, where $\Wh$ is the Whitehead continuum $\Wh$. We merely note that the homological argument of the Freedman--Skora lemma was used in \cite{HeinonenJ:Quansa} to obtain a version of the intersection lemma and to show that the standard defining sequence for the Whitehead continuum has an order of circulation at least $2$. Since order of growth of the Whitehead continuum is $1$, we can use Theorem \ref{thm:general} to recover the nonexistence of quasisymmetric parametrization of $(\R^3/\Wh\times \R^m,d_{\lambda,m})$. Indeed, as shown in \cite{HeinonenJ:Quansa}, \emph{$(\R^3/\Wh\times \R^m, d_{\lambda,m})$ is not quasisymmetric to $\R^{3+m}$ for $\lambda >2^{-(3+m)/(2m)}$}.

\section{Bing's Dogbone}
\label{sec:Bing_Dogbone}

The decomposition space $\R^3/\Db$ associated with Bing's dogbone  \cite{BingR:DecE3p} was the first known example of a decomposition space which is not homeomorphic to $\R^3$ but whose product with a line, $(\R^3/\Db)\times \R$, is homeomorphic to $\R^4$; see \cite{BingR:Carpcn}.

Bing's dogbone space $\R^3/\Db$ is constructed as follows. Let $A$ be a PL cube-with-$2$-handles standardly embedded in $\R^3$, and let $A_1,A_2, A_3, A_4$ be four cubes-with-handles of genus $2$ embedded in the interior of $A$ as illustrated in \cite[Fig. 1, p.486]{BingR:DecE3p}.

Let $\phi_j\colon U\to U_j$ be PL-homeomorphisms from a neighborhood $U$ of $A$ onto  mutually disjoint neighborhoods $U_i, 1\le i\le 4,$ of $A_i$ satisfying $A_i\subset U_i\subset A \subset U$. The intersection
\[
\Db= \bigcap_{k=0}^{\infty} \bigcup_{\alpha\in S_k}\phi_{\alpha}(A)
\]
is called \emph{Bing's dogbone}, where $\phi_\alpha = \phi_{\alpha_1}\circ \cdots \circ \phi_{\alpha_k}$ and $\alpha=(\alpha_1,\ldots, \alpha_k)\in  \{1,2,3,4\}^k$. The decomposition $\R^3/\Db$ is topologically different from $\R^3$ even though each nondegenerate component of $\Db$ is a tame arc \cite{BingR:DecE3p}. On the other hand, $(\R^3/\Db)\times \R$ is $\R^4$.

The initial package $(A,A_1,\ldots, A_4,\phi_1,\ldots, \phi_4)$ yields a defining sequence $\cX_{\Db}=(X_k)$: $X_0=A$ and
\[
X_{k+1} = \bigcup_{\alpha=1}^{4} \phi_\alpha(X_k)
\]
for $k\ge 0$. So $X_k = \bigcup_{\alpha \in \cS_k} \phi_\alpha(A)$. The initial package induces a welding structure $(\fC_{\Db},\fA_{\Db},\fW_{\Db})$ on the defining sequence $\cX_{\Db}$; in particular $\cC$ consists of a single condenser $(A, \cup_{i=1}^4 A_i)$. See Section \ref{sec:ft_ip} for details.

\begin{theorem}
\label{thm:dogbone}
Let $(\R^3/\Db,d_\lambda)$ be a Semmes space associated to the defining sequence $\cX_{Db}$ and the welding structure $(\fC_{\Db},\fA_{\Db},\fW_{\Db})$.
Suppose $m\ge 1$ and $2^{-\frac{1+m}{m}}<\lambda< 2^{-2/3}$.
Then $(\R^3/\Db\times \R^m, d_{\lambda,m})$ is Ahlfors $(3+m)$-regular and linearly locally contractible, but it is not quasisymmetrically equivalent to $\R^{3+m}$.
\end{theorem}

The Ahlfors regularity follows from Proposition \ref{prop:Ahlfors_regularity}, since $\cX$ has order of growth $4$.
The linear local contractibility follow from Proposition \ref{prop:LLC}, since every $A_i$ is contractible in $A$.

To show that $(\R^3/\Db)\times \R^m$ is not quasisymmetric to $\R^{3+m}$, we estimate the order of circulation of $\cX$ in $A$ from below.

As in \cite[Fig. 1]{BingR:DecE3p}, let $C_1$ and $C_2$ be two disjoint $3$-cells in $A$ so that handles of $\cup A_i$ are sorted into two groups, and each group consists of four pair-wise linked handles, one from each $A_i$, and is contained in one of the $3$-cells $C_1$ or $C_2$. Then $C_1 \cup C_2 \cup A_1\cup A_4$ and $C_1\cup C_2 \cup A_2\cup A_3$ is a pair of solid tori in $A$.

The arrangement of cubes-with-handles $\cup A_i$ is understood as follows. We fix essential $2$-disks $D_1,D_2,D_3$ in $A$ as in \cite[Fig. 1]{BingR:DecE3p}. These disks have the property that if $h\colon \R^3\to \R^3$ is a homeomorphism that is identity outside $A$ then
\begin{enumerate}
\item $h(A_1)\cup h(A_4)$ and $h(A_2)\cup h(A_3)$ intersect both $D_1$ and $D_2$,
\item $h(A_1)\cup h(A_3)$  and $h(A_2)\cup h(A_4)$ intersect both $D_1$ and $D_3$, and
\item  $h(A_1)\cup h(A_2)$ and $h(A_3)\cup h(A_4)$ intersect both $D_2$ and $D_3$.
\end{enumerate}

We use topological properties of the initial package to show the following estimate of Freedman--Skora type. This estimate implies that the order of circulation of $\cX$ is at least $4$. This together with Theorem \ref{thm:general} proves Theorem \ref{thm:dogbone}.

\begin{lemma}
Let $\gamma$ be a meridian of $A$ that is isotopic to $\partial D_1$ on $\partial A$. Then
\begin{equation}\label{eq:circulation-dogbone}
\wind( X_k,\gamma,A)\ge 4^{k-1}
\end{equation}
for every $k\ge 1$.
\end{lemma}

\begin{proof}
As a preliminary step, we define tori $T^O$ and $T^X$ as follows
\[
T^O= C_1\cup C_2 \cup A_1\cup A_4 \quad\text{and}\quad
T^X=C_1\cup C_2 \cup A_2\cup A_3.
\]
Then
\[
A_{\alpha 1}\cup A_{\alpha 4} \subset (T^O)_{\alpha} \subset A_{\alpha}\quad\text{and}\quad A_{\alpha 2}\cup A_{\alpha 3} \subset (T^X)_{\alpha} \subset A_{\alpha},
\]
where $A_{\alpha}=\phi_{\alpha} A$, $(T^O)_{\alpha}=\phi_{\alpha} (T^O)$ $(T^X)_{\alpha}=\phi_{\alpha} (T^X)$, and $\alpha\in \{1,2,3,4\}^k $.

Note that tori $(T^O)_1 \cup (T^O)_4 $ are linked in $T^O$ the way that the two first stage tori are linked in the 0-th stage torus as in the construction of the Bing double. Note also that the same can be said about the linking of $(T^X)_1 \cup (T^X)_4 $ in $T^O$, $(T^O)_2 \cup (T^O)_3 $ in $T^X$ and $(T^X)_2 \cup (T^X)_3 $ in $T^X$.

Therefore for every $\alpha\in \{1,2,3,4\}^k$, tori $(T^O)_{\alpha 1} \cup (T^O)_{\alpha 4}$ are linked in $(T^O)_{\alpha }$ the way the  first stage tori are linked in the 0-th stage torus as in the Bing double, and the same can be said about the linking of $(T^X)_{\alpha 1}\cup (T^X)_{\alpha 4} $ in $(T^O)_{\alpha}$, $(T^O)_{\alpha 2} \cup (T^O)_{\alpha 3} $ in $(T^X)_{\alpha }$ and $(T^X)_{\alpha 2} \cup (T^X)_{\alpha 3} $ in $(T^X)_{\alpha }$.

This linking property has the following consequences.

(I). If $f\colon (\bB^2,\partial \bB^2) \to (A,\partial D_1)$ is map with the property  $f(\partial \bB^2) =\partial D_1$, then $f(\bB^2)$ intersects both $T^O$ and $T^X$ virtually interior essentially. Indeed, let $Q$ be a $3$-cell in $\R^3$ so that $Q\cap A \subset \partial A$, $Q\cap \partial D_1 = \emptyset$, and that $Q\cup A$ is a torus. We denote $T=Q\cup A$. Since a core of $T^O$ is also a core of $T$, we have that $f(\bB^2)$ intersects $T^O$ virtually interior essentially. The same argument applies also to $T^X$.

(II). Suppose $\Omega$ is a $2$-manifold in $\bB^2$ and $f\colon (\Omega,\partial \Omega) \to (T^O,\partial T^O)$ is a virtually interior essential map.  Then by the standard argument of filling $T^O$ with $2$-disks, we have  that $f$ has an virtually interior essential intersection with $A_1\cup A_4$; see e.g. the proof of wildness of Antoine's necklace \cite[Prop. 5, pp. 73-74]{DavermanR:Decm}. The same can be said about $T^X$ and $A_2\cup A_3$.

The estimate of circulation \eqref{eq:circulation-dogbone}  follows from the claim below and the relation between the number of essential intersections and the circulation as stated in  Remark \ref{rmk:vie}.

{\emph{Claim.}}
Let $f\colon (\bB^2, \partial \bB^2) \to (A,\partial A)$ be an interior essential map so that $f(\partial \bB^2)$ is isotopic to $\partial D_1$ on $\partial A$. Then $f(\bB^2)\cap X_k$ has at least $ 4^k$ virtually interior essential components. It remains to verify the claim.

Let $\varsigma\colon \{1,2,3,4\}\to \{O, X\}$ be the map defined as $\varsigma(1)= \varsigma(4)= O$ and $\varsigma(2)= \varsigma(3)=X$, and $\varsigma^k=\varsigma\times\cdots \times \varsigma\colon  \{1,2,3,4\}^k \to \{O, X\}^k$ be the product map. Set $\cS_k=\{1,2,3,4\}^k$,  $\Sigma_k=\{O , X\}^k$, and $s_k(w)= (\varsigma^k)^{-1}(w)$. Note that for  $w=(w_1,\ldots,w_k) \in \Sigma_k$
\[
s_k(w)=(\varsigma^k)^{-1}(w)= \{(\alpha_1,\ldots, \alpha_k) \in \cS_k,\, \alpha_j\in \varsigma^{-1}(w_j) \,\,\text{for all} \,\, 1\le j\le k\}.
\]
So $\cS_k= \cup_{w\in \Sigma_k} s_k(w) $ is a disjoint union.

For each $k\ge 1$, we sort the $4^k$ cubes-with-handles
in $X_k$  into $2^k$ mutually disjoint groups as follows.
If $k=1$, the two groups are  $X_1(O)=\{A_1, A_4\}$ and $X_1(X)=\{A_2, A_3\}$.
Suppose $k\ge 2$, define  for $w\in \Sigma_k$
\[X_k(w)=\{A_{\alpha}:\alpha\in s_k(w) \}.\]
So $X_k= \cup_{w\in \Sigma_k} X_k(w) $ is a disjoint union of $2^k$ groups.

Fix a $w \in \Sigma_k$, we will  focus on the $2^k$ cubes-with-handles in $X_k(w)$ and consider a finite defining sequence associated with this particular $w=(w_1,w_2,...,w_k)$ as follows. Set
\[
Z_0=A, \quad Z_1=T^{w_1}, \quad \text{and} \quad Z_{j}=\cup_{\alpha\in s_{j-1}(w_1,w_2,\ldots,w_{j-1})} (T^{w_j})_\alpha
\]
for $2\le j\le k$. Note that
\[
Z_{j+1} \cap (T^{w_j})_\alpha = (T^{w_{j+1}})_{\alpha i_1} \cup (T^{w_{j+1}})_{\alpha i_2},
\]
where $\{i_1,i_2\} = \varsigma^{-1}(w_j)$, for every $(T^{w_j})_\alpha$ in $Z_j$.

Let $f\colon (\bB^2, \partial \bB^2) \to (A,\partial A)$ be an interior essential map so that $f(\partial \bB^2)$ is isotopic to $\partial D_1$ on $\partial A$.
By applying a homotopy near $\partial A$, we may assume that $f(\bB^2)\cap \partial A = \partial D_1$. Then by (I), $f(\bB^2)$ intersects both $T^O$ and $T^X$ virtually interior essentially. In particular, $f(\bB^2)\cap Z_1$ has at least one virtually interior essential component.

In view of  the linking relation (of the Bing double type) between tori in consecutive generations, we may apply the lemma of  Freedman and Skora (Lemma \ref{lemma:FS}) iteratively to conclude that $f(\bB^2)\cap Z_k$ has at least $2^{k-1}$ virtually interior essential components.

Tori in $Z_k$ are pair-wise disjoint and each torus contains two cubes-with-handles in $X_k(w)$. It follows from (II) above that $f(\bB^2)\cap X_k(w)$ has at least $2^{k-1}$ virtually interior essential components.

The claim follows by summing over all $w\in \Sigma_k$. This completes the proof of the theorem.
\end{proof}


\def\cprime{$'$}

\vskip 5pt

\noindent P.P.\quad Department of Mathematics and Statistics (P.O. Box 68), FI-00014 University of Helsinki, Finland\\
e-mail: pekka.pankka@helsinki.fi

\vskip 3pt

\noindent J.-M.W.\quad Department of Mathematics, University of Illinois,  1409 West Green Street, Urbana, IL 61822, USA\\
e-mail: wu@math.uiuc.edu

\end{document}